\documentclass[12pt]{amsart}

\usepackage[utf8]{inputenc}
\usepackage[OT2,T1]{fontenc}
\usepackage{ae}
\usepackage{aecompl}
\usepackage{stmaryrd} 

\usepackage{amsmath,amsthm,amssymb,mathrsfs,calligra,dsfont,mathtools}
\usepackage[final, babel=true]{microtype}
\usepackage{caption}
\usepackage{subcaption}
\usepackage{eucal}
\usepackage{verbatim}
\usepackage{centernot}
\usepackage{enumitem}
\usepackage{graphicx}
\usepackage{tikz}
\usetikzlibrary{matrix,arrows,cd}
\usepackage[all]{xy}
\usepackage{hyperref}
\usepackage{xcolor}
\usepackage{url}
\definecolor{darkgreen}{rgb}{0,0.45,0}
\usepackage{booktabs}

\hypersetup{
	colorlinks,
	linktoc=all,
	linkcolor={darkgreen},
	citecolor={blue!50!black},
	urlcolor={blue!80!black}
}

\DeclareFontFamily{U}{mathx}{\hyphenchar\font45}
\DeclareFontShape{U}{mathx}{m}{n}{
	<5> <6> <7> <8> <9> <10>
	<10.95> <12> <14.4> <17.28> <20.74> <24.88>
	mathx10
}{}
\DeclareSymbolFont{mathx}{U}{mathx}{m}{n}
\DeclareFontSubstitution{U}{mathx}{m}{n}
\DeclareMathAccent{\widecheck}{0}{mathx}{"71}
\DeclareMathAccent{\wideparen}{0}{mathx}{"75}

\newcommand\I{\mathbf{I}}

\newcommand\bC{\mathbb{C}}
\newcommand\N{\mathbb{N}}
\newcommand\Z{\mathbb{Z}}
\newcommand\Q{\mathbb{Q}}
\newcommand\F{\mathbb{F}}

\newcommand\T{\mathbf{T}}
\newcommand\bS{\mathbf{S}}
\newcommand\M{\mathbf{M}}
\newcommand\K{\mathbf{K}}

\newcommand\Zp{{\mathbb{Z}_p}}
\newcommand\Qp{{\mathbb{Q}_p}}

\newcommand\cC{\mathcal{C}}
\newcommand\cD{\mathcal{D}}
\newcommand\cE{\mathcal{E}}
\newcommand\cF{\mathcal{F}}
\newcommand\cG{\mathcal{G}}

\newcommand\cK{\mathcal{K}}
\newcommand\cL{\mathcal{L}}

\newcommand\cO{\mathcal{O}}

\newcommand\cW{\mathcal{W}}
\newcommand\cX{\mathcal{X}}

\newcommand\cZ{\mathcal{Z}}

\newcommand\ga{\mathfrak{a}}

\newcommand\gm{\mathfrak{m}}

\newcommand\gp{\mathfrak{p}}

\newcommand{\ob}[1]{\mkern 1.5mu\overline{\mkern-1.5mu#1\mkern-1.5mu}\mkern 1.5mu}

\newcommand{\Ind}{\mathrm{Ind}}

\newcommand{\res}{\mathrm{res}}

\newcommand{\ord}{\mathrm{ord}}

\renewcommand{\char}{\textrm{char}}

\DeclareMathOperator{\Gal}{Gal}

\DeclareMathOperator{\Ann}{Ann}

\DeclareMathOperator{\Hom}{Hom}

\DeclareMathOperator{\coker}{coker}
\DeclareMathOperator{\Frac}{Frac}

\DeclareMathOperator{\ad}{ad}
\DeclareMathOperator{\tr}{tr}
\DeclareMathOperator{\End}{End}

\DeclareMathOperator{\Symm}{Symm}

\DeclareMathOperator{\Cl}{Cl}
\DeclareMathOperator{\rk}{rk}
\DeclareMathOperator{\spe}{sp}
\DeclareMathOperator{\Frob}{Frob}

\DeclareMathOperator{\Merel}{Merel}

\DeclareMathOperator{\cyc}{cyc}
\DeclareMathOperator{\GL}{GL}

\DeclareMathOperator{\length}{length}
\DeclareMathOperator{\Spec}{Spec}

\DeclareMathOperator{\spl}{spl}
\DeclareMathOperator{\lcm}{lcm}
\DeclareMathOperator{\den}{den}
\DeclareMathOperator{\CM}{CM}
\DeclareMathOperator{\EmbDim}{EmbDim}

\newcommand{\rM}{\mathrm{M}}
\newcommand{\rD}{\mathrm{D}}

\newcommand{\HH}{\mathrm{H}}
\newcommand{\ff}{\mathrm{f}}

\DeclareSymbolFont{cyrletters}{OT2}{wncyr}{m}{n}
\DeclareMathSymbol{\Sha}{\mathalpha}{cyrletters}{"58}

\newcommand{\lsem}{\llbracket}
\newcommand{\rsem}{\rrbracket}

\makeatletter
\@namedef{subjclassname@2020}{\textup{2020} Mathematics Subject Classification}
\makeatother

\theoremstyle{defstyle}
\newtheorem{definition}[subsubsection]{Definition}

\newtheorem{remark}[subsubsection]{Remark}

\newtheorem*{assumption*}{Assumption}

\theoremstyle{thmstyle}
\newtheorem{theorem}[subsubsection]{Theorem}
\newtheorem{proposition}[subsubsection]{Proposition}
\newtheorem{lemma}[subsubsection]{Lemma}
\newtheorem{corollary}[subsubsection]{Corollary}
\newtheorem{conjecture}[subsubsection]{Conjecture}
\newtheorem*{conjecture*}{Conjecture}

\tikzset{
	symbol/.style={
		draw=none,
		every to/.append style={
			edge node={node [sloped, allow upside down, auto=false]{$#1$}}}
	}
}

\author{Alexandre Maksoud}
\address{University of Vienna, Faculty of mathematics, Oskar-Morgenstern-Platz 1, 1090 Vienna, Austria}
\email{maksoud.alexandre@gmail.com}
\urladdr{https://sites.google.com/view/alexandre-maksoud/}
\subjclass[2020]{11R23, 11F33, 11F80.}
\keywords{Hecke algebras, Eisenstein series, modularity theorems, Sharifi's conjecture, $p$-adic $L$-functions, Exceptional zero formulae}

\title{On Hecke algebras at exceptional Eisenstein primes}
\begin{document}
\maketitle
\begin{abstract}
	We study the structure over $\Zp$ of the Hecke algebra acting on ordinary cusp forms in the Eisenstein, non-$p$-distinguished case. Our results and methods draw on four complementary approaches: deformation theory, through $R=\T$ theorems; Iwasawa theory, through connections to Sharifi's conjecture; $p$-adic $L$-functions, through an exceptional zero formula; and computational experiments using Magma. 
	We also obtain new results regarding the Coleman--Greenberg conjecture.
\end{abstract}

\bigskip 
\bigskip
\section{Introduction}
\subsection{Structure of Hecke algebras in Mazur's setting}\label{sec:mazur}
Let $p,N$ be two prime numbers with $p\geq5$ and $p\mid(N-1)$, and consider the weight $2$ and level $\Gamma_0(N)$ Eisenstein series
\begin{equation}\label{eq:Mazur_eisenstein_series}
	E_2^{(N)}(q) = \dfrac{N-1}{24} + \sum_{n\geq 1}\sum_{\substack{d|n \\ N\nmid d}}d\,q^n.
\end{equation}
 In his influential work \cite{mazur1977}, Mazur studied the structure of the localization $\T_2^{(N)}$ of the Hecke algebra acting on $S_2(\Gamma_0(N),\Zp)$ at the maximal ideal corresponding to $E_2^{(N)}(q)$. 
 The existence of this maximal ideal relies on the assumption $p\mid(N-1)$, which  ensures that $E_2^{(N)}$ behaves like a cusp form modulo $p$, giving rise to congruences between $E_2^{(N)}$ and weight $2$ cusp forms modulo $p$. 

For comparison with the weight $1$ setting, we focus on two fundamental properties of $\T_2^{(N)}$, established by Mazur \cite[Corollaries 16.2 and 16.3]{mazur1977} and Merel \cite[Théorème 2]{merel}, respectively. 

\bigskip
\bigskip
\begin{theorem}[Mazur, Merel]\label{thm:intro_mazur_merel}\ 
	\begin{enumerate}
		\item[(Mazur)] $\T_2^{(N)}$ is Gorenstein. More precisely, it admits a presentation of the form
		\[\T_2^{(N)} \simeq \Zp\lsem x\rsem/(c_1x+c_2x^2+\ldots+x^r).\]
		\item[(Merel)] The $\Zp$-rank of $\T_2^{(N)}$ is $1$ if and only if 
		\[ \lambda^{(N)}\left(u_{\Merel}\right) \not\equiv 0 \pmod p, \qquad u_{\Merel}:=\prod_{k=1}^{\frac{N-1}{2}} k^k, \]
		where $\lambda^{(N)} \colon (\Z/N\Z)^\times \twoheadrightarrow \Z/p\Z$ is any surjective group homomorphism, equivalently a discrete logarithm modulo $p$.
	\end{enumerate}
\end{theorem}

Many other authors have greatly contributed to the understanding of $\T_2^{(N)}$, using various tools such as modularity theorems as in the work of Calegari--Emerton \cite{calegariemerton}, higher Eisenstein elements as in Lecouturier's work \cite{lecouturier}, or Galois cohomology as in works by Wake and Wang-Erickson \cite{WWE}. 

Hecke algebras at Eisenstein primes have also been studied in broader settings, allowing weights $k>2$ or non-prime levels $N$ \cite{harderpink,kuriharaJNT,skinnerwilesPNAS,ohta2005,wakeJEMS,WWEadvances,deoJussieu}. In weight $1$, by contrast, the Eisenstein localization has received comparatively little attention. Although Calegari and Geraghty, and subsequently Calegari, established integral modularity theorems for weight $1$ Katz modular forms \cite{calegarigeraghty,calegariCrelle}, their results concern non-Eisenstein maximal ideals. We also mention the work of Berger and Klosin \cite{bergerklosin}, which studies weight $1$ Hecke algebras in the non-exceptional Eisenstein setting, complementary to the exceptional case considered in this paper. 

\subsection{The exceptional weight $1$ setting}
Throughout the paper, $p$ denotes an odd prime and $N\geq3$ an integer.  Let $\chi \colon (\Z/N\Z)^\times \to \bC^\times$ be a primitive Dirichlet character. We assume the following:
\begin{assumption*}
	$p$ does not divide $N\phi(N)$ and $\chi$ is odd. 
\end{assumption*}
In particular, $\chi$ has prime-to-$p$ order. Fix an embedding $\iota_p \colon \ob{\Q} \hookrightarrow \ob{\Q}_p$, and an isomorphism $\ob{\Q}_p\simeq \bC$. We regard $\chi$ as a $p$-adic character, and we let $\cO$ denote the finite unramified extension of $\Zp$ generated by the values of $\chi$. Consider the $p$-stabilized weight $1$ and level $\Gamma_1(Np)$ Eisenstein series
\begin{equation}
	 f(q)=E_1^{(p)}(1,\chi)=\frac{1}{2}(1-\chi(p))L(\chi,0)+ \sum_{n\geq 1} \sum_{\substack{d|n \\ p\nmid d}}\chi(d)\, q^n.
\end{equation}
The constant term of $f(q)$ is one half of the value at $s=0$ of the Kubota--Leopoldt $p$-adic $L$-function $L_p(\chi\omega,s)$. It turns out that $f(q)$ becomes cuspidal \emph{as an overconvergent $p$-adic} modular form when $L_p(\chi\omega,s)$ has an exceptional zero at $s=0$, that is:
\begin{assumption*}[Exceptional case]
	We have $\chi(p)=1$.
\end{assumption*}
As $f$ is ordinary at $p$, it gives rise to a maximal ideal $\gm_f$ of the Hecke algebra $\T_1(N)$ acting on the $\cO$-module $S_1^{p-\ord}(N,\cO)$ of weight $1$ overconvergent $p$-ordinary cusp forms of tame level $N$ with Fourier coefficients in $\cO$. 

We let 
\[\T_1=\T_1(N)_{\gm_f}\]
be the localization of $\T_1(N)$ at $\gm_f$. It can also be obtained as the quotient $\T/X\T$ of a local factor $\T$ of the Hida--Hecke algebra of tame level $N$, where $X$ is the variable of a power series ring $\Lambda= \cO\lsem X\rsem$ through which diamond operators act (see \S\ref{sec:Hecke_algebras_and_eigencurve} for details).

\subsection{The Gorenstein property in the quadratic weight $1$ case}\label{sec:intro_chi_quadratic}
Our first result characterizes precisely when $\T_1$, or equivalently $\T$, is Gorenstein. We focus on the case where $\chi^2=1$ (so that $\cO=\Zp$), the general case being discussed in \S\ref{sec:intro_chi_general}.

\begin{theorem}[Corollary \ref{coro:structure_T_1_chi_quadratic} and Theorem \ref{thm:converse_gorenstein}]\label{thm:intro_Gorenstein_quadratic}
	Assume $\chi^2=1$ and let $K$ be the imaginary quadratic field attached to $\chi$. Let $A_K$ denote the $p$-part of the class group of $K$. Let $A'_K$ be the quotient of $A_K$ by the subgroup generated by the ideal classes of primes of $K$ above $p$.
	\begin{enumerate}
		\item If $A'_K=\{0\}$, then $\T_1$ is Gorenstein; in fact, it is a local complete intersection of the form 
		\[\T_1\simeq \Zp\lsem x,y\rsem/(g_1,g_2)\] 
		for some $g_i\in \Zp\lsem x,y\rsem$ with $g_i(0,0)=0$.
		\item Conversely, if $A'_K\neq\{0\}$, then $\T_1$ is not Gorenstein. 
	\end{enumerate}
\end{theorem}

The Cohen--Lenstra heuristics predict that $A'_K=\{0\}$ for a positive proportion of imaginary quadratic fields $K$ in which $p$ splits; this is known unconditionally for $p=3$, where our computations give a proportion of about $89\%$ for $N\leq 10^5$ (see \S\ref{sec:numerical_data}).

Theorem \ref{thm:intro_Gorenstein_quadratic} has arithmetic consequences when $A'_K = \{0\}$. First, Sharifi's map $\varpi$ is injective as soon as it is surjective, as conjectured by McCallum--Sharifi (see Theorem~\ref{thm:sharifi}). This conjecture is part of Sharifi's broader conjectural framework relating modular symbols to cyclotomic class groups \cite{mccallumsharifi,sharifi}; we refer to \S\ref{sec:sharifi} for further discussion. Second, by work of Shih--Wang \cite{shihwang}, the algebraic and analytic $\mu$-invariants of every ordinary eigenform of weight $k \geq 2$ and tame level $N$ congruent to $f$ vanish.

\subsection{A rank-$1$ criterion in weight $1$}\label{sec:intro_rank1}
We now consider an arbitrary $\chi$, and we let $K$ be the abelian CM field cut out by $\chi$; when $\chi^2=1$, this is the imaginary quadratic field of Theorem \ref{thm:intro_Gorenstein_quadratic}. Our second main result is an analogue of Merel's theorem, for which we need $p$-adic substitutes for its two ingredients, namely $\lambda^{(N)}$ and $u_{\Merel}$.

The map $\lambda^{(N)}$ is replaced by
\[\lambda = \frac{\log_p}{\log_p(1+p)} \,\colon\, \Q_p^\times \to \Zp,\]
where $\log_p$ is Iwasawa's $p$-adic logarithm, normalized by $\log_p(p)=0$; the factor $\log_p(1+p)^{-1}$ ensures that $\lambda$ is surjective.

The number $u_{\Merel}$ is replaced by a pair of Gross--Stark units. As $p$ splits completely in $K$, the $\chi$- and $\chi^{-1}$-isotypic components of the $\cO$-module $\cO_K[1/p]^\times \otimes_\Z \cO$ of $p$-units of $K$ are free of rank $1$. We call any generators $u_\chi$ and $u_{\chi^{-1}}$ of these components \emph{normalized Gross--Stark units}, and we define the \emph{normalized $\cL$-invariants}
\begin{equation}\label{eq:intro_L_invariants}
	\cL_{\chi}=-\dfrac{\lambda(u_{\chi^{-1}})}{\ord_p(u_{\chi^{-1}})}, \qquad \cL_{\chi^{-1}}=-\dfrac{\lambda(u_{\chi})}{\ord_p(u_{\chi})},
\end{equation}
where $\ord_p$ is induced by the $p$-adic valuation on $K\subset \ob{\Q}_p$. They are nonzero elements of $\Frac(\cO)$ which do not depend on the choice of $u_{\chi^{\pm1}}$. Replacing $\lambda$ with $\log_p$ in \eqref{eq:intro_L_invariants} recovers the usual $\cL$-invariants of $\chi^{\pm1}$, which describe, via the Ferrero--Greenberg formula \cite{ferrerogreenberg}, the derivative at $s=0$ of the Kubota--Leopoldt $p$-adic $L$-functions $L_p(\chi^{\pm1}\omega,s)$. Put
\begin{equation}\label{eq:conjectural_congruence_number}
	c_f:=(\cL_{\chi}+\cL_{\chi^{-1}}) \cdot L(\chi,0) \cdot L(\chi^{-1},0),
\end{equation}
which is an element of $\cO$.

\begin{theorem}[Corollary \ref{coro:rank1}]\label{thm:intro_rank1}
	We have $\T_1=\cO$ if and only if $c_f$ is a $p$-adic unit.
\end{theorem}

Equivalently, $c_f$ is a $p$-adic unit if and only if the localization of $S_1^{p-\ord}(N,\cO)$ at $\gm_f$ is free of rank $1$ over $\cO$, generated by $f$.

The sum $\cL_{\chi}+\cL_{\chi^{-1}}$ is a familiar quantity. It occurs in the works of Darmon--Dasgupta--Pollack and Dasgupta--Kakde--Ventullo on the Gross--Stark conjecture \cite{DDP,dasguptakakdeventullo} and in the elliptic Stark conjecture of Darmon--Lauder--Rotger \cite[\S1.2.2]{DLR1}. Its non-vanishing allows Betina--Dimitrov--Pozzi to show that the cuspidal eigencurve is étale over the weight space at $f$ \cite{betinadimitrovpozzi}. Theorem \ref{thm:intro_rank1} shows that, together with the $L$-values $L(\chi^{\pm1},0)$, it also detects congruences between $f$ and other $p$-adic cusp forms.

When $\chi^2=1$, the factors of $c_f$ admit a finer interpretation. Let $\gp_K$ be the prime of $K$ above $p$ singled out by $\iota_p$, and let $u_K$ be a generator of $\gp_K^{h_K}$, where $h_K$ is the class number of $K$. Then $c_f$ equals $\lambda(u_K)\cdot L(\chi,0)$ up to a $p$-adic unit, so that 
\[\T_1=\Zp \iff p\nmid \lambda(u_K)\cdot L(\chi,0).\]
By the analytic class number formula, $L(\chi,0)$ has the same $p$-adic valuation as $\#A_K$, and it accounts for eigenforms with \emph{complex multiplication}, cut out by a quotient $\T^{\CM}$ of $\T$ whose weight $1$ points are parametrized by characters of $A_K$ (\S\ref{sec:CM}). By contrast, $\lambda(u_K)$ accounts for eigenforms without CM: we have $\T=\T^{\CM}$ if and only if $A'_K=\{0\}$ and $p\nmid\lambda(u_K)$ (Proposition \ref{prop:modularity_1_spl}).

\subsection{The Gorenstein property and congruence numbers}\label{sec:intro_chi_general}
For general $\chi$, let $A_K$ be the $p$-part of the class group of $K$, let $A'_K$ be its quotient by the subgroup generated by the classes of primes above $p$, and let $A'_K[\chi^{\pm1}]$ be the $\chi^{\pm1}$-isotypic components of $A'_K\otimes_{\Zp}\cO$.

Theorem \ref{thm:intro_rank1} is best understood in terms of the congruence number $c^0_f$ of $f$ (Definitions \ref{def:C} and \ref{def:congruence_number_of_f}). Roughly speaking, $c^0_f$ is an element of $\cO$, well defined up to a unit, which measures the $p$-adic congruences between $f$ and other forms in $S_1^{p-\ord}(N,\cO)$; it is a unit if and only if $\T_1=\cO$ (Lemma \ref{lem:rank1}). Thus Theorem \ref{thm:intro_rank1} says that $c_f$ and $c^0_f$ are simultaneously units. We propose the following conjecture, which strengthens Theorems \ref{thm:intro_Gorenstein_quadratic} and \ref{thm:intro_rank1}.

\begin{conjecture}[Conjectures \ref{conj:converse_gorenstein} and \ref{conj:length_C0}]\label{conj:intro_main}\ 
	\begin{enumerate}
		\item If $A'_K[\chi]=A'_K[\chi^{-1}]=\{0\}$, then $\T_1$ is a local complete intersection of the form $\T_1\simeq \cO\lsem x,y\rsem/(g_1,g_2)$ for some $g_i$ with $g_i(0,0)=0$. In particular, $\T_1$ is Gorenstein.
		\item Conversely, if $\T_1$ is Gorenstein, then $A'_K[\chi]=A'_K[\chi^{-1}]=\{0\}$.
		\item The element $c_f$ is a congruence number of $f$, \textit{i.e.}, $c_f=c^0_f$ up to a unit in $\cO$.
	\end{enumerate}
\end{conjecture}

When $\chi^2=1$, parts (1) and (2) of Conjecture \ref{conj:intro_main} follow from Theorem \ref{thm:intro_Gorenstein_quadratic}, and part (3) is proven in Theorem \ref{thm:congruence_bound}(1).
For non-quadratic $\chi$, we provide further theoretical and numerical evidence in support of Conjecture \ref{conj:intro_main}. We refer to the main body of the paper for precise statements. 

\begin{theorem}[Theorems \ref{thm:modularity}, \ref{thm:sharifi} and \ref{thm:trivial_zero_formula}]\label{thm:intro_non_quadratic}Assume $\chi^2\neq1$.
\begin{enumerate}
	\item If $\lambda(u_{\chi})$ or $\lambda(u_{\chi^{-1}})$ is a unit, then Conjecture \ref{conj:intro_main}(1) holds.
	\item If McCallum--Sharifi's Conjecture \ref{conj:mccallum_sharifi} holds for both $\chi$ and $\chi^{-1}$, then Conjecture \ref{conj:intro_main}(2) holds as well.
	\item Conjecture \ref{conj:intro_main}(3) implies Conjecture \ref{conj:intro_main}(1). It is also compatible with the exceptional zero formula for the adjoint $p$-adic $L$-function of $f$, with its ``motivic'' normalization.
\end{enumerate}
\end{theorem}

In Theorem \ref{thm:intro_non_quadratic}(3), we define the adjoint $p$-adic $L$-function $L_p(\ad f,s)$ of $f$ by specialization of Hida--Schmidt's adjoint $p$-adic $L$-function $L_p(\ad \cF,s)$ attached to the unique Hida family $\cF$ passing through $f$ \cite{schmidt,hidaGL3}. As the semisimple Galois representation attached to the adjoint motive of $f$ is
\[\ad\rho_f = 1\oplus \chi \oplus \chi^{-1},\]
one may expect that $L_p(\ad f,s)$ factors as the product of three Kubota--Leopoldt $p$-adic $L$-functions, following a $p$-adic Artin formalism.
However, as observed in Hida's monograph \cite{hidagenuine}, the weight $k\geq2$ specializations of $L_p(\ad \cF,s)$ should be slightly renormalized in order to fit the general conjectures on $p$-adic $L$-functions of motives formulated by Greenberg, Coates and Perrin-Riou \cite{greenberg1989iwasawa,coatesperrinriou,perrin1995fonctions}. 
Indeed, the automorphic periods involved in the interpolation formula for $L_p(\ad \cF,s)$ differ from Deligne's motivic period by a factor which should essentially be a congruence number (\S\ref{sec:genuine}). 

We prove an exceptional zero formula for 
\[L^\textrm{mot}_p(\ad f,s):=c_f \cdot L_p(\ad f,s), \]
with $c_f$ as in \eqref{eq:conjectural_congruence_number}, which is consistent with $p$-adic Artin formalism, yielding further evidence for Conjecture \ref{conj:intro_main}(3). Our formula is an Eisenstein analogue of an exceptional zero formula for weight $1$ cusp forms proven by Rivero--Rotger \cite[Theorem A]{riveroadjoint}. It implies a new Eisenstein case of a conjecture due to Darmon--Lauder--Rotger \cite{DLR2}.

We mention, in passing, that Shih--Wang \cite[Theorem 1.2]{shihwang} prove that the standard $\GL_2$ $p$-adic $L$-function of $f$, obtained by specializing the Mazur--Kitagawa $p$-adic $L$-function, factors, up to a unit in $\cO$, as a product of two Kubota--Leopoldt $p$-adic $L$-functions, as predicted by $p$-adic Artin formalism. They assume that $\T$ is Gorenstein, which is the content of Conjecture~\ref{conj:intro_main}(1), in order to normalize the periods integrally. Without this assumption, such a factorization only holds up to a nonzero constant (see Polo--Rivero \cite[Theorem A]{poloriveroeisenstein} for tame twists of $f$).

\subsection{The Coleman--Greenberg conjecture}
Given a classical newform $g$ of weight $k\geq2$ with $2$-dimensional $p$-adic Galois representation $\rho \colon \Gal(\ob{\Q}/\Q)\to \GL_2(\ob{\Q}_p)$, the restriction $\rho_{|\Gal(\ob{\Q}_p/\Qp)}$ is reducible when $g$ is $p$-ordinary. When $g$ has, moreover, complex multiplication by an imaginary quadratic field, $\rho_{|\Gal(\ob{\Q}_p/\Qp)}$ even splits as the sum of two characters. 

A question attributed to R. Greenberg and independently raised by R. Coleman asks whether the converse holds, namely, for $p$-ordinary $g$,
\[\rho_{|\Gal(\ob{\Q}_p/\Qp)}\simeq \psi_1 \oplus \psi_2 \overset{?}{\implies} \textrm{ $g$ has complex multiplication.} \]
Different approaches have been used to tackle this question: via Serre--Tate theory and its generalizations when $k=2$ \cite{serre89,hidaJAMS,zhao}, or via Hida theory and deformation theory when the residual Galois representation $\ob{\rho}$ satisfies a Taylor--Wiles hypothesis \cite{ghatevatsal,ghatevatsalcanadian,castellaWE,deowiese}. We refer to \cite[\S1.1]{castellaWE} for an overview on this question.

Our main result regarding the Coleman--Greenberg conjecture is the following.

\begin{theorem}[Corollary \ref{coro:CG}]\label{thm:intro_CG}
	Assume $\chi^2=1$ and $A_K'=\{0\}$, and let $g$ be an ordinary overconvergent cuspidal eigenform of arbitrary $p$-adic weight with Galois representation $\rho\colon \Gal(\ob{\Q}/\Q) \to \GL_2(\ob{\Q}_p)$. Assume that:
	\begin{enumerate}
		\item $g$ has tame level $N$,
		\item $\tr(\rho)\equiv \chi+1 \mod \gm_{\ob{\Q}_p}$, and
		\item $\rho_{|\Gal(\ob{\Q}_p/\Qp)}$ is strongly split. 
	\end{enumerate}
	Then $g$ has complex multiplication by $K$.
\end{theorem}
The strongly split condition is introduced in Definition \ref{def:strongly_split}; a relation to the $\cL$-invariant of $\chi$ is discussed in Lemma \ref{lem:strongly_split_reducible_case}. 

More generally, we prove in Theorem \ref{thm:CG} that a strong analogue of the Coleman--Greenberg question holds for strongly split Galois representations taking values in certain finite $p$-adic rings. 

\subsection{Overview of the proof}
\subsubsection{Defining a universal ordinary deformation ring}
We derive many of our results from a modularity theorem which identifies the Hecke algebra $\T$ with a universal deformation ring. Our residual semisimple Galois representation is $\ob{\rho}=\ob{\chi}\oplus1$, which is reducible and trivial on $\Gal(\ob{\Q}_p/\Qp)$. We cannot therefore rely on classical modularity theorems of Wiles and Taylor--Wiles \cite{wilesfermat,taylorwiles} (which typically assume $\ob{\rho}$ is irreducible), or on the notion of ordinary pseudo-deformations as in \cite{WWEajm} (which assumes $\ob{\rho}_{|\Gal(\ob{\Q}_p/\Qp)}$ is multiplicity-free).

We work instead with Calegari--Emerton-style deformation functors which parametrize Galois representations $\rho$ deforming $\ob{\rho}$ that are minimally ramified outside $p$, and possess a $p$-ordinary line in \emph{generic position} with respect to the canonical basis of $\ob{\rho}$. This last condition, introduced in \cite{calegariemerton} in Mazur's setting, rigidifies the deformation problem, yielding the existence of a universal minimal ordinary deformation ring $R^\ord$.

\subsubsection{Constructing a map $R^\ord \to \T$}
Perhaps not surprisingly, this generic ordinarity condition is already a crucial ingredient in the proof of the Gross--Stark conjecture by Darmon--Dasgupta--Pollack and Dasgupta--Kakde--Ventullo \cite{DDP,dasguptakakdeventullo}. They essentially prove, in a Hilbert setting, that the completion of $\T$ at $f$ carries a Galois representation that has an ordinary line in generic position, in order to apply Ribet's method. Their argument implicitly rests on working with $p$ inverted, where the finite group $A'_K$ becomes invisible. When working globally as in this paper, we show using GMA methods that the triviality of $A'_K[\chi^{\pm1}]$ is precisely what is needed to construct a map $R^\ord \to \T$. 

\subsubsection{Proving $R^\ord=\T$}
We apply Lenstra's strengthening of Wiles' numerical criterion in order to prove that the surjective map $R^\ord \to \T$ is an isomorphism. Unlike the usual approach, which proceeds by specializing to weight $k\geq2$, we work directly with the weight $1$ specialization $R^\ord_1 \to \T_1$. Although very little is known \emph{a priori} about $\T_1$, not even whether it is reduced, the criterion requires only that $\T_1$ be finite flat over $\cO$ and that the congruence number $c_f^0$ of $f$ be nonzero with sufficiently large $p$-adic valuation. We establish flatness using Hida theory, and prove the non-vanishing of $c^0_f$ by exploiting the étaleness of the eigencurve over the weight space at $f$, established in \cite{betinadimitrovpozzi}. 

The second key idea of the paper is to construct sufficiently many congruences between $f$ and other overconvergent weight $1$ cusp forms using the two Eisenstein families passing through $f$. These congruences yield a lower bound on the $p$-adic valuation of $c^0_f$, thereby proving Theorem \ref{thm:intro_non_quadratic}(1). We take a different approach in the quadratic case: we relate $c^0_f$ to a special value of Katz's $p$-adic $L$-function, which we compare to $c_f$ in \eqref{eq:conjectural_congruence_number} using the $p$-adic Kronecker limit formula.

\subsubsection{Failure of Gorensteinness of $\T$}
 Assume that $\chi^2=1$ and $A'_K\neq \{0\}$. To prove that $\T$ is not Gorenstein, we construct another weight $1$ point of $\Spec(\T)$ corresponding to a theta series $\theta$ that is irregular at $p$, in the sense that its associated $\ob{\Q}_p$-valued Galois representation is trivial on $\Gal(\ob{\Q}_p/\Qp)$. The geometry of $\Spec(\T)$ at such points is the main theme of the work of Betina--Dimitrov \cite{betinadimitrovKatz}, who show, in particular, that $\Spec(\T)$ is not Gorenstein at $\theta$. It follows that $\T$ itself is not Gorenstein.  
 
 Without assuming that $\chi$ is quadratic or that $A'_K$ is nontrivial, we show that McCallum--Sharifi's Conjecture \ref{conj:mccallum_sharifi}, together with a Gorensteinness assumption on $\T$, forces certain Iwasawa modules attached to $\chi$ and $\chi^{-1}$ to be cyclic. Standard arguments in Iwasawa theory inspired by the work of Wake \cite{wake} then imply that $A'_K[\chi]=A'_K[\chi^{-1}]=\{0\}$.

\subsubsection{Splitting slopes and the strong Coleman--Greenberg question}
Assume $\chi^2=1$, \textit{i.e.}, $K$ is imaginary quadratic. The strong Coleman--Greenberg question asks for sufficient conditions under which a Galois representation with values in a ring such as $\cO/(p^n)$, whose restriction to $\Gal(\ob{\Q}_p/\Qp)$ is split, must have abelian image when restricted to $\Gal(\ob{\Q}/K)$. Under the Taylor--Wiles hypotheses, which in particular require the residual representation to be irreducible and $p$-distinguished, Castella and Wang-Erickson give a sufficient condition in \cite[Theorem 1.3.1]{castellaWE}, namely the vanishing of a certain class group. Their criterion is established by proving an $R^{\spl}=\T^{\CM}$ theorem, where $R^{\spl}$ is a universal split deformation ring constructed by Ghate--Vatsal \cite{ghatevatsalcanadian}. 

In our setting, this construction of $R^{\spl}$ does not apply. The reason is that $R^\ord$ admits an $\cO$-point corresponding to $\rho_1=\chi\oplus 1$ whose restriction to $\Gal(\ob{\Q}_p/\Qp)$ is trivial, so it is split with respect to every $\cO$-basis. To circumvent this difficulty, we construct a universal deformation ring $R^{\mu-\spl}$ parametrizing split representations with a prescribed splitting slope $\mu\in\cO$. We then prove that $R^{\mu-\spl}\simeq\T^{\CM}$ when $\mu=1$ and $A'_K=\{0\}$. This reduces the problem to finding a condition that forces a split Galois representation to have splitting slope $\mu$ equal to $1$. Our main contribution is to show that the condition we call \emph{strong splitness} is sufficient to enforce this equality.

\subsection*{Notation}
Throughout the paper, $p$ is an odd prime, $\chi \colon (\Z/N\Z)^\times \to \ob{\Q}_p^\times$ is a primitive odd character with $p\nmid N\phi(N)$ and $\chi(p)=1$. We let $\cO$ be the unramified extension of $\Z_p$ generated by the values of $\chi$ and we denote by $E$ its fraction field. We also let $\rD=E/\cO$.

For any number field $L\subseteq \ob{\Q}$, let $G_L=\Gal(\ob{\Q}/L)$. We denote by $\tau\in G_{\Q}$ the complex conjugation and, for any prime $\ell$, by $I_\ell$ and $G_{\Q_\ell}=\Gal(\ob{\Q}_\ell/\Q_\ell)$ the inertia and decomposition subgroups of $G_{\Q}$ corresponding to a fixed embedding $\iota_\ell \colon \ob{\Q} \hookrightarrow \ob{\Q}_\ell$. We sometimes see $\chi$ as a character of $G_{\Q}$. Let $K$ be the abelian CM field cut out by $\chi$, and put $G=\Gal(K/\Q)$, which is a finite cyclic group of prime-to-$p$ order.

We let $v_p$ be the valuation of $\cO$, normalized so that $v_p(p)=1$. Hence, for $a\in\cO\setminus\{0\}$, the length of $\cO/(a)$ as an $\cO$-module is $v_p(a)$. If $b=(b_i)_{1\leq i\leq n}\in E^n\setminus\{0\}$, we let $\gcd(b)\in E\setminus\{0\}$ be any generator of the fractional ideal generated by the $b_i$'s.

If $A$ is any abelian group, we let $A^\wedge=\varprojlim_r A/p^r \otimes_{\Zp}\cO$ be its $p$-adic completion (linearly extended to $\cO$). If $A$ is finitely generated, then $A^\wedge=A \otimes_{\Z} \cO$.

For any number field $L$, we let $A_L=\Cl(L)^\wedge$, where $\Cl(L)$ is the ideal class group of $L$. We also denote by $A'_L$ the quotient of $A_L$ by the $\cO$-submodule generated by the images of ideal classes of primes of $L$ lying above $p$.

For an $\cO[G]$-module $M$, and an $\cO$-valued character $\varphi$ of $G$, we let $M[\varphi]=\{m\in M : g\cdot m=\varphi(g)m\}$. Since $p\nmid \#G$, the assignment $M \mapsto M[\varphi]$ defines an exact functor from $\cO[G]$-Mod to $\cO$-Mod.

Let $\chi_{\cyc}$ be the usual $p$-cyclotomic character of $G_{\Q}$, $\omega$ the Teichm\"uller character, and $\varepsilon_{\cyc}=\chi_{\cyc}\omega^{-1}$. Therefore, $\varepsilon_{\cyc}$ identifies the Galois group $\Gamma=\Gal(\Q_{\cyc}/\Q)$ of the $\Zp$-cyclotomic extension $\Q_{\cyc}$ of $\Q$ with $1+p\Zp$. We let $\gamma\in \Gamma$ be the topological generator sent to $1+p$. This allows us to identify the completed group ring $\Lambda:=\cO\lsem\Gamma\rsem$ with $\cO\lsem X\rsem$ via $[\gamma]\mapsto 1+X$. We also fix a lift $\phi \in G_{\Qp}$ of the Frobenius element such that $\varepsilon_{\cyc}(\phi)=1$.

If $M$ is a module over a ring $A$ and $m\in M$, we sometimes denote by $\langle m \rangle_A$ or $\langle m \rangle$ the $A$-submodule of $M$ generated by $m$. If $a\in A$ then $M[a]$ denotes the $A$-submodule of $M$ consisting of all $a$-torsion elements. 

\tableofcontents

\section{Galois representations and their deformations}

\subsection{Gross--Stark units and $\cL$-invariants}\label{sec:GS_units}
Let $U_K=\cO_K^\times \otimes_{\Z}\cO$ be the $p$-adic completion of units of $K$, and $U'_K=\cO_K[1/p]^\times \otimes_{\Z}\cO$ the one of $p$-units of $K$. Let $\varphi$ be any odd $\cO$-valued character of $G$. As $p$ splits completely in $K$, $\varphi(p)=1$. In the next sections, we will mainly take $\varphi$ to be $\chi$ or $\chi^{-1}$.

Herbrand's strengthening of Dirichlet's theorem \cite[Chapitre 1, \S4]{tate} implies that $U_K[\varphi]$ and $U'_K[\varphi]$ have respective $\cO$-ranks $0$ (since $\varphi$ is odd) and $1$ (since $\varphi(p)=1$). Moreover, $p$ is odd and unramified in $K$, so these $\cO$-modules are torsion-free, hence free.

\begin{definition}\label{def:GS_unit}
	We define the normalized Gross--Stark unit $u_\varphi$ attached to $\varphi$ to be any generator of the free $\cO$-module $U_K'[\varphi]$ of rank $1$. In particular, $u_\varphi$ is unique up to a unit of $\cO$.
\end{definition}

Recall from class field theory that the continuous cohomology group 
\[\HH^1(\Qp,E)\simeq\Hom_{\textrm{cts}}(\Q_p^\times,E)= \Hom_{\textrm{cts}}(1+p\Z_p,E) \times \Hom(p^\Z,E) \]
is of dimension $2$ over $E$, and it is generated by the $p$-adic logarithm $\log_p$ and the $p$-adic valuation map $\ord_p$. Since $\log_p(1+p\Z_p)=\log_p(1+p)\Z_p=p\Z_p$ and $\ord_p(p^\Z)=\Z$, it follows that its $\cO$-submodule \[\HH^1(\Qp,\cO)\simeq\Hom_{\textrm{cts}}(\Q_p^\times,\cO)\]
is free of rank $2$ with basis $\lambda$ and $\ord_p$, where
\[\lambda = \dfrac{\log_p}{\log_p(1+p)} \colon \Q_p^\times \to \cO. \]
We may see $\lambda$ and $\ord_p$ as morphisms from $U'_K$ to $\cO$ as follows. As $p$ splits completely in $K$, the distinguished $p$-adic embedding $\iota_p\colon K\subset \ob{\Q} \hookrightarrow \ob{\Q}_p$ induces a natural map $\cO_K[1/p]^\times \hookrightarrow \Q_p^\times$.  Post-composing with $\lambda$ and $\ord_p$, and taking a $p$-adic completion of the source induce two $\cO$-linear maps
\begin{equation}\label{eq:def_lambda_and_ord}
	\lambda,\ord_p \colon U'_K \to \cO,
\end{equation}
both of which are surjective. 

\begin{lemma}\label{lem:calcul_ord}
We have $\ord_p(u_\varphi)\neq 0$, and
		\[\length \cO/(\ord_p(u_\varphi)) = \length A_K[\varphi] - \length A'_K[\varphi]. \]
\end{lemma}
\begin{proof}
	If $\ord_p(u_\varphi)=0$, then $u_\varphi \in U'_K[\varphi] \cap U_K=U_K[\varphi]=\{0\}$, which is a contradiction. Hence, $\ord_p(u_\varphi)\neq0$.
		
	Consider the free $\cO$-module $B_K$ generated by classes of primes above $p$, and let $\gp_K$ be the distinguished prime of $K$ above $p$. We identify $U'_K$ with a finite index $\cO$-submodule of $B_K$ via the map $u\mapsto \sum_{\gp\mid p}\ord_\gp(u)[\gp]$, where $\ord_\gp$ is the $\gp$-adic valuation map.
	
	 The projection $B_K \twoheadrightarrow \cO$ onto the $[\gp_K]$-component restricts to a surjection $B_K[\varphi] \twoheadrightarrow \cO \twoheadrightarrow \cO/(\ord_p(u_\varphi)) $ whose kernel can be identified with $\langle u_\varphi\rangle=U'_K[\varphi]$. On the other hand, the natural map $B_K \to A_K$ induces an exact sequence
	\[0 \to U'_K[\varphi] \to B_K[\varphi] \to A_K[\varphi] \to A'_K[\varphi] \to 0  \]
	of $\cO$-modules. Therefore, the length of $\cO/(\ord_p(u_\varphi))\simeq  B_K[\varphi]/U'_K[\varphi]$ is equal to $\length A_K[\varphi]- \length A'_K[\varphi]$.
\end{proof}

\begin{definition}
	The normalized $\cL$-invariant attached to $\varphi$ is 
	\begin{equation}\label{eq:def_L_inv}
		\cL_\varphi = -\dfrac{\lambda(u_{\varphi^{-1}})}{\ord_p(u_{\varphi^{-1}})}\in E.
	\end{equation}
\end{definition}
The definition of $\cL_\varphi$ does not depend on the generator of $U'_K[\varphi^{-1}]$. However, we warn the reader that the usual definition of $\cL$-invariant for $\varphi$ is obtained by replacing $\lambda$ with $\log_p$ in \eqref{eq:def_L_inv}.

\begin{lemma}\label{lem:linear_independence_L_invariants}
	The elements $\cL_\varphi \in E$, for $\varphi$ running over the set of odd characters of $G$, are $\ob{\Q}$-linearly independent. In particular, they are nonzero.
\end{lemma}
\begin{proof}
	Since $\varphi$ takes values in $\ob{\Q}^\times\cap\cO^\times$, we may assume for every $\varphi$ that the generator $u_\varphi$ of $U'_K[\varphi]$ lives in $\cO_K[1/p]^\times \otimes_{\Z} \ob{\Q}$. In particular, $\ord_p(u_\varphi)\in\ob{\Q}^\times$. 
	
	Any vanishing $\ob{\Q}$-linear combination of the $\cL_\varphi$'s gives rise to an element $u=\prod_\varphi u_\varphi^{a_\varphi}\in \cO_K[1/p]^\times \otimes_\Z \ob{\Q}$ (written multiplicatively) lying in the kernel of $\log_p \colon \cO_K[1/p]^\times \otimes_\Z \ob{\Q} \to \ob{\Q}_p$.
	
	 We may then apply \cite[Proposition 2.2]{maksoud_gross}, which shows that $u$ is proportional to $p$. Since $G$ acts trivially on $p$, and via $\varphi$ on $u_\varphi$, we deduce that $a_\varphi=0$ for any $\varphi$, by applying the idempotent of $\varphi$ to $u$. This shows that $u=1$, and the lemma follows easily.
\end{proof}

In the next theorem due to Mazur and Wiles, recall that $L(\varphi,0)$ is (the negative of) a generalized Bernoulli number, hence a nonzero $\Z[1/N]$-linear combination of values of $\varphi$. In particular, $L(\varphi,0)\in \cO\setminus\{0\}$.
\begin{theorem}[Refined class number formula]\label{thm:refined_CNF}
	Under our running assumption $p\nmid \#G$, we have $\length A_K[\varphi^{-1}]=  \length \cO/(L(\varphi,0))$. 
\end{theorem}
\begin{proof}
	This is \cite[Theorem 2, p.216]{mazurwiles}.
\end{proof}

\subsection{Integral Galois cohomology of Dirichlet characters}\label{sec:galois_cohomology}
Let $\cX_K=\Gal(M_K/K) \otimes_{\Zp} \cO$, where $M_K/K$ is the maximal abelian pro-$p$ extension of $K$ unramified outside the places dividing $p$. By \cite[Lemma 10.3.12]{neukirchcohomology}, there is a commutative diagram 
\begin{equation}\label{eq:diagramm_CFT}
\begin{tikzcd}
	0 \rar & U_K \rar \ar[d, hook] & \prod_{\gp|p} \cO_{K_\gp}^{\times,\wedge} \ar[d, hook] \ar[r, "\textrm{art}" ]  & \cX_K \ar[d, equal] \rar & A_K \rar \ar[d, two heads] & 0 \\
	0 \rar & U'_K \rar & \prod_{\gp|p} K_\gp^{\times,\wedge} \ar[r, "\textrm{art}" ]  & \cX_K \rar & A'_K \rar & 0
\end{tikzcd}
\end{equation}
of finitely generated $\cO[G]$-modules with exact rows, where $\textrm{art}$ stands for the Artin reciprocity map. 

Given any $\cO$-valued odd character $\varphi$ of $G$, let $\rD(\varphi)$ be the $\cO$-module $\rD=E/\cO$ endowed with a $G$-action given by $\varphi$. We then may form the discrete Selmer group
\begin{equation}\label{eq:def_Sel_phi}
 \HH^1_{\ff,p}(\Q,\rD(\varphi))= \ker\left[\HH^1(\Q,\rD(\varphi)) \longrightarrow \prod_{\ell \neq p} \HH^1(\Q_\ell^\mathrm{nr},\rD(\varphi)) \right]   
\end{equation}
consisting of cohomology classes that are unramified outside $p$. 
As $p$ does not divide the order of $G=\Gal(K/\Q)$, the inflation-restriction exact sequence provides identifications
\begin{equation}\label{eq:description_H1(Q,chi)}
	\HH^1_{\ff,p}(\Q,\rD(\varphi))=\Hom_\cO(\cX_K,\rD)[\varphi^{-1}]=\Hom_\cO(\cX_K[\varphi],\rD).
\end{equation}
In the next lemma, we still use the same symbols to denote the image of the distinguished elements $\gamma \in \Gamma$ and $\phi\in G_{\Qp}$ in $\cX_K$ and its direct summand $\cX_K[\varphi]$. 

\begin{lemma}\label{lem:description_Xvarphi}
	Let $\varphi$ be an $\cO$-valued odd character of $G$, and let $x_{\varphi}=\lambda(u_{\varphi})$ and $y_{\varphi}=\ord_p(u_{\varphi})$. There is a short exact sequence of finitely generated $\cO$-modules
	\[ 0 \longrightarrow \dfrac{\cO \times \cO}{\langle (x_\varphi,y_\varphi)\rangle} \overset{(p_1,p_2)}{\longrightarrow}\cX_K[\varphi] \longrightarrow A'_K[\varphi] \longrightarrow 0, \]
	with $p_1(1)=\gamma$ and $p_2(1)=\phi$.
\end{lemma}
\begin{proof}
	Let $\gp_K$ be the prime of $K$ above $p$ determined by $\iota_p(K) \subset \Qp$. By Frobenius reciprocity, there are $\cO$-linear isomorphisms 
\begin{equation*}\label{eq:identification_local_units}
	\left(\prod_{\gp|p}\cO^{\times,\wedge}_{K_\gp}\right)[\varphi]\simeq \cO^{\times,\wedge}_{K_{\gp_K}}=(1+p\Z_p) \otimes_{\Zp}\cO \simeq \cO,
\end{equation*}
	the first one being obtained by projecting onto the $\gp_K$-component of the product, and the last one being induced by $\lambda$. We also note that by definition of $\gamma$, we have $\lambda(\gamma)=1$.
	
	Similarly, if $p^\cO$ denotes the $\cO$-submodule of $K_{\gp}^{\times,\wedge}$ generated by $p$, then
	\[(\prod_{\gp|p}p^\cO)[\varphi]\simeq p^\cO \simeq \cO,\]
	where the last isomorphism is induced by $\ord_p$. We also note that the image of $\phi\in G_{\Qp}$ in the $\gp_K$-component of the product $\prod_{\gp|p} K_{\gp}^{\times,\wedge}$ is precisely $p$, as $\lambda(\phi)=0$.
	
	The lemma hence follows from an application of the exact functor $M \mapsto M[\varphi]$ to the second row in \eqref{eq:diagramm_CFT} and the previous description of its second term.
\end{proof} 

\subsection{Ordinary deformations of Galois representations}\label{sec:ordinary_deformations}
Let $\mathcal{CNL}_{\cO}$  be the category of complete local noetherian $\cO$-algebras $(A,\gm_A)$ such that $k=\cO/(p) \to A/\gm_A$ is an isomorphism, and the morphisms are local $\cO$-algebra homomorphisms.

Let $\ob{\rho}=\ob{\chi}\oplus 1 \colon G_{\Q} \to \GL_2(k)$ and fix a $k$-line $\ob{L}$ of $\ob{V}:=k^2$ that is \emph{not} $G_\Q$-stable. 
	
\begin{definition}\label{def:ord}
	The minimal ordinary deformation functor 
	\[\cD^\ord \colon \mathcal{CNL}_\cO \to \mathcal{SET}\]
	is the functor sending an object $A$ of $\mathcal{CNL}_\cO$ to the set of pairs $(\rho,L)$ where $\rho \colon G_\Q \to \GL_2(A)$ is a strict equivalence class of Galois representations unramified outside $Np$, and $L$ is a free rank-$1$ direct summand of $V:=V_\rho=A^2$ such that $(V,L)$ is a deformation of $(\ob{V},\ob{L})$, and:
	\begin{enumerate}
		\item for all primes $\ell|N$,  $\rho_{|I_\ell} \simeq \chi_{|I_\ell}\oplus\mathds{1}$, where $\chi$ is seen as an $A$-valued character via the structural map $\cO \to A$, and
		\item $L$ is stable under $G_{\Qp}$ with $I_p$ acting trivially on $V/L$. 
	\end{enumerate}
\end{definition}
Here, we recall that two $A$-valued representations $\rho,\rho'$ are \emph{strictly equivalent} if they are conjugate by a matrix in $1+\gm_AM_2(A)$. In our context, it is possible to pin down a canonical representative in this equivalence class, as we shall see in the following lemma.

\begin{lemma}\label{lem:eisenstein_basis}
	For any $\rho \colon G_\Q \to \GL_A(V)$ and $L\subset V$ such that $(\rho,L)\in \cD^\ord(A)$, there exists a unique $A$-basis of $V$, up to homothety, in which $\rho(\tau)=(\begin{smallmatrix}-1 & 0 \\ 0 & 1	\end{smallmatrix})$ and $L$ is generated over $A$ by $(\begin{smallmatrix} 1 \\ 1	\end{smallmatrix})$.
\end{lemma}
\begin{proof}
	As $p$ is odd, there exists an $A$-basis $(e_1,e_2)$ of $V$ in which $\rho(\tau)=(\begin{smallmatrix}-1 & 0 \\ 0 & 1 \end{smallmatrix})$. Since $L$ does not reduce to $\langle e_1 \rangle$ or $\langle e_2 \rangle$ modulo $\gm_A$, we may rescale $e_2$ so that $L$ is generated by $e_1+e_2$, as claimed.

	If $(e'_1,e'_2)$ is another $A$-basis of $V$ in which $\rho(\tau)=(\begin{smallmatrix}-1 & 0 \\ 0 & 1 \end{smallmatrix})$, then $e'_1=s_1e_1$ and $e'_2=s_2e_2$ for some $s_i\in A^\times$. The condition that $L=\langle e'_1+e'_2\rangle$ implies that $s_1=s_2$, and therefore shows that $(e'_1,e'_2)$ and $(e_1,e_2)$ are homothetic.
\end{proof}

\begin{definition}\label{def:eisenstein_basis}
	The basis constructed in Lemma \ref{lem:eisenstein_basis} is referred to as the Eisenstein basis for $(\rho,L)$. We denote by $\rho^{\text{Eis},L}$, or simply $\rho$, the resulting $\GL_2(A)$-valued group homomorphism.
\end{definition}

The proof of Lemma \ref{lem:eisenstein_basis} essentially tells us that the endomorphism ring of any $(\rho,L)\in \cD^\ord(A)$ is $A$. This is the key point underlying the representability of $\cD^\ord$, which we now establish.

\begin{proposition}\label{prop:representability}
	$\cD^\ord$ is representable by a complete local noetherian ring $R^\ord$. 
\end{proposition}
\begin{proof}
	We apply Grothendieck's theorem, which is a special case of Schlessinger's criterion (see \cite[\S18]{mazurdeforming}). 
	
	The Zariski tangent space of $\cD^\ord$ is finite dimensional over $k$, since the Galois group $G_\Q^{(Np)}$ of the maximal extension of $\Q$ unramified outside $Np$ satisfies Mazur's $p$-finiteness condition. Therefore, it is enough to check that 
	\[\cD^\ord(B) \overset{\sim}{\longrightarrow} \cD^\ord(A_1) \times_{\cD^\ord(A)} \cD^\ord(A_2) \]
	for any maps $A_i \to A$ in $\mathcal{CNL}_{\cO}$ and $B=A_1 \times_{A} A_2$.
	
	Choose $(\rho_i,L_i)\in\cD^\ord(A_i)$ with $\rho_i=\rho_i^{\text{Eis},L_i} \colon G_{\Q} \to \GL_2(A_i)$ computed in the Eisenstein basis, and assume they reduce to the same point $(\rho,L)\in\cD^\ord(A)$. By the uniqueness part in Lemma \ref{lem:eisenstein_basis}, we literally have $\rho_1=\rho_2$ as $A$-valued homomorphisms. We may therefore define $\rho' =(\rho_1,\rho_2) \colon G_{\Q} \to \GL_2(B)$ and $L'=\langle(\begin{smallmatrix} 1 \\ 1	\end{smallmatrix})\rangle_B$. 
	
	To conclude that $(\rho',L')\in \cD^\ord(B)$, we need to check that it satisfies the minimality condition at every prime $\ell|N$. For this, one may argue as in the proof of Lemma \ref{lem:eisenstein_basis}, after noticing that $\bar{\chi}_{|I_\ell}\neq 1$, as $\chi$ is of prime-to-$p$ order. The uniqueness of the lift $(\rho',L')\in \cD^\ord(B)$ again follows from the existence of the Eisenstein basis.
\end{proof}

We now analyze some quotients of $R^\ord$ corresponding to a ``weight $k$ specialization'' by requiring adequate determinant conditions.

Recall that $\Gamma$ is the Galois group of the $\Zp$-cyclotomic extension $\Q_\textrm{cyc}/\Q$, and $\varepsilon_{\cyc} =\chi_{\cyc}\omega^{-1} \colon \Gamma \simeq 1+p\Zp$ is the cyclotomic character. 
\begin{definition}
	Let $k\in \Z$. We let $\cD_k^\ord$ be the subfunctor of $\cD^\ord$ parametrizing Galois representations with fixed determinant $\chi\varepsilon_{\cyc}^{k-1}$. We define the weight $k$ specialization of the ring $R^\ord$ constructed in Proposition \ref{prop:representability} as the quotient $R^\ord_k$ representing $\cD_k^\ord$.
\end{definition}
We may explicitly construct $R_k^\ord$ as follows. Let $\Lambda:=\cO\lsem \Gamma\rsem $ be the completed group ring of $\Gamma$. Then $\chi\varepsilon^\mathrm{u}$ is a universal lift of $\chi$, where $\varepsilon^\mathrm{u} \colon G_\Q \to \Lambda^\times$ is the continuous Galois character factoring through $\Gamma$ and sending any $\gamma\in\Gamma$ to $[\gamma]\in\Lambda^\times$. If $\rho$ is the universal ordinary deformation of $\ob{\rho}$, then $\det(\rho)$ is a lift of $\chi$, and therefore endows $R^\ord$ with a $\Lambda$-algebra structure.

Hence, if $\gamma$ denotes the topological generator of $\Gamma$ sent to $1+p$ via $\varepsilon_{\cyc}$ and yielding the identification $\Lambda\simeq\cO\lsem X\rsem $ sending $[\gamma]$ to $1+X$, then 
\[R^\ord_k=R^\ord/(1+X-(1+p)^{k-1}).\]

\subsection{Congruence numbers and the numerical criterion}
We start by recalling Lenstra's strengthening of Wiles' numerical criterion, involving the congruence ideals which we first review. 
 
 \begin{definition}\label{def:C}
 	Let  $\pi \colon R \twoheadrightarrow \cO$ be a map in $\mathcal{CNL}_\cO$, with kernel $\gp$. For $i=0,1$ we define the $i$-th congruence module $C^i(\pi,R)=C^i(\pi)$ of $\pi$ as 
 	\[C^0(\pi) = \cO/\eta, \qquad C^1(\pi)=\gp/\gp^2, \]
 	where $\eta=\pi(\Ann_R(\gp))$. Any generator of the principal ideal $\eta$ will be called a \emph{congruence number} of $\pi$.
 \end{definition}
 
 \begin{lemma}\label{lem:rank1}
 	With the notation as in Definition \ref{def:C}, $\pi$ is an isomorphism if and only if $C^0(\pi)=\{0\}$, if and only if $C^1(\pi)=\{0\}$.
 \end{lemma}
\begin{proof}
	If $\gp=\{0\}$, then we clearly have $C^i(\pi)=\{0\}$ for $i\in\{0,1\}$. The converse statement for $C^1$ follows from Nakayama's lemma. Similarly, if $C^0(\pi)=\{0\}$, then there is $a\in R$ such that $\pi(a)=1$ and $a\,\gp = \{0\}$. Letting $b=1-a$, we have $b\in \gp$ and $b\,\gp=\gp$, which shows $\gp=\{0\}$.
\end{proof}
 
\begin{theorem}\label{thm:numerical_criterion}
	Suppose we are given a commutative triangle 
	\[ 
	\begin{tikzcd}[sep=small]
		R \arrow[rr, "\varphi"] \arrow[rd, "\pi_R"'] &   & T \arrow[ld, "\pi_T"] \\
		& \cO &
	\end{tikzcd}
	\]
	of surjective maps in $\mathcal{CNL}_\cO$, with $T$ finite flat over $\cO$. 
	
	If $C^0(\pi_T)$ is of finite length over $\cO$, then 
\begin{equation}\label{eq:C1geqC0}
		\length C^1(\pi_R)\geq \length C^0(\pi_T), 
\end{equation}
	with equality if and only if $\varphi$ is an isomorphism of local complete intersections. 
\end{theorem}
\begin{proof}
	This is \cite[Criterion 1]{desmitrubinschoof}.
\end{proof}
In the above theorem, we recall that a finite local $\cO$-algebra $T$ is a local complete intersection if there exists an integer $n\geq0$ such that $T$ is of the form $\cO\lsem x_1,\ldots,x_n\rsem/(f_1,\ldots,f_n)$. Moreover, $f_1,\ldots,f_n$ is necessarily a regular sequence.

We conclude our commutative algebra discussion with the proof of elementary facts about the congruence module $C^0$. In both Theorem \ref{thm:numerical_criterion} and Lemma \ref{lem:congruence_module_facts} below, $T$ is not assumed to be reduced.

\begin{lemma}\label{lem:congruence_module_facts}
	Let $\pi\colon T \twoheadrightarrow \cO$ be a map in $\mathcal{CNL}_\cO$ with $T$ finite flat over $\cO$. Let $\gp=\ker\pi$, and assume that $T$ is étale over $\cO$ at $\gp$. 
	\begin{enumerate}
		\item $C^0(\pi)$ has finite length over $\cO$.
		\item  There exists an idempotent $e\in T \otimes_{\cO}E$ such that the map $e(T \otimes_{\cO}E)\subseteq T \otimes_{\cO}E \to E$ induced by $\pi$ is an isomorphism.
		\item Let $S=\Hom_{\cO}(T,\cO)$ be the dualizing module of $T$. Then $\length C^0(\pi)$ is the largest integer $n\geq0$ such that there exists $\pi'\in S$ with $e(\pi')=0$ and $\pi \equiv \pi' \!\!\mod p^nS$.
	\end{enumerate}
\end{lemma}
\begin{proof}
	Let $\eta=\pi(\Ann_{T}(\gp))\subseteq \cO$. By assumption, $T_\gp$ is an étale $E$-algebra. It is local with residue field $E$, so $T_\gp=E$. Therefore, we have an $E$-algebra decomposition $T \otimes_{\cO}E =E \times X$ for some artinian $E$-algebra $X$, where the projection to the first factor is induced by $\pi$. This already proves (2), with $e=(1,0)$.
	
	Since $T$ is $\cO$-flat, we may inject it into $E \times X$, and $\Ann_{T}(\gp)\supseteq T\cap(E\times\{0\})$. For $n\geq0$ large enough, $t:=(p^n,0)$ is in $T$, and $\pi(t)=p^n$ is a nonzero element of $\eta$. This proves (1).
	
	For the last point, let $T_1=T\cap(E\times\{0\})=T \cap e(T)$. We first note that the inclusion $\Ann_{T}(\gp)\supseteq T_1$ noticed above is, in fact, an equality. Indeed, we may take $n'$ large enough so that $t'=(0,p^{n'})$ lies in $T$, and hence in $\gp$. Since $p$ is a regular element of the $E$-algebra $X$, any $a\in T$ annihilating $t'$ must be of the form $(a',0)$. In particular, $\pi$ identifies $e(T)/T_1$ with $\cO/\pi(\Ann_T(\gp))=C^0(\pi)$. 
	
	We now consider $S_1=S\cap e(S)$, the intersection being taken inside $S \otimes_{\cO} E$. It is an $\cO$-submodule of $e(S\otimes_{\cO} E)=E\pi$, and it is, in fact, generated by $\pi$. On the other hand, the canonical perfect pairing $S \times T \to \cO$ induces a pairing $e(S) \times e(T)/T_1 \to E/\cO$ whose left kernel is $S_1$. We obtain isomorphisms $e(S)/S_1\simeq e(T)/T_1\simeq C^0(\pi)$ of finite cyclic $\cO$-modules. Hence, $\length C^0(\pi)$ is the largest integer $n\geq0$ such that there exists $\pi_1\in S$ with $p^n e(\pi_1)=\pi$. Letting $\pi'=\pi-p^n \pi_1$, we have $e(\pi')=0$ and the equivalence in (3) easily follows.
\end{proof}

\subsection{Computation of tangent spaces}
In the rest of the paragraph, we consider the $\cO$-valued lift
\begin{equation}\label{eq:rho_f}
	\rho_1=\begin{pmatrix}\chi & 0 \\ 0 & 1\end{pmatrix} 
\end{equation}
of $\ob{\rho}$, together with its $\cO$-line $L_1$ generated by $(\begin{smallmatrix} 1 \\ 1	\end{smallmatrix})$. Note that $L_1$ (as any other line) is stable under the action of $G_{\Qp}$, because $\rho_{1|G_{\Qp}}$ is trivial.

\begin{lemma}\label{lem:description_C1_in_weight_one}
	 Let $\pi_1 \colon R_1^\ord \twoheadrightarrow \cO$ be the morphism induced by the point $(\rho_1,L_1)\in\cD_1^\ord(\cO)$, and call $\pi$ the composition $R^\ord \twoheadrightarrow R_1^\ord \twoheadrightarrow \cO$. 

	\begin{enumerate}
	\item There is a natural isomorphism 
	\[\Hom_\cO(C^1(\pi),\rD)\simeq \HH^1_\ord(\Q,\ad \rho_1 \otimes_{\cO}\rD), \]
	where the right-hand side is the $\cO$-submodule of $\HH^1(\Q,\ad \rho_1 \otimes_{\cO}\rD)$ consisting of cohomology classes of cocycles $\xi=(\begin{smallmatrix}a & b \\ c & d	\end{smallmatrix})$ unramified outside $p$ and such that
	\begin{enumerate}
		\item[(i)] $a+b=c+d$ on $G_{\Qp}$, and
		\item[(ii)] $b=d$ on $I_p$.
	\end{enumerate}	
	\item The isomorphism in (1) induces by restriction an isomorphism $\Hom_\cO(C^1(\pi_1),\rD) \simeq \HH^1_\ord(\Q,\ad^0 \rho_1 \otimes_{\cO}\rD)=\HH^1_\ord(\Q,\ad \rho_1 \otimes_{\cO}\rD)^{\tr=0}$. 
	\end{enumerate}
\end{lemma}
\begin{proof}
	We first note that, as $\ad \rho_1$ is trivial on $G_{\Qp}$, the conditions on $\xi$ at $p$ do not depend  on the representative $\xi$. 
	
	Write $\rD$ as $\cup_r k_r$ with $k_r=p^{-r}\cO/\cO$. The natural $\cO$-linear injection from $\Hom_\cO(C^1(\pi),\rD)$ to $\HH^1(\Q,\ad \rho_1 \otimes_{\cO}\rD)$ is described in \cite[\S28]{mazurdeforming}; its image consists of cohomology classes $[\xi]$ such that 
	\[\rho':=(1+\varepsilon\xi)\cdot \begin{pmatrix}\chi & 0 \\ 0 & 1	\end{pmatrix}, \]
	together with the line spanned by $(\begin{smallmatrix}1 \\ 1	\end{smallmatrix})$, defines a point of $\cD^\ord(\cO \oplus \varepsilon k_r)$ for some $r$. Here, the multiplication in $\cO \oplus \varepsilon k_r$ is given by the formula $(s+\varepsilon s')(t+\varepsilon t')=st+\varepsilon(s't+st')$. 
	
	It is clear from the definition that $\xi$ is unramified outside $Np$, and one easily checks that the minimality condition at $\ell|N$ for $\rho'$ is equivalent to $\xi_{|I_\ell}$ being a coboundary, \textit{i.e.}, $[\xi]$ is unramified at $\ell$.
	
	Write $\xi$ as $(\begin{smallmatrix}a & b \\ c & d	\end{smallmatrix})$. Since $\chi_{|G_{\Qp}}=1$, the condition satisfied by $\rho'$ at $p$ is
	\begin{equation}\label{eq:conjugation_matrix}
		\begin{array}{rl}
		\begin{pmatrix} 1 & 0 \\ -1 & 1	\end{pmatrix} \left(1+\varepsilon \begin{pmatrix} a & b \\ c & d	\end{pmatrix}\right) \begin{pmatrix} 1 & 0 \\ 1 & 1	\end{pmatrix}&=1+\varepsilon\begin{pmatrix} a+b & b \\ c+d-a-b & d-b \end{pmatrix}   \\ & = \begin{pmatrix} * & * \\ 0 & \mathbf{unr}	\end{pmatrix} \quad \textrm{on $G_{\Qp}$},
	\end{array} 
	\end{equation}
	where $\mathbf{unr}$ is an unramified character. This proves (1).
	
	The point (2) follows from (1). Indeed, the additional condition that $\det(\rho')=\chi$ in weight $1$ shows that $1=\det(1+\varepsilon\xi)=1+\varepsilon\tr(\xi)$, so it is equivalent to $\tr(\xi)=0$.
\end{proof}

We now combine Lemma \ref{lem:description_C1_in_weight_one} with the preliminary calculations of \S\S\ref{sec:GS_units}--\ref{sec:galois_cohomology} to determine much of the structure of $C^1(\pi_1)$.

\begin{proposition}\label{prop:length_C1_R}
	With the notation of Lemma \ref{lem:description_C1_in_weight_one}, there is a short exact sequence of $\cO$-modules
	\[ 0 \to A'_K[\chi] \times A'_K[\chi^{-1}] \to \Hom_\cO(C^1(\pi_1),\rD)  \to \cO/(\alpha) \times \cO/(\beta) \to 0,  \]
	where 
\begin{equation}\label{eq:alpha}
		\alpha=\gcd(\lambda(u_{\chi}),\lambda(u_{\chi^{-1}}),\ord_p(u_{\chi}),\ord_p(u_{\chi^{-1}})), 
\end{equation}
	and 
	\begin{equation}\label{eq:beta}
		\beta=\frac{(\cL_{\chi}+\cL_{\chi^{-1}})\cdot \ord_p(u_{\chi})\cdot \ord_p(u_{\chi^{-1}})}{\alpha}. 
	\end{equation}
	In particular, $C^1(\pi_1)$ has a finite length given by the $p$-adic valuation of $(\cL_{\chi}+\cL_{\chi^{-1}})L(\chi,0)L(\chi^{-1},0)$.
\end{proposition}
\begin{proof}
	By Lemma \ref{lem:description_C1_in_weight_one}(2), $\Hom_\cO(C^1(\pi_1),\rD)\simeq \HH^1_\ord(\Q,\ad^0 \rho_1 \otimes_{\cO}\rD)$, and we describe the latter module.
	First note that, as $\cO[G_{\Q}]$-modules and in matrix notation, 
	\[\ad \rho_1 \otimes_{\cO} \rD = \begin{pmatrix} \rD & \rD(\chi) \\ \rD(\chi^{-1}) & \rD	\end{pmatrix}. \]
	Moreover, elements $a\in \HH^1(\Q,\rD)=\Hom_{\textrm{cts}}(G_{\Q},\rD)$ unramified outside $p$ induce homomorphisms from $\Gamma\simeq 1+p\Zp$ to $\rD$, hence they are of the form $a=s\lambda$ for some $s\in \rD$. 
	
	In light of \eqref{eq:description_H1(Q,chi)}, $\HH^1_\ord(\Q,\ad^0 \rho_1 \otimes_{\cO}\rD)$ is identified with quadruples 
\begin{equation}\label{eq:xi}
		 \xi= \begin{pmatrix}
		a & b \\ c & d
	\end{pmatrix}=
	\begin{pmatrix}
		s\lambda & b \\ c &-s\lambda
	\end{pmatrix},
\end{equation}
	where $s\in\rD$, $b\in\Hom_{\cO}(\cX_K[\chi],\rD)$, $c\in\Hom_{\cO}(\cX_K[\chi^{-1}],\rD)$, such that
	$a+b=c+d$ on $G_{\Qp}$ and $b=d$ on $I_p$. In particular, $a$ and $d$ are determined by $b$, and $\xi$ by $b$ and $c$.
	
	We now make use of Lemma \ref{lem:description_Xvarphi} and its notation. Write $x=x_\chi$, $x'=x_{\chi^{-1}}$, and $y=y_\chi$, $y'=y_{\chi^{-1}}$ for short. Denote also by $\gamma$ and $\phi$ the distinguished elements seen in $\cX=\cX_K[\chi]$, and denote them by $\gamma',\phi'$ when seen in $\cX'=\cX_K[\chi^{-1}]$. In particular, $\gamma,\phi$ generate the $\cO$-submodule $\cZ=(\cO \times \cO)/\langle(x,y)\rangle$ of $\cX$, and likewise for $\cZ'\subseteq\cX'$ and $\gamma',\phi'$. Since $\lambda(\gamma)=1$ and $\lambda(\phi)=0$, the conditions at $p$ satisfied by $\xi$ are
	\[b(\gamma)=-c(\gamma)=-s, \quad b(\phi)=c(\phi). \]
	
	We now apply the (exact) Pontryagin functor $M^\vee=\Hom_{\cO}(M,\rD)$ to the exact sequence of Lemma \ref{lem:description_Xvarphi} for $\varphi=\chi$ and $\varphi=\chi^{-1}$. The map $\xi \mapsto (b,c)$ injects $\HH:=\HH^1_\ord(\Q,\ad^0 \rho_1 \otimes_{\cO}\rD)$ into $\cX^\vee \times \cX^{'\vee}$, and its image contains $A'_K[\chi]^\vee \times A'_K[\chi^{-1}]^\vee$, which is isomorphic to $A'_K[\chi] \times A'_K[\chi^{-1}]$ since $M\simeq M^\vee$ if $M$ is finite. 
	
	To determine the image of $\HH$ in $\cZ^\vee \times \cZ^{'\vee}$, we first remark that $\cZ^\vee=\{(u,v) \in \rD^2 \; : \; xu+yv=0\}$, and likewise for $\cZ^{'\vee}$. Hence, the image of $\HH$ consists of quadruples $(b(\gamma),c(\gamma),b(\phi),c(\phi))=(u,u',v,v')\in \rD^4$ such that $xu+yv=x'u'+y'v'=0$, $u=-u'$, and $v=v'$. In other words, it is the $\cO$-submodule of $\rD^2$ of all pairs $(u,v)$ such that
	\[ \begin{pmatrix}
		x & y \\ -x' & y'
	\end{pmatrix}\begin{pmatrix}
	u \\ v
	\end{pmatrix}= \begin{pmatrix}
	0 \\ 0
	\end{pmatrix} \qquad \textrm{in $\rD^2$.} \]
	The square matrix has coefficients in $\cO$, which is a PID, so we can solve the equation by writing it in its Smith normal form, which is given by $(\begin{smallmatrix} \gcd(x,x',y,y') & 0 \\ 0 & (xy'+x'y)/\gcd(x,x',y,y') \end{smallmatrix})=(\begin{smallmatrix} \alpha & 0 \\ 0 & -\beta \end{smallmatrix})$. The set of solutions $(u,v)$ is therefore isomorphic to $\rD[\alpha] \times \rD[\beta]\simeq \cO/(\alpha) \times \cO/(\beta)$, which establishes the short exact sequence.
	
	The last statement follows by combining the first one, the non-vanishing of $\cL_{\chi}+\cL_{\chi^{-1}}$ (which follows from Lemma \ref{lem:linear_independence_L_invariants}), Lemma \ref{lem:calcul_ord} and Theorem \ref{thm:refined_CNF}. 
\end{proof}

\subsection{Split and CM representations}\label{sec:split_CM}
We introduce and study in this paragraph a subfunctor of $\cD^\ord$ parametrizing split representations. The definitions and results of this paragraph will only be relevant for \S\ref{sec:CG}.

Let $\cO'\in\mathcal{CNL}_\cO$ be a finite totally ramified extension of $\cO$, let $\mu\in\cO'$ and put
\[
P_\mu=\begin{pmatrix} 1 & 0 \\ 1 & 1 \end{pmatrix}
\begin{pmatrix}1 & -\mu \\ 0& 2\end{pmatrix}
= \begin{pmatrix}1 & -\mu \\ 1 & 2-\mu \end{pmatrix}
\in \GL_2(\cO').
\]
The reader may assume $\cO'=\cO$ in what follows, but allowing $\cO'$ to be arbitrarily large will later be useful in this paper.
\begin{definition}\label{def:split}
	We let $\cD_{\cO'}^{\mu-\spl}$ be the subfunctor of $\cD^\ord \otimes_{\cO}\cO'$ defined as follows. Let $A\in \mathcal{CNL}_\cO$ be an $\cO'$-algebra and let $(\rho,L)\in \cD^\ord(A)$, with $\rho$ computed in the Eisenstein basis of Definition \ref{def:eisenstein_basis}. Then $(\rho,L)$ defines an $A$-point of $\cD_{\cO'}^{\mu-\spl}$ if $P^{-1}_\mu \rho_{|G_{\Qp}}P_\mu$ is of the form $(\begin{smallmatrix} \eta_1& 0 \\ 0 &\eta_2 \end{smallmatrix})$, where $\eta_i$ are $A$-valued characters of $G_{\Qp}$.
\end{definition}
In the above definition, we note that $\eta_2$ is unramified. More precisely, letting $\rho^\ord=(\begin{smallmatrix}	1 & 0 \\ -1 & 1\end{smallmatrix})\cdot \rho \cdot (\begin{smallmatrix}	1 & 0 \\ 1 & 1\end{smallmatrix})$, one then has $\rho^\ord_{|G_{\Qp}}=(\begin{smallmatrix}\eta_1 & \kappa \\ 0 & \eta_2 \end{smallmatrix})$ with $\eta_2$ unramified, and the $\mu$-split condition reads $2\kappa=\mu(\eta_1-\eta_2)$ on $G_{\Qp}$. 
We also remark that, for any $\rho\in\cD^{\ord}(A)$ with $A$ a quotient of $\cO'$, if $\rho_{|G_{\Qp}}$ is split, then there exists $\mu\in\cO'$ such that $\rho$ is $\mu$-split. 

If $\rho$ is taken to be the universal ordinary deformation, we see that the quotient
\[R_{\cO'}^{\mu-\spl}:= (R^\ord\otimes_{\cO}\cO')/(2\kappa(\gamma)-\mu(\eta_1-\eta_2)(\gamma),2\kappa(\phi)-\mu(\eta_1-\eta_2)(\phi))\]
represents $\cD_{\cO'}^{\mu-\spl}$. By slight abuse of notation, we simply write $\cD^{\mu-\spl}$ and $R^{\mu-\spl}$ for $\cD_{\cO'}^{\mu-\spl}$ and $R_{\cO'}^{\mu-\spl}$ when the choice of $\cO'$ is clear from the context.

As in \S\ref{sec:ordinary_deformations}, we denote by $\cD^{\mu-\spl}_k$ and $R^{\mu-\spl}_k$ the weight $k\in\Z$ specializations of $\cD^{\mu-\spl}$ and $R^{\mu-\spl}$ respectively.

We note that, if we let $\rho^{\spl}=P^{-1}_\mu\cdot  \rho \cdot P_\mu$, then 
\begin{equation}\label{eq:conjugation_tau_split}
	\rho^{\spl}(\tau)=\begin{pmatrix} \mu-1 & \mu(2-\mu) \\ 1 & 1-\mu\end{pmatrix}.
\end{equation}
In particular, if $\mu=1$, then $\rho^{\spl}(\tau)=(\begin{smallmatrix}0 & 1 \\ 1 & 0\end{smallmatrix})$. As we shall see in Proposition \ref{prop:modularity_1_spl}, Galois representations attached to CM forms are $1$-split, and our normalization ensures that the ``splitting basis'' and the dihedral basis match in the CM case.

An example of a $\mu$-split representation of weight $k=1$ is given by $\rho_1$ defined in \eqref{eq:rho_f}. We describe its congruence module in the next proposition.

\begin{proposition}\label{prop:tangent_space_split}
	Let $\pi_1 \colon R^{\mu-\spl}_1 \twoheadrightarrow \cO'$ be the map induced by $\rho_1$. Let $A'_{\cO'}=A'_K \otimes_{\cO} \cO'$ and $\rD'=\rD \otimes_{\cO} \cO'$. There is a short exact sequence of $\cO'$-modules
	\[ 0 \to A'_{\cO'}[\chi] \times A'_{\cO'}[\chi^{-1}] \to \Hom_{\cO'}(C^1(\pi_1,R^{\mu-\spl}_1),\rD')  \to \cO'/(\delta) \to 0,  \]
	where 
	\[\delta=\gcd(\mu-1,\ord_p(u_{\chi}),\ord_p(u_{\chi^{-1}})).\]
\end{proposition}
\begin{proof}
	To make the notation less heavy, we will assume $\cO'=\cO$. The proof in the general case is nearly identical.
	
	Since $\cD^{\mu-\spl}_1$ is a subfunctor of $\cD^\ord_1$, there is a natural inclusion $\HH':=\Hom_\cO(C^1(\pi_1,R^{\mu-\spl}_1),\rD) \subseteq \HH:=\Hom_\cO(C^1(\pi_1,R^{\ord}_1),\rD)$ of tangent spaces. We rely on the arguments of the proof of Proposition \ref{prop:length_C1_R}, and we freely use the notation introduced there. 
	
	Recall that $\HH$ consists of cocycles $\xi=(\begin{smallmatrix}s\lambda & b \\ c & -s\lambda\end{smallmatrix})$ as in \eqref{eq:xi} with $s\in\rD$, $b\in\cX^{\vee}$, $c\in\cX^{',\vee}$, such that $c-b=2s\lambda$ on $G_{\Qp}$ and $b=-s\lambda$ on $I_p$. Using the matrix identity \eqref{eq:conjugation_matrix}, the condition that $\xi$ belongs to $\HH'$ is equivalent to the local condition $b=\mu(s\lambda+b)$ on $G_{\Qp}$. Restricting this condition to $I_p$ gives that $s=0$, and evaluating at $\phi$ gives $(\mu-1)b(\phi)=0$. 
	
	We have proven so far that $\HH'$ consists of cocycles $\xi=(\begin{smallmatrix}0 & b \\ c & 0 \end{smallmatrix})$, $(b,c)\in \cX^{\vee}\times \cX^{',\vee}$, such that $b,c$ are unramified, $b(\phi)=c(\phi)$ and $(\mu-1)b(\phi)=0$. In particular, $\HH'$ contains $A':=A'_K[\chi]^\vee \times A'_K[\chi^{-1}]^\vee$. Moreover, $A'$ is the kernel of the map 
	\[\HH'\to \cZ^\vee \times\cZ^{'\vee}=\{(u,v,u',v') \in \rD^4 \; : \; xu+yv=x'u'+y'v'=0\}\]
	sending $\xi$ to $(b(\gamma),b(\phi),c(\gamma),c(\phi))=(0,b(\phi),0,b(\phi))$. The image of this map is clearly isomorphic to $\rD[\delta]$, where $\delta=\gcd(\mu-1,y,y')$. This proves that $\HH'/A'\simeq \rD[\delta]\simeq \cO/(\delta)$.
\end{proof}

We conclude this paragraph with the following simple lemma. 

\begin{lemma}\label{lem:lift_characters}
	Assume $\chi$ is quadratic (so $\cO=\Zp$) and $A'_K=\{0\}$. Call $\gp_K$ the distinguished prime of $K$ above $p$, and let $\ob{\gp}_K$ be its complex conjugate. Denote by $\cX_K'=\cX_K/\langle \gamma\rangle$ the quotient of $\cX_K$ which is the Galois group of the maximal abelian pro-$p$ extension of $K$ unramified outside $\ob{\gp}_K$. Let $x=\lambda(u_\chi)$ and $y=\ord_p(u_\chi)$, and put $\ob{\gamma}=\tau\gamma\tau\in\cX_K$ and $\ob{\phi}=\tau\phi\tau\in\cX_K$.
	
	\begin{enumerate}
		\item $\cX_K$ is a $\Zp$-module of rank $2$ generated by $\gamma,\phi,\ob{\gamma},\ob{\phi}$ with relations $\gamma^x\ob{\gamma}^{-x}\phi^y\ob{\phi}^{-y}=1$ and $\phi\ob{\phi}=1$. 
		Alternatively, it is generated by $\gamma,\phi,\ob{\gamma}$, with unique relation $\gamma^x\ob{\gamma}^{-x}\phi^{2y}=1$.
		\item $\cX'_K$  is a $\Zp$-module of rank $1$ generated by $\ob{\gamma},\phi$ with unique relation $\ob{\gamma}^{x}=\phi^{2y}$.
	\end{enumerate}
\end{lemma}
\begin{proof}
Both points are proved by explicitly describing $\cX_K$ and $\cX_K'$.

Note that $U_K'$ is generated over $\Zp$ by $u_\chi$ and $p$. Arguing as in the proof of Lemma \ref{lem:description_Xvarphi} (based on the second row of \eqref{eq:diagramm_CFT}), $\cX_K$ is isomorphic to $\Z_p^4/\langle(x,y,-x,-y),(0,1,0,1)\rangle$. Here, the standard basis of $\Z_p^4$ is sent to the quadruple $(\gamma,\phi,\ob{\gamma},\ob{\phi})$. From this, one can eliminate $\ob{\phi}$, and one sees that $\cX_K'=\cX_K/\langle \gamma\rangle$ is generated by $\ob{\gamma}$ and $\phi$ over $\Zp$, with unique relation $\phi^{2y}=\ob{\gamma}^{x}$. This ends the proof of the lemma.
\end{proof}

\subsection{Weak ordinarity}
We establish in this paragraph a useful descent lemma comparing the ordinarity condition in the definition of $\cD^\ord$ with a weak notion of ordinarity.

\begin{lemma}\label{lem:weak_ordinarity}
	Assume $A'_K[\chi]=A'_K[\chi^{-1}]=\{0\}$ and let $A\in \mathcal{CNL}_\cO$. Assume $A$ is reduced with total ring of fractions $\cK=\prod_j K_j$. Let $\rho \colon G_\Q \to \GL_2(\cK)$ be a Galois representation unramified outside $Np$ and satisfying the minimality condition of Definition \ref{def:ord} at all $\ell|N$. Suppose furthermore that 
	\begin{enumerate}
		\item the pushforward of $\rho$ to any $K_j$ is irreducible,
		\item $\tr(\rho)$ takes values in $A$ and $\tr(\rho) \equiv \chi+1$ mod $\gm_A$, and
		\item there exists a $\cK$-line $L$ of $V:=\cK^2$ stabilized by the $G_{\Qp}$-action given by $\rho$, and the character $\delta \colon G_{\Qp} \to \cK^\times$ giving the action on $V/L$ is unramified and takes values in $1+\gm_A$.
	\end{enumerate}
	Then  the pair $(\rho,L)$ may be conjugated so that it descends to an $A$-point of $\cD^\ord$, and $\rho\equiv (\begin{smallmatrix}\chi& 0 \\ 0 &1 \end{smallmatrix})$ mod $\gm_A$ on $G_\Q$.
\end{lemma}
\begin{proof}
	Recall that $\chi(\tau)\not\equiv1$ mod $\gm_A$ since $p$ is odd. Hence, conjugating $\rho$ if necessary, we may assume that $\rho(\tau) = (\begin{smallmatrix}	-1 & 0 \\ 0 & 1	\end{smallmatrix})$.  
		
	Let $(\begin{smallmatrix}a&b \\ c&d\end{smallmatrix})$ be the entries of $\rho$. By \cite[Lemme 1]{bellaichechenevierlissite}, the $A$-submodules of $\cK$ generated by $a(G_\Q)$ and by $d(G_\Q)$ are both equal to $A$, and we have the congruences $a \equiv \chi$ mod $\gm_A$ and $d \equiv 1$ mod $\gm_A$. Furthermore, the $A$-submodules $B$ and $C$ generated by $b(G_\Q)$ and $c(G_\Q)$ respectively are fractional ideals, and $BC\subseteq \gm_A$.
	
	We claim that the ordinary line $L \subset V$ is in generic position with respect to the canonical basis $(e_1,e_2)$ of $V$, \textit{i.e.}, it is generated by $(\begin{smallmatrix} 1 \\ s  \end{smallmatrix})$ for some $s\in \cK^\times$. We may and we shall assume $\cK$ is a field.
	
	Assume $s=0$, \textit{i.e.}, $L=\cK e_1$, so that we have $c_{|G_{\Qp}}=0$.  Let $\ob{C}=C/\gm_A C$, and consider the map $\ob{\kappa}=\ob{\chi}^{-1}\ob{c} \colon G_\Q \to \ob{C}$. The multiplicativity of $\rho$ and the above description of $a$ and $d$ modulo $\gm_A$ imply that it satisfies the cocycle condition with respect to $\ob{\chi}^{-1}$. Moreover, $\ob{\kappa}_{|G_K}$ is unramified outside $p$ because of the minimality condition, which indeed forces $\rho_{|G_K}$ to be unramified outside $p$. The same property holds for the class of $\ob{\kappa}$, as $p\nmid [K:\Q]$. Therefore, it induces a map $\ob{\kappa} \colon \cX_K[\chi^{-1}] \to \ob{C}$ (see \eqref{eq:description_H1(Q,chi)}). Since it is also trivial on $G_{\Qp}$ by assumption, it factors through $A'_K[\chi^{-1}]$ by Lemma \ref{lem:description_Xvarphi}, which is $\{0\}$, also by assumption. But the proof of \cite[Lemme 3]{bellaichechenevierlissite} shows that the image of $\ob{\kappa}_{|G_K}$ spans $\ob{C}$, so $\ob{C}=0$, and $C=0$ by Nakayama's lemma. This contradicts the assumption that $\rho$ is irreducible.
	
	Likewise, the fact that $A'_K[\chi]=\{0\}$ shows that $L$ is not generated by $e_2$. Hence, rescaling $e_2$ if necessary, we may assume $L$ is generated by $(\begin{smallmatrix} 1 \\1  \end{smallmatrix})$. A similar matrix calculation as in \eqref{eq:conjugation_matrix} thus shows that $d-b =\delta$ and $c+d-a-b=0$ on $G_{\Qp}$. As $\delta$ takes values in $1+\gm_A$, the first equation shows that $b(G_{\Qp}) \subseteq \gm_A$, and the second one that $c(G_{\Qp}) \subseteq \gm_A$ as well. 
	
	Let $C_p$ be the $A$-submodule of $C$ generated by $c(G_{\Qp})$. An argument similar to the one above, applied to $\ob{C}=C/(C_p+\gm_A C)$, proves that $C_p=C$, and hence, $C\subseteq \gm_A$. Likewise, $B\subseteq \gm_A$. This shows that $\rho$ is indeed $A$-valued, $\rho\equiv (\begin{smallmatrix}\chi & 0 \\ 0 &1 \end{smallmatrix})$ mod $\gm_A$, and $(\rho,\langle(\begin{smallmatrix} 1 \\1  \end{smallmatrix})\rangle_A)\in\cD^\ord(A)$.
	\end{proof}
	
We also include the following lemma which describes some of the reducible points of $\cD^\ord$.
	
\begin{lemma}\label{lem:reducible_ordinary_points}
	Let $A\in \mathcal{CNL}_\cO$. Assume that $A$ is $\cO$-flat and reduced. For any point $(\rho,L)\in\cD^\ord(A)$ with $\rho$ computed in the Eisenstein basis, if $\tr(\rho)=1+\chi$, then 
\[	\rho=\begin{pmatrix}
		\chi & 0 \\ 0 & 1
	\end{pmatrix}.\]
\end{lemma}
\begin{proof}
	Assume, without loss of generality, that $A$ is a domain with fraction field $\cK$. By the Brauer--Nesbitt theorem, $\rho$ is reducible over $\cK$, so there is a $G_{\Q}$-stable line $L'\subset V_\rho \otimes_{A} \cK$. In particular, it must be stable under $\rho(\tau)=(\begin{smallmatrix}-1 & 0 \\ 0 & 1	\end{smallmatrix})$, so $L'=\langle(\begin{smallmatrix}1 \\ 0	\end{smallmatrix})\rangle_\cK$ or $L'=\langle(\begin{smallmatrix}0 \\ 1	\end{smallmatrix})\rangle_\cK$. Therefore, there is a cocycle $\kappa \colon G_{\Q} \to A$ such that either $\rho=(\begin{smallmatrix}\chi & \kappa \\ 0 & 1 	\end{smallmatrix})$ or $\rho=(\begin{smallmatrix}	\chi & 0 \\ \chi\kappa & 1 	\end{smallmatrix})$. 
	
	Assume we are in the first case, the second case being similar. The fact that $\langle(\begin{smallmatrix}1 \\ 1	\end{smallmatrix})\rangle_A$ is $G_{\Qp}$-stable implies that $\kappa_{|G_{\Qp}}=0$. As in the proof of Lemma \ref{lem:weak_ordinarity}, the continuous cocycle class $[\kappa]\in\HH^1(\Q,A(\chi))$ is unramified outside $p$, so its restriction to $G_K$ induces a map $A'_K[\chi]\to A$, which is necessarily trivial as $A$ is $\cO$-flat. Therefore, $\kappa$ is a coboundary of the form $(\chi-1)\cdot a$ for some $a\in A$. Since $\kappa(\tau)=0$, one sees that $a=0$, and thus, $\kappa=0$. 
\end{proof}

\section{Structure of Hecke algebras and modularity}
\subsection{The geometry of the cuspidal eigencurve in Eisenstein weight $1$}\label{sec:Hecke_algebras_and_eigencurve}
We first review the work of Betina--Dimitrov--Pozzi \cite{betinadimitrovpozzi} and we translate some of its main results into the language of Hida theory.

Let $\kappa \colon \cC \to \cW$ be the \emph{cuspidal} eigencurve of tame level $N$ endowed with its weight map. It is a rigid analytic space over $\Qp$ parametrizing systems of Hecke eigenvalues corresponding to overconvergent cuspidal eigenforms of finite slope (see \cite[\S3.1]{betinadimitrovpozzi}).

As explained in \S3.3 of \textit{loc. cit.}, the $p$-stabilized weight $1$ Eisenstein series 
\[ f(q)=E_1^{(p)}(1,\chi)= \sum_{n\geq 1} \sum_{\substack{d|n \\ p\nmid d}}\chi(d) q^n \]
gives rise to an $E$-point $\mathrm{x}$ of $\cC$. As $f$ has $U_p$-eigenvalue $1$, it lies, in fact, in the ordinary locus $\cC^\ord$ of $\cC$, defined by the open and closed condition $|U_p|=1$. 

\begin{theorem}\label{thm:BDP}
	The weight map $\kappa$ is étale at $\mathrm{x}$. In other words, it induces an isomorphism between the completed local rings of $\cW$ at $\kappa(\mathrm{x})$ and of $\cC^\ord$ at $\mathrm{x}$.
\end{theorem}
\begin{proof}
	This is \cite[Theorem A(i)]{betinadimitrovpozzi}.
\end{proof}

As explained in \cite[\S6]{pilloni2013formes}, the map $\kappa_{|\cC^\ord}\colon \cC^\ord \to \cW$ admits a canonical integral model afforded by Hida theory. More precisely, it is the rigid analytic generic fiber of 
\[\operatorname{Spf}(\T(N)) \to \operatorname{Spf}(\Lambda(N)),\]
where $\Lambda(N)=\Zp\lsem \tilde{\Gamma}\rsem$, $\tilde{\Gamma}=(1+p\Zp) \times  (\Z/Np\Z)^\times$, is a ring of diamond operators $\langle a\rangle$, $a\in \tilde{\Gamma}$, acting faithfully on the space $\bS(N)$ of cuspidal Hida families of tame level $N$, and $\T(N)\subseteq \End_{\Lambda(N)}(\bS(N))$ is the $\Lambda(N)$-adic Hecke algebra generated by $T_\ell$ (for $\ell\nmid Np$) and $U_q$ (for $q|Np$).
We normalize the action of diamond operators so that $\langle \gamma\rangle$ acts on weight $k$ specializations of Hida families as $\gamma^{k-1}$ for $k\in\Z$.

Let
\begin{equation}\label{eq:maximal_ideal_eisenstein}
	\gm=\langle p,\ T_\ell-(1+\chi(\ell)),\ U_q-1,\ \langle a \rangle - \chi(a) \ \colon \ \ell\nmid Np,\ q|Np,\ a\in \tilde{\Gamma} \rangle, 
\end{equation}
be the maximal ideal of $\T(N) \otimes_\Zp \cO$ attached to $f$ mod $p$, and let
\begin{equation}\label{eq:definition_Hida_Hecke_algebra}
	\T = (\T(N) \otimes_\Zp \cO)_\gm. 
\end{equation}
Then $\T$ acts faithfully on $\bS=(\bS(N)\otimes_{\Zp}\cO)_\gm$, and Hida theory implies that $\T$ is finite flat over $\Lambda=\cO\lsem 1+p\Z_p\rsem$. Also note that eigenforms in $\bS$ have tame nebentype $\chi$ because of our assumption $p\nmid \phi(N)$.

One can check that $\T$ is a reduced ring. Indeed, since $N$ is the Artin conductor of $\ob{\chi}$, we may identify $\T$ with the localization at a certain maximal ideal of the sub-algebra $\T'(N)$ of $\T(N) \otimes \cO$ generated by all Hecke operators excluding $U_q$ for $q| N$ (see \cite[Proposition 2.4.2]{emerton2006variation} with $\Sigma=\emptyset$). Since $\T'(N)$ is reduced, so is $\T$.

Recall that $\Lambda\simeq \cO\lsem X \rsem$ via $[\gamma] \mapsto 1+X$. For any $k\in\Z$,  the finite flat $\cO$-module $\bS_k:=\bS/(1+X-(1+p)^{k-1})\bS$ coincides with the module $S_k^\ord(N,\cO)_\gm$ of weight $k$ ordinary $p$-adic cusp forms with Fourier coefficients in $\cO$, localized at $\gm$. Of course, if $k\geq2$, this module entirely consists of classical cusp forms by Hida's control theorem, but in general, we only know that its elements are overconvergent (again by \cite{pilloni2013formes}). For instance, the $q$-expansion of $f\in \bS_1$ is not the one of a classical cusp form of weight $1$.

Let $\T_k=\T/(1+X-(1+p)^{k-1})\T$; it is a finite flat local $\cO$-algebra. By Hida--Hecke duality for cuspidal Hida families, 
\[\T \simeq \Hom_\Lambda(\bS,\Lambda), \qquad T \mapsto (F\mapsto a_1(TF))\]
as $\T$-modules. We deduce that $\T_k\simeq \Hom_\cO(\bS_k,\cO)$ for all $k\in\Z$ as $\bS$ and $\T$ are $\Lambda$-free.

\begin{proposition}\label{prop:finite_length_C}\
\begin{enumerate}
	\item Let $\gp \subseteq \T$ be the kernel of the system of Hecke eigenvalues of $f$. Then $\T_\gp$ is a discrete valuation ring, hence a domain. In particular, there exists a unique minimal prime ideal $\ga\subseteq \T$ contained in $\gp$. 
	\item Let $\pi_1\colon \T_1 \to \cO$ be the system of Hecke eigenvalues of $f$. Then $C^0(\pi_1)$ is of finite length over $\cO$.
\end{enumerate}
\end{proposition}
\begin{proof}
	By Theorem \ref{thm:BDP}, $\T_\gp$ is étale over the localization $\Lambda_{(X)}$ of $\Lambda$ at $(X)$, which is a DVR. Therefore, $\T_\gp$ is itself a DVR, hence proving (1).
	
	Consider (2). Modding out by $X$ shows that $\T_{1}$ is étale over $\cO$ at the prime  $\gp_1=\ker\pi_1$. Lemma \ref{lem:congruence_module_facts}(1) then implies that $C^0(\pi_1)$ is indeed of finite length.
\end{proof}

\begin{definition}\label{def:congruence_number_of_f}
	Any congruence number $c^0_f$ attached to the system $\pi_1 \colon \T_1 \to \cO$ of Hecke eigenvalues of $f$ will be called a (cuspidal) congruence number of $f$.
\end{definition}

We refer to Definition \ref{def:C} for a definition of a congruence number attached to $\pi_1$; it is necessarily nonzero by Proposition \ref{prop:finite_length_C}(2).

\subsection{Congruences between modular forms}\label{sec:construction_congruences}
We study in this paragraph the length of the congruence module $C^0(\pi_1)$ for the system of Hecke eigenvalues $\pi_1 \colon \T_1 \to \cO$ attached to $f$, which is finite by Proposition \ref{prop:finite_length_C}(2).

For $a,b\in E^\times$, let $\lcm(a,b)$ be $a$ if $a/b\in\cO$, and $b$ otherwise. Also let $\den(a)=a^{-1}$ if $a\not\in\cO$, and $\den(a)=1$ otherwise. This paragraph is devoted to the proof of the following theorem, which will be separated according to whether $\chi$ is quadratic or not.

\begin{theorem}\label{thm:congruence_bound}
	Let $\pi_1 \colon \T_1 \to \cO$ be the system of Hecke eigenvalues of $f$. 
	\begin{enumerate}
		\item If $\chi$ is quadratic, then $\length C^0(\pi_1)=\length \cO/(\cL_\chi \cdot L(\chi,0)^2)$.
		\item In general, $\length C^0(\pi_1)\geq\length \cO/(c)$, where 
		\begin{equation}\label{eq:congruence_number_lower_bound}
			c=(\cL_\chi+\cL_{\chi^{-1}}) \cdot \lcm\left(\den(\cL_{\chi^{-1}}) \cdot L(\chi,0),\den(\cL_{\chi}) \cdot L(\chi^{-1},0)\right)
		\end{equation}
	\end{enumerate}
\end{theorem}
We remark that, if $\varphi\in\{\chi,\chi^{-1}\}$, then $\cL_\varphi \cdot L(\varphi,0)\in \cO$ by Lemma \ref{lem:calcul_ord} and Theorem \ref{thm:refined_CNF}. Hence, $c$ is indeed in $\cO$.

The following notation is in force. We let $\gp$ be the prime of $\T$ attached to $f$, and $\ga\subseteq \T$ the unique minimal prime contained in $\gp$, as in Proposition \ref{prop:finite_length_C}(1). Let $\I$ be the integral closure of $\T/\ga$ in its fraction field $\K$. We think of $\I$ as a coefficient ring, and we consider the finite flat $\I$-algebra $\T_\I=\T \otimes_{\Lambda} \I$. The natural map $\T \twoheadrightarrow \T/\ga$ induces an $\I$-algebra homomorphism $\pi \colon \T_\I \twoheadrightarrow \I$ which is an $\I$-retraction. We let $\cF\in\bS_\I=\bS\otimes_{\Lambda}\I$ be the $\I$-adic primitive Hida family corresponding to $\pi$.

The system of Hecke eigenvalues $\T \twoheadrightarrow \cO$ of $f$ gives rise to a map $\spe_f \colon \I \twoheadrightarrow \cO$ which we call specialization at $f$. When applied to the $q$-expansion of $\cF$, one gets back the $q$-expansion of $f$, \textit{i.e.}, $\cF$ passes through $f$.

We first treat the quadratic case, which is based on the explicit computation of the congruence ideal of CM Hida families in terms of Katz's $p$-adic $L$-function. Indeed, assuming $\chi^2=1$, $f$ is a CM form, and $\cF$ is a CM Hida family. Note that, in this case, $\T/\ga$ is already integrally closed, since it is a power series ring, but we will not use this fact (see \cite[p. 235]{hidaLFE} and \S\ref{sec:CM} for a detailed description of CM components of big Hecke algebras).

\begin{proof}[Proof of Theorem \ref{thm:congruence_bound}(1)]
	Assume $\chi^2=1$, \textit{i.e.}, $K$ is imaginary quadratic. Let $\gp=\ker \pi$, and define the congruence ideal $\eta=\pi(\Ann_{\T_\I}(\gp))\subseteq \I$ attached to $\cF$. We first compare $\eta$ with the alternative definition given in \cite[\S4.3.1]{burungale2024zetaelementsellipticcurves}.
	
	Let $\T_\K=\T_\I \otimes_{\I}\K$ and $\bS_\K=\bS_\I \otimes_{\I}\K$. Since $\cF$ is primitive of level $N$, we have a decomposition $\T_\K=\K \times Y$ of $\K$-algebras, where the projection onto the first factor is induced by $\pi$ (see \cite[Theorem 4.2]{hidaAIF}). Letting $e=(1,0)\in \T_\K$ be the idempotent of $\cF$, we have $e(\bS_\K)=\K \cF$ and $\bS_\I \cap e(\bS_\I)=\I \cF$. Arguing as in the proof of Lemma \ref{lem:congruence_module_facts}, we see that $e(\bS_\I)/\I \cF \simeq e(\T_\I) / (e(\T_\I) \cap \T_\I)\simeq \I/\eta$. In particular, the unique $\I$-linear map $\bS_\I \to \K$ factoring through $e(\bS_\I)$ and sending $\cF$ to $1$ has image $\eta^{-1}$, so $\eta$ coincides with what is denoted by $I^\textrm{cusp}_{\mathbf{h}_v}$ in \cite[\S4.3.1]{burungale2024zetaelementsellipticcurves}.
	
	As a next step, we remark that $\spe_f(\eta)=\eta_1$, where $\eta_1=\pi_1(\Ann_{\T_1}\gp_1)$ with $\gp_1=\ker(\pi_1 \colon \T_1 \to \cO)$. Indeed, the idempotent $e_1\in \T_1$ attached to $\pi_1$ in Lemma \ref{lem:congruence_module_facts}(2) lifts to $\T_\I$, and it must be equal to $e$ as $\ga \subseteq \gp$. In particular, $e(\T_\I) \otimes_{\I,\spe_f} \cO \simeq e_1(\T_1)$, inducing isomorphisms
\begin{equation}\label{eq:specialization_congruence_module}
	\I/\eta \otimes_{\I,\spe_f} \cO \simeq e(\T_\I)/(e(\T_\I)\cap \T_\I) \otimes_{\I,\spe_f} \cO \simeq e_1(\T_1)/(e_1(\T_1) \cap \T_1)\simeq \cO/\eta_1. 
\end{equation}
	Therefore, $\spe_f(\eta)=\eta_1$. 
	
	Let now $\breve{\cO}$ be the completion of the maximal unramified extension of $\cO$, let $\breve{\I}=\I \hat{\otimes} \breve{\cO}$, and let $\tilde{\eta}$ be the reflexive closure of $\eta$ inside $\breve{\I}$. Then, by \cite[Theorem 4.20]{burungale2024zetaelementsellipticcurves}, $\tilde{\eta}$ is generated by $\frac{h_K}{w_K}\cdot \cL^\textrm{Katz}_K$, where $h_K$ is the class number of $K$, $w_K=\#\mu(K)$, and $\cL^\textrm{Katz}_K\in \breve{\I}$ is the Katz $p$-adic $L$-function over $K$ for the trivial character branch. 
	
	Note that $w_K$ is a $p$-adic unit, as $p$ is odd and unramified in $K$. Moreover, $h_K=L(\chi,0)\cdot w_K/2$ by the class number formula for $K$. On the other hand, Katz's $p$-adic analogue of the Kronecker limit formula for the trivial character reads  
	\[ \spe_f(\cL^\textrm{Katz}_K)= \frac{4}{w_K}\left(1-\frac{1}{p}\right)\log_p(\ob{u}), \]
	where $\ob{u}$ is a generator of $\ob{\gp}^{h_K}_K$, $\ob{\gp}_K$ being the conjugate of the distinguished prime $\gp_K$ of $K$ above $p$ (see \cite[end of p.90]{grossfactorization}). Notice that $u=p^{h_K}/\ob{u}$ is a $\gp_K$-unit of $K$, and $u/\ob{u}$ is a multiple of the normalized Gross--Stark unit for $\chi$ (see Definition \ref{def:GS_unit}). In particular, 
\[	\cL_{\chi}=-\dfrac{\lambda(u/\ob{u})}{\ord_p(u/\ob{u})}=2\cdot \dfrac{\lambda(\ob{u})}{h_K}. \]
	Therefore,
	\[\spe_f(\cL^\textrm{Katz}_K)\sim p^{-1}\log_p(\ob{u}) \sim \lambda(\ob{u})\sim h_K \cL_{\chi} \sim L(\chi,0)\cL_{\chi},  \]
	where $a\sim b$ means that $a$ and $b$ have the same $p$-adic valuation. This proves that $\breve{\I}/\tilde{\eta} \otimes_{\breve{\I},\spe_f}\breve{\cO}$ has $\breve{\cO}$-length equal to $\length \cO/(\cL_\chi \cdot L(\chi,0)^2)$.
	
	To conclude, we only have to check that $\breve{\I}/\tilde{\eta} \otimes_{\breve{\I},\spe_f}\breve{\cO}$ and $C^0(\pi_1)$ have the same length. Since $P=\ker \spe_f \subset \I$ is of height $1$ and $\I$ is normal, $\tilde{\eta}_P=\eta_P \hat{\otimes} \breve{\cO}$ as ideals of the localized ring $\breve{\I}_P$, so the length of $\breve{\I}/\tilde{\eta} \otimes_{\breve{\I},\spe_f}\breve{\cO}$ is equal to the length of $\I/\eta \otimes_{\I,\spe_f} \cO$, which is the one of $\cO/\eta_1=C^0(\pi_1)$ by \eqref{eq:specialization_congruence_module}.
\end{proof}

The proof of the second part of Theorem \ref{thm:congruence_bound} consists in explicitly constructing a congruence between $f$ and another weight $1$ cusp form in $\bS_1$ using the Eisenstein families passing through $f$. 

Let $\M\supseteq \bS$ be the localization at the Eisenstein maximal ideal $\gm$ defined in \eqref{eq:maximal_ideal_eisenstein} of the space of $\Lambda(N)$-adic Hida families of tame level $N$, and let $\tilde{\T}$ be the full Hecke algebra acting faithfully on $\M$. There are two Eisenstein families $\cE_{1,\chi}$ and $\cE_{\chi,1}$ in $\M$ which are Hecke eigenforms for $\tilde{\T}$ and whose associated Galois representations are $1\oplus \chi \varepsilon^\textrm{u}$ and $\chi\oplus \varepsilon^\textrm{u}$ respectively. We refer to \cite[\S3.2]{lafferty} for their $q$-expansion; however, we warn the reader that we follow the notation in \cite{betinadimitrovpozzi}, which is different from the one in \cite{lafferty}: our $\cE_{1,\chi}$ is precisely what is denoted by $\cE_{\chi,1}$ in \cite{lafferty}. In particular, $\cE_{1,\chi}$ has a nonzero residue at the cusp $\infty$.

\begin{proof}[Proof of Theorem \ref{thm:congruence_bound}(2)]
For $\varphi\in\{\chi,\chi^{-1}\}$, let $\zeta_\varphi\in \Lambda$ be the Iwasawa power series such that $\zeta_\varphi((1+p)^{-s}-1)$ equals the Kubota--Leopoldt $p$-adic $L$-function $L_p(\varphi\omega,s)$ for $s\in \Zp$. The Ferrero--Greenberg formula \cite{ferrerogreenberg} then reads 
\begin{equation}\label{eq:ferrero_greenberg}
	 \zeta_\varphi \equiv -\frac{L_p'(\varphi\omega,0)}{\log_p(1+p)}X \equiv -\cL_\varphi L(\varphi,0)X \mod X^2. 
\end{equation}

We now make use of the work of Lafferty \cite{lafferty}. It is shown that the values at all cusps of $\cE_{1,\chi}$ (resp. of $\cE_{\chi,1}$) are $\Lambda$-multiples of $\zeta_\chi$ (resp. of $\zeta_{\chi^{-1}}$, see also the proof of \cite[Proposition 4.7]{betinadimitrovpozzi}). In fact, more is known: by \cite[\S5]{lafferty} and, in particular, (29) of \textit{loc. cit.}, there exist cusp forms $\cG_{1,\chi},\cG_{\chi,1}\in \bS$ such that
\begin{equation}\label{eq:construction_G}
	 \cE_{1,\chi} \equiv \cG_{1,\chi} \mod \zeta_\chi, \qquad \cE_{\chi,1} \equiv \cG_{\chi,1} \mod \zeta_{\chi^{-1}}. 
\end{equation}
(Our cusp forms $\cG_{1,\chi},\cG_{\chi,1}$ are the negative of the ones in loc. cit., sent onto $\bS$ via the idempotent associated to the local factor $\T$ of $\T(N) \otimes \cO$.) 

We will use these cuspidal families to produce a congruence modulo $c$ (defined in \eqref{eq:congruence_number_lower_bound}) between $f$ and an adequate cusp form. 
More precisely, if $c=(\cL_{\chi}+\cL_{\chi^{-1}})\cdot \den(\cL_{\chi^{-1}})\cdot L(\chi,0)$, then we use $\cE_{1,\chi},\cG_{1,\chi}$ in our argument, and if $c=(\cL_{\chi}+\cL_{\chi^{-1}})\cdot \den(\cL_{\chi})\cdot L(\chi^{-1},0)$, then we use the other two. Since the arguments are similar, we focus on the first case. 

Recall that $\I$ is the normalization of the branch $\T/\ga$ defining $\cF$, and let $P=\ker \spe_f \subseteq \I$. By Proposition \ref{prop:finite_length_C}(1), the completion $\I_P^\wedge$ of $\I_P$ is an unramified extension of $\Lambda_{(X)}^\wedge=E\lsem X \rsem$. Since $\I_P^\wedge$ has residue field $E$, it follows that $\I_P^\wedge=E\lsem X\rsem$ canonically. 

Since $\cE_{1,\chi}$ and $\cG_{1,\chi}$ both specialize to $f$ at $X=0$, we may define
\[  f_{1,\chi} =\dfrac{\cL_{\chi}+\cL_{\chi^{-1}}}{\cL_{\chi}}\cdot \dfrac{\cF-\cE_{1,\chi}}{X}_{\big| X=0}, \quad g_{1,\chi} =\dfrac{\cL_{\chi}+\cL_{\chi^{-1}}}{\cL_{\chi}}\cdot \dfrac{\cF-\cG_{1,\chi}}{X}_{\big| X=0}, \]
the differences of the families being taken in $\M \otimes_{\Lambda} E\lsem X \rsem$.
We have $f_{1,\chi}\in \M_1 \otimes_{\cO}E$ and $g_{1,\chi}\in \bS_1 \otimes E$, where we have put $\M_1=\M/X\M$. Moreover, $f^\dagger_{1,\chi}:=\log_p(1+p)f_{1,\chi}$ is the form defined in \cite[(53)]{betinadimitrovpozzi}, whose Fourier coefficients are computed in p. 262 of \textit{loc. cit.}. In particular, one obtains
\[a_0(f_{1,\chi})=\dfrac{c}{2\cdot \den(\cL_{\chi^{-1}})}, \qquad a_p(f_{1,\chi})=\cL_{\chi^{-1}},  \]
and for $\ell\neq p$,
\begin{equation*}\label{eq:a_ell_coeff_f_dagger}
	 a_\ell(f_{1,\chi})=(1-\chi(\ell))\lambda(\ell).
\end{equation*}
Choose any prime $\ell$ such that $\chi(\ell)\neq1$, $\ell\equiv 1 \mod p$ and $\ell \not\equiv 1 \mod p^2$. Then we have $a_\ell(f_{1,\chi})\in\cO^\times$. An inspection of the Fourier coefficients $a_n(f_{1,\chi})$ (using (54) of \textit{loc. cit.} if $n$ is not a prime) shows that the form $f'=\den(\cL_{\chi^{-1}})\cdot f_{1,\chi}$ belongs to $\M_1$. Furthermore, $a_\ell(f')$ or $a_p(f')$ belongs to $\cO^\times$, so $f'$ does not vanish mod $p$.

Let $g'=\den(\cL_{\chi^{-1}})\cdot g_{1,\chi}$. Then $g'\in\bS_1$, and the congruences \eqref{eq:ferrero_greenberg} and \eqref{eq:construction_G} imply that $g' \equiv f' \mod c$.

To show that $\length C^0(\pi_1)\geq \length \cO/(c)$, we are going to apply Lemma \ref{lem:congruence_module_facts}(3) with $T=\T_1$ and $\pi=\pi_1$. Let $e_1\in \T_1 \otimes_{\cO} E$ be the idempotent attached to $\pi_1$. We will construct a Hecke operator $t\in \gp_1=\ker\pi_1$ such that $t(f')=f$. Since $e_1t=0$ in $\T_1\otimes_{\cO}E$, this will show that the form $h:=t(g')\in \bS_1$ satisfies $h \equiv f$ mod $c$ and $e_1(h)=0$. 

It is well known that $f'$ is a generalized eigenform for $\pi_1$. Indeed, one easily checks from the definition of $f'$ that $\tilde{f}:=f+Xf' \mod X^2$ is an eigenform for $\tilde{\T}/(X^2)$.  Hence, given any Hecke operator $T\in \tilde{\T}/(X^2)$, we have 
\[T\tilde{f}=(a_1(Tf)+Xa_1(Tf'))\tilde{f}.\]
In particular, if $a_1(Tf')\in \cO^\times$, then $t(f')=f$ for 
\[t=(T-a_1(Tf))/a_1(Tf')\in\gp_1. \]
If $a_p(f')\in\cO^\times$, then we may take $T=U_p$. Otherwise, $a_\ell(f')\in\cO^\times$ for the prime $\ell$ fixed earlier, and $T=T_\ell$ works. 
\end{proof}

The following conjecture, which holds by Theorem \ref{thm:congruence_bound}(1) if $\chi$ is quadratic, will be further justified by analytic means in \S\ref{sec:trivial_zero_conjecture}.

\begin{conjecture}\label{conj:length_C0}
	Let $\pi_1 \colon \T_1 \to \cO$ be the system of Hecke eigenvalues of $f$ as in Theorem \ref{thm:congruence_bound}. 
	Then 
	\[\length C^0(\pi_1)=\length \cO/((\cL_\chi +\cL_{\chi^{-1}})\cdot L(\chi,0)\cdot L(\chi^{-1},0)).\]
\end{conjecture}

\subsection{Modularity and Gorensteinness}
We are now ready to formulate and prove the main modularity theorem of this paper, comparing $\T$ with the universal deformation ring $R^\ord$ of Proposition \ref{prop:representability}.

\begin{proposition}\label{prop:surjection_R_to_T}
	Assume $A'_K[\chi]=A'_K[\chi^{-1}]=\{0\}$. Then there exists a natural surjective map $\varphi\colon R^\ord \twoheadrightarrow \T$ in $\mathcal{CNL}_{\cO}$.
\end{proposition}
\begin{proof}
	We are going to apply Lemma \ref{lem:weak_ordinarity}. Recall from \S\ref{sec:Hecke_algebras_and_eigencurve} that $\T$ is reduced; let $\cK$ be its total ring of fractions. By results of Hida and Wiles (\cite{hida1986galois} and \cite[Theorem 2.2.2]{wiles1988ordinary}), there exists an irreducible Galois representation $\rho \colon G_\Q \to \GL_2(\cK)$ unramified outside $Np$, and an ordinary $\cK$-line $L$ of $V=V_\rho$ such that the Frobenius element $\phi$ acts on $V/L$ by $U_p\in \T^\times$. Moreover, $\tr(\rho)(\Frob_\ell)=T_\ell$ for $\ell\nmid Np$, so it is clear from the description of $\gm$ in \eqref{eq:maximal_ideal_eisenstein} that $\tr(\rho)\equiv 1+\chi \mod \gm$.  
	
	Finally, $\rho$ is minimally ramified at $q|N$, as $\T$ has tame level $N$. More precisely, every Hida family contributing to $\T$ has tame nebentype $\chi$ since $p\nmid \phi(N)$. Hence, \cite[Lemma 6.3(ii)]{dimitrov2014local} implies that $\rho_{|G_{\Q_q}}\simeq \delta_1 \oplus \chi\delta_2$ with $\delta_i$ unramified characters of $G_{\Q_q}$ such that $\delta_1(\Frob_q)=U_q$. Therefore, $(\rho,L)$ can be conjugated so that it gives rise to a $\T$-point of $\cD^\ord$.
	
	Let $\varphi \colon R^\ord \to \T$ be the morphism induced by $(\rho,L)\in\cD^\ord(\T)$. To show that it is surjective, it is enough to check that its image contains $T_\ell$, $U_q$ and $U_p$ for $\ell\nmid Np$ and $q|N$. Let $(\rho^\ord,L^\ord)$ be the universal deformation, so that $\varphi \circ \rho^\ord=\rho$ and $\varphi(L^\ord)=L$. If $\ell\nmid Np$, then $T_\ell$ is the image under $\varphi$ of $\tr(\rho^\ord)(\Frob_\ell)$. If $q|N$, then $U_q$ is the image of the eigenvalue of $\Frob_q$ acting on the unramified rank-$1$ direct summand of $V^\ord$. Finally, $U_p$ is the image of the eigenvalue of $\phi$ acting on $V^\ord/L^\ord$. This proves that $\varphi$ is surjective.
\end{proof}

\begin{theorem}\label{thm:modularity}
	Assume $A'_K[\chi]=A'_K[\chi^{-1}]=\{0\}$. If $\chi$ is not quadratic, further assume that $\lambda(u_\chi)$ or $\lambda(u_{\chi^{-1}})$ is a $p$-adic unit, where $u_{\chi},u_{\chi^{-1}}$ are the normalized Gross--Stark units in Definition \ref{def:GS_unit}, or assume Conjecture \ref{conj:length_C0}.
		
		Then the surjective map $\varphi \colon R^\ord \twoheadrightarrow \T$ is an isomorphism of local complete intersections. In particular, $\T$ is Gorenstein, and for all $k\in\Z$, we have an isomorphism $R^\ord_k\simeq\T_k$ of local complete intersections.
\end{theorem}
\begin{proof}
	Let $\varphi_1  \colon R_1^\ord \twoheadrightarrow \T_1$ be the weight $1$ specialization of $\varphi$, obtained by quotienting by $X$.  We first prove that $\varphi_1$ is an isomorphism of local complete intersections by applying the numerical criterion (Theorem \ref{thm:numerical_criterion}) to $R=R^\ord_1$, $T=\T_1$, and to the system of Hecke eigenvalues $\pi_{1,T}\colon T \to \cO$ attached to $f$.  
	
	Note that the point $(\rho_1,L_1)\in\cD^\ord(\cO)$ coming from $\pi_{1,R}=\pi_{1,T} \circ \varphi_1 \colon R \to \cO$ must satisfy $\rho_1=(\begin{smallmatrix}\chi & 0 \\ 0 & 1\end{smallmatrix})$ in the Eisenstein basis, by Lemma \ref{lem:reducible_ordinary_points}. Therefore, Proposition \ref{prop:length_C1_R} tells us that 
\begin{equation}\label{eq:length_C1_R}
		\length C^1(\pi_{1,R})=\length \cO/(c_f), 
\end{equation}
	where $c_f=(\cL_{\chi}+\cL_{\chi^{-1}})\cdot L(\chi,0)\cdot L(\chi^{-1},0)$ is the quantity appearing in Conjecture \ref{conj:length_C0}. 
	
	Without assuming Conjecture \ref{conj:length_C0}, we may still apply Theorem \ref{thm:congruence_bound}, and we only have to check that, in the non-quadratic case, the quantity $c$ of \eqref{eq:congruence_number_lower_bound} and $c_f$ have the same $p$-adic valuation. If $\lambda(u_\chi)$ is a unit, then $\den(\cL_{\chi^{-1}})=\ord_p(u_\chi)=L(\chi^{-1},0)$ up to $\cO$-units by Lemma \ref{lem:calcul_ord} and Theorem \ref{thm:refined_CNF}. It is then clear that $v_p(c)= v_p(c_f)$. The case where $\lambda(u_{\chi^{-1}})$ is a unit is similar, and we deduce that $\varphi_1$ is an isomorphism.
	
	Let now $J=\ker( \varphi \colon R^\ord \twoheadrightarrow \T)$. Then there is an exact sequence of finitely generated $\cO$-modules of the form
	\[\T[X] \longrightarrow J/XJ \longrightarrow R_1^\ord \overset{\varphi_1}{\longrightarrow} \T_1 \longrightarrow 0. \]
	Since $\T$ is $\Lambda$-flat and $\varphi_1$ is injective, we deduce that $J/XJ=\{0\}$, and hence, $J=\{0\}$ by Nakayama's lemma. 
	
	The property of $\T$ being a local complete intersection follows from the one of $\T_1$ and the fact that $X$ is a nonzero divisor in $\T$. The same is true for $\T_k=\T/(1+X-(1+p)^{k-1})$ for the same reason. This ends the proof of the theorem.
\end{proof}

\begin{corollary}
	Let $\rho \colon G_{\Q} \to \GL_2(\ob{\Q}_p)$ be a continuous irreducible Galois representation that is unramified outside $Np$ and such that the semisimplification $\ob{\rho}^\textrm{ss}$ of its residual representation is $1\oplus \ob{\chi}$. Further assume that $\rho$ is minimally ramified at primes $\ell|N$ and is ordinary at $p$. Then, under the hypotheses of Theorem \ref{thm:modularity} on $\chi$, $\rho$ is the Galois representation attached to an overconvergent $p$-adic cusp form.
\end{corollary}
\begin{proof}
	We may conjugate $\rho$ so that it takes values in a finite extension $E'$ of $E$. By Lemma \ref{lem:weak_ordinarity}, it can further be conjugated so that it gives rise to an $\cO_{E'}$-point of $\cD^\ord$, and hence, an $\cO_{E'}$-point of $\T$ thanks to Theorem \ref{thm:modularity}. Such points give rise to $E'$-points of the cuspidal eigencurve $\cC$, and hence, to an overconvergent $p$-adic cusp form.
\end{proof}

We now derive from Theorem \ref{thm:modularity} some structural results on Hecke algebras. 

Recall the notation $\iota_p \colon K\hookrightarrow \Qp$ for the fixed $p$-adic embedding, and $\gp_K$ the distinguished prime of $K$ above $p$. If $\chi$ is quadratic, choose any generator 
\[u_K\in \cO_K[1/p]^\times\] 
of $\gp_K^{h_K}$, where $h_K$ is the class number of $K$. We see $u_K$ as a nonzero element of $\cO=\Zp$ via $\iota_p$.

\begin{corollary}\label{coro:rank1}
	$\T_1=\cO$ if and only if 
	\[c_f := (\cL_{\chi}+\cL_{\chi^{-1}}) \cdot L(\chi,0) \cdot L(\chi^{-1},0) \not\equiv 0 \pmod{p}. \]
	In particular, if $\chi$ is quadratic, then $\T_1=\Zp$ if and only if $p\nmid L(\chi,0)$ and $p\nmid \lambda(u_K)$.
\end{corollary}
\begin{proof}
	Let $c\in \cO$ be the quantity defined in \eqref{eq:congruence_number_lower_bound}, and let $c^0_f$ be a congruence number of $f$ as in Definition \ref{def:congruence_number_of_f}. 
	
	We record some facts concerning $c,c_f$ and $c^0_f$. A comparison of valuations shows that $c\mid c_f$, and $c_f \mid c^2$, so $c\in\cO^\times$ if and only if $c_f\in\cO^\times$. Theorem \ref{thm:congruence_bound} shows that $c\mid c^0_f$, and Lemma \ref{lem:rank1} says that $\T_1=\cO$ if and only if $c^0_f\in \cO^\times$. From this, one readily sees that, if $\T_1=\cO$, then $c$ must be an $\cO$-unit, and $c_f$ as well. 
	
	Conversely, if $c_f\in\cO^\times$, then expanding the sum in $c_f$ shows that either $(\cL_{\chi}\cdot L(\chi,0)) \cdot L(\chi^{-1},0)$ or $(\cL_{\chi^{-1}}\cdot L(\chi^{-1},0)) \cdot L(\chi,0)$ is a unit. By symmetry, we may assume the first expression is a unit. Then both factors $\cL_{\chi}\cdot L(\chi,0)\in\cO$ and $L(\chi^{-1},0)\in\cO$ are $\cO$-units. Using Lemma \ref{lem:calcul_ord} and Theorem \ref{thm:refined_CNF}, this forces $A'_K[\chi^{-1}]=\{0\}$, $\lambda(u_{\chi^{-1}})\in\cO^\times$ on the one hand, and $A_K[\chi]=\{0\}$ on the other hand. In particular, we may apply Theorem \ref{thm:modularity} and its proof, which give an isomorphism $R_1^\ord \simeq \T_1$ and, at the same time, the divisibility $c^0_f \mid c_f$ by \eqref{eq:C1geqC0} and \eqref{eq:length_C1_R}. Hence, $c^0_f$ is a unit, and $\T_1=\cO$.
	
	If $\chi$ is quadratic, then $c_f=2\, \cL_\chi \, L(\chi,0)^2$, and the proof of Theorem \ref{thm:congruence_bound}(1) shows that $\cL_\chi=-2\lambda(u_K)/h_K$, while $h_K=L(\chi,0)\,w_K/2$ with $w_K\in\Z_p^\times$. Hence $c_f=\lambda(u_K)\cdot L(\chi,0)$ up to a $p$-adic unit, and the last assertion follows.
\end{proof}

In the next corollaries, we promote the Hecke algebra $\T_1$ to an augmented $\cO$-algebra using the map $\pi_1 \colon \T_1 \to \cO$ attached to $f$. Rings of the form $\cO\lsem x_1,\ldots,x_n\rsem/(g_1,\ldots,g_n)$, $g_i(0)=0$, will also be considered as augmented $\cO$-algebras via the specialization map $x_1=\cdots =x_n=0$.

\begin{corollary}\label{coro:structure_T_1_chi_quadratic}
	Assume $\chi$ is quadratic (so that $\cO=\Zp$) and $A'_K=\{0\}$. 
	\begin{enumerate}
		\item There exists a presentation $\T_1\simeq \Zp\lsem x\rsem/(g(x))$, $g(0)=0$, as augmented $\Zp$-algebras if and only if $p\nmid L(\chi,0)$ or $p\nmid \lambda(u_K)$. If these equivalent conditions hold, then $g'(0)=\lambda(u_K)\cdot L(\chi,0)$ up to a unit.
		\item  In general, there exists a presentation $\T_1\simeq \Zp\lsem x,y\rsem/(g_1,g_2)$, $g_i(0)=0$, as augmented $\Zp$-algebras. If $g_1(x,y)=ax+by$ and $g_2(x,y)=cx+dy$ modulo higher terms, then $\alpha=\gcd(\lambda(u_K),L(\chi,0))$ and $\beta=\lambda(u_K)\cdot L(\chi,0)/\alpha$ are the elementary divisors of the matrix $(\begin{smallmatrix}
			a & b \\ c & d
		\end{smallmatrix})$.
	\end{enumerate}
\end{corollary}
\begin{proof}
		Let $\gp=\ker\pi_1 \subseteq \T_1$. Since $\T_1$ is $\gp$-adically complete and a local complete intersection, the minimal number of generators of $C^1(\pi_1)=\gp/\gp^2$ as a $\Zp$-module is the minimal integer $n\geq0$ such that there exists a presentation $\T_1\simeq \Zp\lsem x_1,\ldots,x_n\rsem/(g_1,\ldots,g_n)$, $g_i(0)=0$, as augmented $\Zp$-algebras. The prime $\gp$ is then generated by $x_1,\ldots,x_n$, and it is easy to see that, if $g_i=a_{i1}x_1+\ldots+a_{in}x_n$ modulo higher terms, then
		\[\coker((a_{ij})_{1\leq i,j \leq n} \colon \Z_p^n \to \Z_p^n) \simeq \gp/\gp^2.\]
		 
		 By Theorem \ref{thm:modularity} and Proposition \ref{prop:length_C1_R}, the integer $n$ is precisely the number of generators of $\Zp/(\alpha)\times\Zp/(\beta)$, where $\alpha=\gcd(\lambda(u_{\chi}),\ord_p(u_\chi))$ and $\beta=\lambda(u_{\chi})\cdot \ord_p(u_\chi)/\alpha$. We further notice that $\alpha$ divides $\beta$, so $n\leq1$ if and only if $\alpha$ is a unit.
		
		As $A'_K=\{0\}$, $L(\chi,0)$, $\ord_p(u_\chi)$ and $h_K$ are equal up to a $p$-adic unit by Theorem \ref{thm:refined_CNF} and Lemma \ref{lem:calcul_ord}. In particular, $u_K/\ob{u}_K$ is a valid choice of a normalized Gross--Stark unit $u_{\chi}$, where $\ob{u}_K=\pm p^{h_K}/u_K$ is the complex conjugate of $u_K$. Hence, $\lambda(u_{\chi})=\lambda(u_K/\ob{u}_K)=2\lambda(u_K)$, which is a unit multiple of $\lambda(u_K)$. The corollary easily follows from these observations.
\end{proof}

\begin{corollary}\label{coro:structure_T_1_chi_non_quadratic}
	Assume $\chi$ is not quadratic, and that the hypotheses of Theorem \ref{thm:modularity} hold. Let $\ell$ be the $p$-adic valuation of $\cL_{\chi}+\cL_{\chi^{-1}}$, and let $v,v',w$ be the $p$-adic valuations of $L(\chi,0)$, $L(\chi^{-1},0)$ and $\gcd(\lambda(u_{\chi}),\lambda(u_{\chi^{-1}}),L(\chi,0),L(\chi^{-1},0))$ respectively.
	\begin{enumerate}
		\item There exists a presentation $\T_1\simeq \cO\lsem x\rsem/(g(x))$, $g(0)=0$, as augmented $\cO$-algebras if and only if $w=0$ or $w=v+v'+\ell$. If these equivalent conditions hold, then $g'(0)=(\cL_{\chi}+\cL_{\chi^{-1}})L(\chi,0)L(\chi^{-1},0)$ up to an $\cO$-unit.
		\item In general, there exists a presentation $\T_1\simeq \cO\lsem x_1,x_2\rsem/(g_1,g_2)$, $g_i(0)=0$, as augmented $\cO$-algebras. If $g_1(x,y)=ax+by$ and $g_2(x,y)=cx+dy$ modulo higher terms, then $\alpha$ and $\beta$ in \eqref{eq:alpha} and \eqref{eq:beta} are the elementary divisors of the matrix $(\begin{smallmatrix}
			a & b \\ c & d
		\end{smallmatrix})$.
	\end{enumerate}
\end{corollary}
\begin{proof}
	It is similar to the proof of Corollary \ref{coro:structure_T_1_chi_quadratic}.
\end{proof}

\begin{remark}\label{rem:invariance_twist_chi}
	Many of the results and hypotheses regarding the structure of $R^\ord$ and $\T$ are invariant under exchanging $\chi$ and $\chi^{-1}$. 
	
	The following observation explains this phenomenon: if one denotes by $R^{\ord,\#}$ and $\T^\#$ the universal deformation ring and the Hecke algebras constructed with $\chi^{-1}$ instead of $\chi$, then there exist canonical isomorphisms $R^\ord \simeq R^{\ord,\#}$ and $\T\simeq\T^\#$. The first one is induced by the natural transformation $(\rho,L) \mapsto (\rho \otimes \chi^{-1},L)$, and the second one is induced by $F \mapsto F \otimes \chi^{-1}$ at the level of cuspidal Hida families (as described e.g. in \cite[p.438]{hsiehAJM}).
\end{remark}

\subsection{Failure of Gorensteinness of Hecke algebras}
Numerical calculations in \S\ref{sec:examples} consistently suggest the following converse to Theorem \ref{thm:modularity}.

\begin{conjecture}\label{conj:converse_gorenstein}
	If $\T$ is Gorenstein, then $A'_K[\chi]=A'_K[\chi^{-1}]=\{0\}$.
\end{conjecture}

For the remainder of this section, we provide two theoretical results towards this conjecture. We first establish it in the quadratic case, and we then discuss some connections to a conjecture that originates in the work of McCallum--Sharifi.

\begin{theorem}\label{thm:converse_gorenstein}
	Assume $\chi$ is quadratic and let $m$ be the $\Zp$-length of $A'_K$. If $A'_K\neq\{0\}$, then $\T$ is not Gorenstein and it has $\Lambda$-rank at least $1+2(p^m-1)$. In particular, Conjecture \ref{conj:converse_gorenstein} holds for $\chi$.
\end{theorem}
\begin{proof}
	Recall from \cite[Theorems 18.2 and 18.3]{matsumura} that, if the local noetherian ring $\T$ is Gorenstein and $\gp\in\Spec(\T)$, then the localization-completion $\T_\gp^\wedge$ at $\gp$ is Gorenstein as well. Therefore, assuming that $A'_K$ is \emph{not} $\{0\}$, it is sufficient to find a height one prime $\gp$ for which we know $\T_\gp^\wedge$ is \emph{not} Gorenstein.
	
	Assume $\Psi =\{\psi \colon A'_K \to \mu_{p^\infty}(\ob{\Q}_p)\}\setminus\{1\}$ is not empty, and fix a character $\psi \in \Psi$ of exact order $p^r$, which we see as an unramified Hecke character of $G_K$ of finite order. A classical theorem of Hecke attaches to $\psi$ a weight $1$ newform
	\[\theta = \sum_{0 \neq \ga \subseteq \cO_K} \psi(\ga)q^{\mathbf{N}(\ga)}\]
	of level $|\mathrm{Disc}(K)|=N$, where $\mathbf{N}$ is the norm map of $K/\Q$. Let $\gp_K\ob{\gp}_K$ be the factorization of $p\cO_K$. The condition that $\psi$ factors through $A'_K$ implies that $\psi(\gp_K)=\psi(\ob{\gp}_K)=1$. Hence, under the terminology of \cite{bellaiche2016eigencurve}, $\theta$ is $p$-irregular, with a unique $p$-stabilization
		\[\theta^{(p)} = \sum_{\substack{0\neq\ga \subseteq \cO_K \\ \ob{\gp}_K\nmid \ga}} \psi(\ga)q^{\mathbf{N}(\ga)}.\]
	Moreover, we have a congruence $\psi \equiv 1$ mod $\varpi$ over $G_K$, for $\varpi$ a uniformizer of $\cO[\zeta_{p^r}]$ and $\zeta_{p^r}$ a primitive $p^r$-th root of unity. Hence, for any integer $n\geq1$ we have $a_n(f)\equiv a_n(\theta^{(p)})$ mod $\varpi$, meaning that $\theta^{(p)}$ belongs to $\T_1$.
	
	Let $\gp_{\psi}\subseteq \T$ and $\gp_{1,\psi}\subseteq \T_1$ (resp. let $\mathrm{y}$) be the height one prime (resp. the point of the cuspidal eigencurve $\cC$ of \S\ref{sec:Hecke_algebras_and_eigencurve}) corresponding to $\theta^{(p)}$. By \cite[Proposition 5.4]{betinadimitrovKatz}, $\cC$ is not Gorenstein at $\mathrm{y}$, and the completed local ring $\cO_{\cC,\mathrm{y}}^\wedge$ has rank $4$ over $\cO^\wedge_{\cW,{\kappa(y)}}$. By the discussion in \S\ref{sec:Hecke_algebras_and_eigencurve}, this implies that $\T_{\gp_\psi}^\wedge$ is not Gorenstein. Moreover, $\dim_\Qp (\T_{1})_{\gp_{1,\psi}}\geq4$. 
	
	To obtain the bound on the rank of $\T_1$, note that any two characters $\psi,\psi'\in \Psi$ as above give rise to the same theta series if and only if $\psi'=\psi^{\pm1}$. Therefore, 
	\[\T_1 \otimes \Qp \simeq \prod_{\gp} (\T_{1})_{\gp} \twoheadrightarrow \Qp \times \prod_{\psi^{\pm1} \in \Psi} (\T_{1})_{\gp_{1,\psi}}, \]
	where the first factor corresponds to $f$. Since there are $(p^m-1)/2$ pairs $\{\psi,\psi^{-1}\}$ of elements in $\Psi$, we indeed obtain $\rk_\Lambda \T=\rk_{\Zp} \T_1 \geq 1+2(p^m-1)$.
\end{proof}

\subsection{Link to McCallum--Sharifi's conjecture}\label{sec:sharifi}
We close this section with a further theoretical justification of Conjecture \ref{conj:converse_gorenstein}, assuming a conjecture that goes back to McCallum--Sharifi's work \cite{mccallumsharifi}. We mainly follow Shih--Wang's exposition in \cite[\S\S2--3]{shihwangTAMS}.

Assume in this paragraph that $p\geq 5$. Given any $\Zp[G_{\Q}]$-module $M$, let $M^\pm$ be the submodule of $M$ where the complex conjugation acts by $\pm1$. If $M$ has a structure of module over $\Lambda=\cO\lsem\Gamma\rsem$, we denote by $M(i)=M\otimes_{\Zp}\Zp(i)$ the $i$-th Tate twist of $M$, which we make into a $\Lambda$-module by letting $\Gamma$ act diagonally (and by $\chi_{\cyc}^i=\varepsilon_{\cyc}^i\omega^i$ on $\Zp(i)$). 

For $r\geq1$, let $\HH_r=\HH^1_{\text{ét}}(X_1(Np^r)_{/\ob{\Q}},\Zp)^\ord$ be the ordinary part of the étale cohomology of the closed modular curve of level $\Gamma_1(Np^r)$, and consider the continuous cohomology group $S_r=\HH^2(\Z[\zeta_{Np^r},\tfrac{1}{p}],\Zp(2))$. In \cite[Proposition 5.7]{sharifi}, Sharifi constructed a map 
\[\varpi_r \colon \HH_r^-(1) \to \HH^2(\Z[\zeta_{Np^r},\tfrac{1}{Np}],\Zp(2))^+,\]
and whose image is contained in $S_r$ by \cite[Theorem 5.3.5]{fukayakato}. Roughly speaking, it maps certain cycles (seen in $\HH_r$ via Poincaré duality) to cup products of cyclotomic units. The maps $\varpi_r$ are $\Zp[(\Z/Np^r)^\times]$-linear, and form a compatible system inducing a map
\[\varpi \colon \HH^-(1) \to S,\]
where $\HH=\varprojlim_r \HH_r$ and $S=\varprojlim_r S_r$. 

Recall that $\HH$ has a natural action of the \textit{dual} big Hecke algebra $\T^*(N)$ generated by Hecke operators $T_\ell^*$ and $U_q^*$, $\ell\nmid Np$ and $q|Np$. It was conjectured in Conjecture 5.8 of \textit{loc. cit.} (and established in \cite[Theorem 5.2.3]{fukayakato}), that $\varpi$ factors through $\HH^-(1)/I^*\HH^-(1)$, where $I^*$ is the Eisenstein ideal generated (with our conventions) by $T^*_\ell-(1+\langle\ell\rangle^{-1})$ for $\ell\nmid Np$ and $U^*_q-1$ for $q|Np$.

 Let $\varphi$ be an $\cO$-valued Dirichlet character of modulus $Np$. Then $\varphi$ induces a homomorphism $\Zp\lsem \tilde{\Gamma} \rsem \to \Lambda$, where $\tilde{\Gamma}=(1+p\Zp) \times (\Z/Np\Z)^\times$. Given a $\Zp\lsem \tilde{\Gamma}\rsem$-module $M$, let $M_\varphi= M \otimes_{ \Zp\lsem \tilde{\Gamma} \rsem,\varphi} \Lambda$ be the $\varphi$-component of $M$. Since $p\nmid \phi(N)$, we may identify $M_\varphi$ with $(M\otimes_{\Zp} \cO)[\varphi]$.

\begin{conjecture}[McCallum--Sharifi]\label{conj:mccallum_sharifi}
	Assume $p\geq5$ as before and let $\theta$ be an $\cO$-valued primitive even Dirichlet character of conductor $Np$. Then the homomorphism $\varpi$ induces a surjective homomorphism of $\Lambda$-modules
	\[\varpi_\theta \colon H_\theta^-(1)/I^*H_\theta^-(1) \twoheadrightarrow S_\theta. \]
\end{conjecture}

This is the content of \cite[Conjecture 1.1]{shihwangTAMS}. The main observation in our paper regarding this conjecture is that it implies Conjecture \ref{conj:converse_gorenstein}.

\begin{theorem}\label{thm:sharifi}
	Assume that $p\geq5$ as before, and assume Conjecture \ref{conj:mccallum_sharifi} holds for $\theta=\chi\omega$ and $\theta=\chi^{-1}\omega$. If $\T$ is Gorenstein, then $A'_K[\chi]=A'_K[\chi^{-1}]=\{0\}$ and the maps $\varpi_{\chi\omega}$ and $\varpi_{\chi^{-1}\omega}$ are isomorphisms. In particular, Conjecture \ref{conj:converse_gorenstein} holds.
\end{theorem}
\begin{proof}
	Assume $\T$ is Gorenstein; we will prove that $A'_K[\chi]=\{0\}$ and $\varpi_{\chi\omega}$ is injective. By Remark \ref{rem:invariance_twist_chi}, the Hecke algebra $\T^\#$ constructed like $\T$ with $\chi$ replaced with $\chi^{-1}$ is isomorphic to $\T$, hence Gorenstein, so our argument applied to $\T^\#$ will prove that $A'_K[\chi^{-1}]=\{0\}$ and $\varpi_{\chi^{-1}\omega}$ is injective as well. 
	
	Recall that the Hecke algebras $\T(N)$ and $\T^*(N)$ are canonically isomorphic, and $T_\ell$ (resp. $U_q$, $\langle \ell\rangle$) is identified with $T^*_\ell$ (resp. $U^*_q$, $\langle \ell\rangle^{-1}$). In particular, $\T\simeq \T^*$, where $\T^*$ is the localization at the maximal ideal $\gm^*$ of $\T^*(N) \otimes \cO$ corresponding to $\gm$. We remark that the $\Lambda$-module structure on $\T^*$ is a twist by $\iota$ of that of $\T$, where $\iota \colon \Lambda \to \Lambda$ is the automorphism sending $[a]$ to $[a^{-1}]$ for $a\in \Gamma$. 
	
	The ideal $\gm^*$ is the unique maximal ideal of $\T^*(N) \otimes \cO$ containing $I^*$. Therefore, we have \[\HH(1)^+_{\gm^*}/I^*\HH(1)^+_{\gm^*}\simeq \HH(1)^+_\chi/I^*\HH(1)^+_\chi=\HH^-_\theta(1)/I^*\HH^-_\theta(1)\]
	as $\T^*/I^*$-modules. On the other hand, $ \T^*$ being Gorenstein implies that $\HH(1)_{\gm^*}$ is free of rank $2$ over $\T^*$; this follows from Ohta's construction of an Eichler--Shimura isomorphism and a Poincaré duality for cuspidal Hida families (see e.g. \cite[\S7.1]{BMP}). Hence, $\HH(1)_{\gm^*}^+\simeq \T^*(1)$, and so \[\HH^-_\theta(1)/I^*\HH^-_\theta(1)\simeq (\T^*/I^*)(1)\]
	as $\Lambda$-modules. 
	
	We remark that $I=\iota(I^*)$ is the image in $\T$ of the annihilator of the Eisenstein family denoted $\cE_{\chi,1}$ in the proof of Theorem \ref{thm:congruence_bound}(2). Hence, by \cite[Proposition 13]{lafferty}, the system of Hecke eigenvalues of $\cE_{\chi,1}$ induces an isomorphism $\T/I\simeq \Lambda/(\zeta_{\chi^{-1}})$ of $\Lambda$-modules, where $\zeta_{\chi^{-1}}\in\Lambda$ is the Iwasawa power series attached to $\chi^{-1}$ as in \S\ref{sec:construction_congruences}. This shows that 
	\[\HH^-_\theta(1)/I\HH^-_\theta(1)\simeq (\Lambda/(\zeta_{\chi^{-1}}^\iota))(1)=\Lambda/(\xi) \]
	as $\Lambda$-modules, where $\xi(X)=\zeta_{\chi^{-1}}^\iota((1+p)(1+X)-1)$.
	
	We now turn to the description of $S_\theta$, where $\theta=\chi\omega$. By \cite[(3.3)]{shihwangTAMS}, there exists a short exact sequence 
	\begin{equation}\label{eq:Iwasawa_ses}
		0 \to X'_\chi(1) \to S_\theta \to (\Lambda/(X))(1) \to 0 
	\end{equation}
	of $\Lambda$-modules, where $X'_\chi=\varprojlim_r A'_{\Q(\zeta_{Np^r})}[\chi]$ is the $\chi$-component of the $p$-split unramified Iwasawa module for $\Q(\zeta_{Np})$. Using the multiplicativity of characteristic ideals, Lemma 3.1 of \textit{loc. cit.} and  Mazur--Wiles' proof of the Iwasawa main conjecture \cite{mazurwiles} show that $\char_{\Lambda} X'_\chi = (\zeta_{\chi^{-1}}^\iota(X)/X)$, so \eqref{eq:Iwasawa_ses} yields
	\[\char_\Lambda S_\theta= (\xi)=\char_\Lambda\HH^-_\theta(1)/I\HH^-_\theta(1).\]
	Since $\varpi_\theta$ is surjective by assumption, its kernel must be a pseudo-null $\Lambda$-submodule of the cyclic module $\HH^-_\theta(1)/I\HH^-_\theta(1)\simeq \Lambda/(\xi)$, hence it is zero. This proves that $\varpi_\theta$ is an isomorphism and, in particular, $S_\theta$ is $\Lambda$-cyclic.
	
	To show that $A'_K[\chi]=\{0\}$, consider the Iwasawa module $A=\varinjlim_r A_{\Q(\zeta_{Np^r})}$. Then $A_\chi$ is a co-finitely generated cotorsion $\Lambda$-module, and taking the $\Gamma$-invariants yields 
	\[A_\chi^\Gamma\simeq E/\cO \times A'_K[\chi] \]
	by \cite[Proposition (3.4).3]{federergross}. Therefore, it is enough to prove that $(A_\chi)^\vee$ is cyclic over $\Lambda$, where $M^\vee=\Hom_{\cO}(M,E/\cO)$ is the Pontryagin dual of $M$, endowed with the $\Gamma$-action given by $\gamma . h(x)=h(\iota(\gamma)x)$.
	
	Let $E^1(M)=\operatorname{Ext}_\Lambda^1(M,\Lambda)$ be the Iwasawa adjoint of a $\Lambda$-module $M$. We refer to \cite[\S4]{wake} for some of the properties of $E^1$ we will make use of. We claim that 
	\[(A_\chi)^\vee \simeq E^1(S_\theta(-1)).\]
	Since the adjoint  of a cyclic module is cyclic, this will prove that $(A_\chi)^\vee$ is cyclic (and, in fact, isomorphic to $\Lambda/(\zeta_{\chi^{-1}})$).
	
	Let $\Q(\zeta_{Np^\infty})=\cup_r \Q(\zeta_{Np^r})$ and let $\mathfrak{X}$ be the Galois group of the maximal abelian pro-$p$ extension of $\Q(\zeta_{Np^\infty})$ which is unramified outside of $p$. By the proof of \cite[Corollary 4.4]{wake} (which is based on \cite[Corollary 4]{jannseniwasawaadjoints} applied to the discrete module $\Qp/\Zp(2)$), we have 
	\[S_\theta\simeq E^1(\mathfrak{X})(2)_\theta=E^1(\mathfrak{X}_{\theta^{-1}\omega^2}(-2)) \]
	as $\Lambda$-modules.
	To conclude, we use that, by \cite[Propositions 3.4.7 and 3.2.14]{sharifiiwasawa}, $\mathfrak{X}_{\theta^{-1}\omega^2}(-1)\simeq (A_\chi)^\vee$, and these $\Lambda$-modules are torsion and have no nontrivial finite $\Lambda$-submodules. Therefore, they are isomorphic to their double adjoint by \cite[Proposition 4.2(5)]{wake}, yielding
	\[E^1(S_\theta(-1))\simeq E^1(E^1(\mathfrak{X}_{\theta^{-1}\omega^2}(-1)))\simeq \mathfrak{X}_{\theta^{-1}\omega^2}(-1)\simeq (A_\chi)^\vee. \]
	This ends the proof of the theorem.
\end{proof}

\section{Locally split Galois representations}\label{sec:CG}
The focus of this section is on a certain closed subscheme $\Spec(\T^{\CM})$ of $\Spec(\T)$ capturing modular forms with complex multiplication, and its comparison with other loci capturing split Galois representations via the identification $R^\ord=\T$ of Theorem \ref{thm:modularity}.

Throughout this section, $\chi$ is assumed to be quadratic. Then $\cO=\Zp$, $K$ is an imaginary quadratic field, $N=|\mathrm{Disc}(K)|$, and $p$ splits as $\gp_K\ob{\gp}_K$ in $K$.

\subsection{CM forms}\label{sec:CM}
Recall that a $p$-ordinary overconvergent cuspidal eigenform $g$ has complex multiplication by $K$ if $a_\ell(g)=\chi(\ell)a_\ell(g)$ for all but finitely many primes $\ell$. This amounts to the property that the $p$-adic Galois representation $\rho \colon G_\Q \to \GL_2(\ob{\Q}_p)$ is isomorphic to the induction $\Ind^K_{\Q} \nu$, where $\nu \colon G_K \to \ob{\Q}_p^\times$ is a character unramified at $\gp_K$. 

If $g$ is contributing to $\T$, then $\rho$ is minimally ramified in the sense of Definition \ref{def:ord}, so $\nu$ is unramified outside $\ob{\gp}_K$. Moreover, $\nu$ factors through a pro-$p$ quotient of $G_K$, since $\nu$ is residually trivial. In other words, the CM character $\nu$ of $g$ must factor through $\cX_K'$, the Galois group of the maximal abelian pro-$p$ extension of $K$ unramified outside $\ob{\gp}_K$. Furthermore,
\[g(q)=\sum_{\substack{0\neq\ga \subseteq \cO \\ \ob{\gp}_K\nmid \ga}} \nu(\ga)q^{\mathbf{N}(\ga)}\]
where $\mathbf{N}$ is the norm map of $K/\Q$. 

Let $\tilde{\T}^{\CM}:=\Zp\lsem \cX_K'\rsem$ and consider the universal character
\[\boldsymbol{\nu} \colon G_K \twoheadrightarrow \cX_K' \hookrightarrow (\tilde{\T}^{\CM})^\times. \]
We endow $\tilde{\T}^{\CM}$ with the structure of a $\Lambda$-algebra induced by the natural transfer map $\Gamma \to \cX_K \to \cX_K'$.
The universal CM form 
\[\mathbf{G}(q)=\sum_{\substack{0\neq\ga \subseteq \cO_K \\ \ob{\gp}_K\nmid \ga}} \boldsymbol{\nu}(\ga)q^{\mathbf{N}(\ga)}\]
is a $\tilde{\T}^{\CM}$-adic cuspidal eigenform, which induces a map $\T \to \tilde{\T}^{\CM}$ of $\Lambda$-algebras. We call $\T^{\CM}$ its image. Note that the specialization of $\boldsymbol{\nu}$ at the trivial character specializes $\mathbf{G}$ to $f$.

The next proposition characterizes Galois representations attached to CM forms in terms of their splitting behavior on $G_{\Qp}$. For $k\in\Z$, let $\T^{\CM}_k=\T^{\CM}/(1+X-(1+p)^{k-1})$ be the weight $k$ specialization of $\T^{\CM}$. Recall from  \S\ref{sec:split_CM} that $R^{1-\spl}$ represents $\cD^{1-\spl}$.

\begin{proposition}\label{prop:modularity_1_spl}\ 
	\begin{enumerate}
		\item Assume $A'_K=\{0\}$. Then there is a natural isomorphism $R^{1-\spl} \simeq \T^{\CM}$ in $\mathcal{CNL}_{\cO}$. 
		\item Assume $A'_K=\{0\}$. Then $\T=\T^{\CM}$ if and only if $\lambda(u_\chi)$ is a $p$-adic unit.
		\item  Assume $A'_K\neq\{0\}$. Then $\T\neq \T^{\CM}$. 
	\end{enumerate}
\end{proposition}
\begin{proof}
	Let us prove (1), and assume $A'_K=\{0\}$. Then $\T^{\CM}$ equals $\tilde{\T}^{\CM}$ and has an explicit presentation over $\Lambda$. Indeed, we know from Lemmas \ref{lem:lift_characters}(2) and \ref{lem:calcul_ord} that $\cX'_K$ is generated as a $\Zp$-module by $\phi$ and $\ob{\gamma}:=\tau\gamma\tau$, with unique relation of the form $\phi^{p^\ell}=\ob{\gamma}^x$ for $\ell=\length A_K$ and some $x \in \Zp$. 
	
	As $[\ob{\gamma}]=[\gamma \ob{\gamma}]=1+X$, $\tilde{\T}^{\CM}$ is generated over $\Lambda$ by the image $U:=\boldsymbol{\nu}(\phi)=\boldsymbol{\nu}(\gp_K)\in \T^{\CM}$ of $U_p\in\T$. This proves 
\[	\T^{\CM}=\tilde{\T}^{\CM}\simeq \Lambda[U]/(U^{p^\ell}-(1+X)^x).\]
		
	To construct a surjection $R^{1-\spl} \twoheadrightarrow \T^{\CM}$, we let 
	\[\rho^{\CM}=\Ind_K^\Q \boldsymbol{\nu} \colon G_\Q \to \GL_2(\T^{\CM}) \]
	be the Galois representation attached to $\mathbf{G}$, computed in a dihedral basis. Then $\rho^{\CM}(\tau)=(\begin{smallmatrix} 0 & 1 \\ 1 & 0\end{smallmatrix})$, and swapping the two vectors of the basis if necessary, $\rho^{\CM}_{|G_K}=(\begin{smallmatrix} \boldsymbol{\nu}^\tau & 0 \\ 0 & \boldsymbol{\nu}\end{smallmatrix})$. 
	
	We let $\rho = P_1 \cdot \rho^{\CM} \cdot  P_1^{-1}$. Then one readily checks that $\rho(\tau)=(\begin{smallmatrix} -1 & 0\\ 0 & 1\end{smallmatrix})$. Moreover, using that $G_{\Qp} \subset G_K$, the $\T^{\CM}$-line  $\langle(\begin{smallmatrix}1 \\ 1 \end{smallmatrix})\rangle$ is an ordinary line for $\rho$, and in fact, $(\rho,\langle(\begin{smallmatrix}1 \\ 1 \end{smallmatrix})\rangle)\in\cD^{1-\spl}(\T^{\CM})$. 
	
	We deduce the existence of a homomorphism $\varphi \colon R^{1-\spl}\to \T^{\CM}$ of $\Lambda$-algebras. The generator $\boldsymbol{\nu}(\gp_K)$ of $\T^{\CM}$ over $\Lambda$ is the image under $\varphi$ of the eigenvalue of $\phi$ acting on the ordinary quotient of $V_{\rho^{1-\spl}}$, where $V_{\rho^{1-\spl}}$ denotes the underlying space of the universal $1$-split Galois representation $\rho^{1-\spl}$. Hence, $\varphi$ is surjective.
	
	Let $\pi_1 \colon \T^{\CM}_1=\T^{\CM}/(X) \twoheadrightarrow \Zp$ be the morphism attached to $f$. Its associated Galois representation  is the specialization of $\rho^{\CM}$ at the trivial character, \textit{i.e.}, it is given by $\rho_1$ as in \eqref{eq:rho_f}. 
	
	As $\cX'_K/\langle \ob \gamma \rangle\simeq A_K$, we have $\T_1^{\CM}\simeq \Zp[A_K]$ as augmented $\Zp$-algebras, in the sense that $\pi_1$ is identified with the augmentation morphism $\Zp[A_K] \twoheadrightarrow \Zp$. Moreover, $A_K$ is cyclic as $A'_K=\{0\}$, so $\T_1^{\CM}=\Zp[A_K]\simeq \Zp\lsem t \rsem/((1+t)^{p^{\ell}}-1)$. In particular, $\T_1^{\CM}$ is a finite flat local complete intersection over $\Zp$, and 
	\[C^1(\pi_1,\T_1^{\CM})=(t) / (t^2,(1+t)^{p^{\ell}}-1)\simeq A_K.\]
	Applying Wiles' numerical criterion (Theorem \ref{thm:numerical_criterion}) to the identity map on $\T_1^{\CM}$ shows that both $C^0(\pi_1,\T_1^{\CM})$ and $A_K$ have the same length. On the other hand, Proposition \ref{prop:tangent_space_split} and Lemma \ref{lem:calcul_ord} show that $C^1(\pi_1,R^{1-\spl}_1)$ has the same length as $A_K$. Therefore, $R^{1-\spl}_1\simeq \T_1^{\CM}$ by Theorem \ref{thm:numerical_criterion}. As $\T^{\CM}$ is $\Lambda$-free, $\varphi$ is an isomorphism by Nakayama's lemma, hence proving (1).

	To prove (2), we apply Theorem \ref{thm:numerical_criterion} to the surjection $\T_1 \twoheadrightarrow \T_1^{\CM}$ of augmented $\cO$-algebras: it is an isomorphism if and only if  $C^1(\pi_1,\T_1)$ and $A_K$ have the same length. But Theorem \ref{thm:modularity} and its proof, which apply as $\chi$ is quadratic and $A_K'=\{0\}$ by assumption, show that $\length C^1(\pi_1,\T_1) = \length \Zp/(\cL_\chi\cdot L(\chi,0)^2)$. Applying Lemma \ref{lem:calcul_ord} and Theorem \ref{thm:refined_CNF} then proves that this length is $\length \Zp/(\lambda(u_{\chi}))\times A_K$. The equivalence to prove in (2) follows immediately.
		
	Assume $A_K'$ is non-trivial. Then we may construct as in the proof of Theorem \ref{thm:converse_gorenstein} a $p$-irregular $p$-stabilized theta series $\theta^{(p)}$ of weight $1$ contributing to $\T$. By \cite[Theorem 3.3(ii)]{betinadimitrovKatz}, $\theta^{(p)}$ belongs to an irreducible component of the eigencurve $\cC$ that does not have complex multiplication by $K$. In particular, $\T\neq \T^{\CM}$, which proves (3).
\end{proof}

\subsection{Strong splitness and the Coleman--Greenberg conjecture}
We recall the content of the Coleman--Greenberg conjecture, specialized to the setting of this paper.

\begin{conjecture}\label{conj:CG}
	Let $g$ be a classical eigenform of integer weight $k\geq 2$ whose system of Hecke eigenvalues appears in $\T_k$. If the restriction to $G_{\Qp}$ of the $p$-adic Galois representation attached to $g$ splits as the sum of two characters, then $g$ has complex multiplication.
\end{conjecture}
Assume $A'_K=\{0\}$ and let $\cO'$ be the finite (totally ramified) extension of $\Zp$ generated by the Fourier coefficients of $g$. By Lemma \ref{lem:weak_ordinarity}, 
the Galois representation $\rho$ of $g$ defines an $\cO'$-point of $\cD_k^\ord$. Conjecture \ref{conj:CG} then says that, if $\rho$ contributes to $R^{\mu-\spl}_k$ for some $\mu\in\cO'$, then $g$ contributes to $\T^{\CM}_k$. 

We will employ Galois-theoretic arguments which prove, under certain hypotheses, a strong form of Conjecture \ref{conj:CG}. The main hypothesis, along with $A'_K=\{0\}$, will be the following.

\begin{definition}\label{def:strongly_split}
	Let $\cO'$ be a finite totally ramified extension of $\Zp$ with uniformizer $\varpi$. Fix $n\in \N \cup\{\infty\}$ and put $A=\cO'/(\varpi^n)$, where $\varpi^n=0$ if $n=\infty$. Let $(\rho,L)\in\cD^{\ord}(A)$ be an ordinary Galois representation. Call $\alpha\in 1+\varpi A$ the eigenvalue of the Frobenius element $\phi\in G_\Qp$ acting on the ordinary quotient $V_\rho/L$, and $\beta\in 1+\varpi A$ the value $\det(\rho)(\gamma)$.
	
	We say that $\rho$ is \emph{strongly split} on $G_\Qp$ if it splits on $G_\Qp$ as the sum of two $A$-valued characters, and $\alpha-1$ divides $\beta-1$ in $A$.
\end{definition}
Our convention is that, if $0$ divides $\beta-1$, then $\beta=1$. This definition connects to induced representations in the following way.
\begin{lemma}\label{lem:strongly_split_reducible_case}
	Let $(\rho,L)\in\cD^{\ord}(A)$ with $A$ as in Definition \ref{def:strongly_split}, and assume that $\rho\simeq \Ind_K^\Q\nu$ for some character $\nu \colon \cX'_K \to 1+\varpi A$. Also assume that the valuation $v_p(\beta-1)$ of $\beta-1$ is not of the form $v_p(\zeta-1)$ for some root of unity $\zeta$.
	
	If the $\cL$-invariant $\cL_\chi\in\Qp$ has non-positive $p$-adic valuation, then $\rho$ is strongly split on $G_\Qp$.
\end{lemma}
\begin{proof}
	Without loss of generality, we assume $\rho$ is computed in the Eisenstein basis. The proof of Proposition \ref{prop:modularity_1_spl}(1) shows that $\rho$ is $1$-split, and $P_1^{-1}\cdot \rho \cdot P_1=(\begin{smallmatrix}
		\nu^\tau & 0 \\ 0 & \nu	\end{smallmatrix})$ on $G_K$. In particular, $\alpha=\nu(\phi)$, and $\beta=\nu^\tau(\gamma)=\nu(\ob{\gamma})$. By Lemma \ref{lem:lift_characters}(2), $\beta^x=\alpha^{2y}$, where $x=\lambda(u_{\chi})$ and $y=\ord_p(u_{\chi})$. As $\cL_\chi=-x/y$, the assumption on the $\cL$-invariant then says that $x$ divides $y$ in $\Zp$.
		
		Write $\beta$ as $\zeta\cdot \alpha^{2y/x}$ for some root of unity $\zeta$. We note that 
		\[v_p(\beta-\zeta)=v_p(\alpha^{2y/x}-1)\geq v_p(\alpha-1).\]
		On the other hand, $v_p(\beta-\zeta) = \min(v_p(\beta-1),v_p(\zeta-1))$ by our assumption on $v_p(\beta-1)$. Hence, $v_p(\beta-1)\geq v_p(\alpha-1)$, as wanted.
\end{proof}

We now state and prove the main theorem of this section.

\begin{theorem}\label{thm:CG}
	Assume $\chi^2=1$ and $A'_K=\{0\}$. Let $\cO'\in\mathcal{CNL}_{\Zp}$ be a finite extension of $\Zp$ with uniformizer $\varpi$. Fix $n\in \N \cup\{\infty\}$ and let $A=\cO'/(\varpi^n)$, where $\varpi^n=0$ if $n=\infty$. Finally let $(\rho,L)\in\cD^{\ord}(A)$.

	If $\rho$ is strongly split on $G_\Qp$, then $\rho(G_K)$ is abelian.
\end{theorem}

We henceforth assume $A'_K=\{0\}$ and we fix $(\rho,L)\in\cD^{\ord}(A)$ that is strongly split on $G_\Qp$ as in Theorem \ref{thm:CG}. Since $A$ is a quotient of $\cO'$, the discussion following Definition \ref{def:split} shows the existence of $\mu\in\cO'$ such that $\rho$ is $\mu$-split.

If we knew that $\mu=1$, then the result would immediately follow from Proposition \ref{prop:modularity_1_spl}(1). Our proof consists in showing that $\mu=1$ by successive approximations, unless $\rho\sim\rho_1=(\begin{smallmatrix}
	\chi & 0 \\ 0 & 1
\end{smallmatrix})$, in which case $\rho(G_K)$ is clearly abelian.

The following additional notation will be in force. Unlike in the rest of the paper, we will assume $\rho$ is computed in the $\mu$-split basis: $\rho=P^{-1}_\mu\cdot \rho^\textrm{Eis} \cdot P_\mu$, where $\rho^\textrm{Eis} \colon G_\Q \to \GL_2(A)$ is our representation computed in the Eisenstein basis. Recall from \eqref{eq:conjugation_tau_split} the simple identity 
\begin{equation}\label{eq:value_rho_tau}
t:=\rho(\tau)=\begin{pmatrix}
	\mu-1 & \mu(2-\mu) \\ 1 & 1-\mu
\end{pmatrix}.
\end{equation}
Let $\Gamma_K$ be the Galois group of the maximal pro-$p$ extension of $K$ unramified outside $p$, and let $\Gamma'_K=\ob{[\Gamma,\Gamma]}$ be the closure of its derived subgroup. As $\rho$ is minimally ramified and $\rho_{|G_K}$ is trivial modulo $\varpi$, $\rho_{|G_K}$ factors through $\Gamma_K$. We see our distinguished elements $\phi,\gamma\in G_{\Qp}$ as elements of $\Gamma_K$; we also let $\ob{\gamma}=\tau\gamma\tau$ and $\ob{\phi}=\tau\phi\tau$.

Let $x=\lambda(u_\chi)\in\Zp$ and $y=\ord_p(u_\chi)\in\Zp$, where $u_\chi$ is any normalized Gross--Stark unit. We start with the following useful observation.

\begin{lemma}\label{lem:topological_generators}
	$\Gamma_K$ is topologically generated by $\gamma,\ob{\gamma},\phi$, and we have the relations $\gamma^x\ob{\gamma}^{-x}\phi^{2y}\in\Gamma'_K$, and $\phi\ob{\phi}\in\Gamma_K'$. 
\end{lemma}
\begin{proof}
	The abelianization $\Gamma_K/\Gamma'_K$ is canonically isomorphic to the $\Zp$-module $\cX_K$ of \S\ref{sec:galois_cohomology}. By Lemma \ref{lem:lift_characters}(1), $\cX_K$ is generated as a $\Zp$-module by $\gamma,\ob{\gamma},\phi\in\cX_K$, and we have the relations $\gamma^x\ob{\gamma}^{-x}\phi^{2y}=1$ and $\phi\ob{\phi}=1$ in $\cX_K$. 
	
	We now make use of the analogue in \cite[Chap. I \S4. Proposition 23bis]{serrecohomologie} of Nakayama's lemma for pro-$p$ groups:  $\Gamma_K$ is topologically generated by $\gamma,\ob{\gamma},\phi$, because the same is true with its abelianization $\cX_K$. The lemma follows easily.
\end{proof}

Note that we have the following explicit description of the values taken by $\rho$:
\begin{equation}\label{eq:matrices_split}
	\begin{array}{lcr}
	g:=\rho(\phi) = \begin{pmatrix} \alpha^{-1} & 0 \\ 0 & \alpha \end{pmatrix}, && \rho(\ob{\phi}) = t g t, \\
		h:=\rho(\gamma) = \begin{pmatrix} \beta & 0 \\ 0 & 1 \end{pmatrix}, && \rho(\ob{\phi}) = t h t.
\end{array}
\end{equation}
The calculation of the upper-left entry of $\rho(\phi)$ uses that $\varepsilon_{\cyc}(\phi)=1$ by our choice of $\phi$. Notice that, if $\mu=1$, then $t=(\begin{smallmatrix}
	0 & 1 \\ 1 & 0\end{smallmatrix})$, and $tgt=g^{-1}$ and $tht=\beta\cdot h^{-1}$.

Without loss of generality, we may assume that $\alpha\neq 1$. Indeed, if $\alpha=1$, then $\beta=1$ by strong splitness of $\rho$, forcing $\rho$ to be trivial on $G_K$. Write $\alpha,\beta$ as
\[\alpha=1+\pi, \qquad \beta=1+\pi'.\]
Hence, $\pi,\pi'\in A$ have $\varpi$-adic valuation between $1$ and $n$, $\pi$ is nonzero and divides $\pi'$. In particular, $\rho$ induces a continuous map
\[\rho_{|G_K} \colon \Gamma_K \to 1+\pi \rM, \]
where $\rM=M_2(A)$, and $P:=\rho(\Gamma_K)$ is a pro-$p$ group. A simple but important consequence is that
\begin{equation}\label{eq:def_P'}
	P':= \rho(\Gamma_K') \subseteq 1+\pi^2\rM.
\end{equation}
Our ultimate goal is to prove that this subgroup is trivial. 

\begin{proposition}\label{prop:proof_CG}
	Let $\nu=\mu-1\in A$. For any integer $\ell\geq1$, if $P' \subseteq 1+\pi^\ell \rM$, then $\pi^{2\ell}$ divides $\pi\nu$ in $A$ and, in fact, $P' \subseteq 1+\pi^{2\ell} \rM$. 
\end{proposition}

\begin{proof}
	The proof will only use that $\pi$ divides $\pi'$ and that $gtgt\in P'$ by Lemma \ref{lem:topological_generators}. Fix an integer $\ell\geq1$ such that $P' \subseteq 1+\pi^\ell \rM$. The map sending $1+E\in 1+\pi^\ell \rM$ to $E$ induces a continuous morphism
	\begin{equation}\label{eq:map_from_P'_to_additive}
		\delta \colon P' \to \pi^\ell\rM/\pi^{2\ell}\rM,
	\end{equation}
	where the target is seen as an (additive) finitely generated $\Zp$-module. We will describe the image of $P'$ and $gtgt \in P'$ under $\delta$.
	
	First notice that, by Lemma \ref{lem:topological_generators}, $P'$ is topologically generated as a pro-$p$ group by elements of the form $[d,td't]$, where $d,d'$ are of the form $g^k$ or $h^m$, where $k,m\in\Zp$ are arbitrary. We may replace $h$ in $h^m$ with $\beta^{-1/2}h=\mathrm{diag}(\beta^{1/2},\beta^{-1/2})$, since commutators are invariant under homothety.
	
	Let $d=\mathrm{diag}(a^{-1},a)$ and $d'=\mathrm{diag}(b^{-1},b)$, where $a,b$ are of the form $\alpha^{k}$ or $\beta^{-m/2}$ for some $k,m\in\Zp$. Write $a$ and $b$ as $a=1+\pi_a$, $b=1+\pi_b$. Then $\pi_a,\pi_b$ are divisible by $\pi$ or $\pi'$, and hence by $\pi$ because $\pi$ divides $\pi'$. We have
\begin{equation}\label{eq:dtd't}
	dtd't=\frac{1}{ab}\begin{pmatrix}
		1-(\nu^2-1)(b^2-1) & \nu(\nu^2-1)(b^2-1) \\
		-a^2\nu(b^2-1) & a^2(1+\nu^2(b^2-1))
	\end{pmatrix}.
\end{equation}
For $d=d'=g$, that is, $a=b=\alpha$, one already obtains
\begin{equation}\label{eq:gtgt}
\delta(gtgt)=gtgt - 1 \sim_{2\ell} \begin{pmatrix} * & * \\ \pi\nu & *\end{pmatrix},
\end{equation}	
where $E\sim_{2\ell} F$ means that the entries $E_{ij}$ and $F_{ij}$ have the same valuation modulo $\pi^{2\ell}$ for all $1\leq i,j\leq2$. Here, $c,c'\in \cO'$ are said to have the same valuation modulo $\pi^{2\ell}$ if either they are both divisible by $\pi^{2\ell}$, or they have the same $\varpi$-adic valuation (less than the one of $\pi^{2\ell}$). 

Since $gtgt\in P'$, we already see from \eqref{eq:def_P'} that $\pi^2$ divides $\pi\nu$, so $\nu$ is not a unit, and hence, $\nu^2-1$ is a unit.
This remark and the identity \eqref{eq:dtd't} allow us, for general $d,d'$, to prove by a direct calculation of $[d,td't]=(dtd't)(d^{-1}t(d')^{-1}t)$ that
\begin{equation}\label{eq:commutator}
	\delta([d,td't])=[d,td't]-1 \sim_{2\ell} \begin{pmatrix}
		\pi_a\pi_b^2\nu^2 & \pi_a\pi_b\nu \\ 
		\pi_a\pi_b\nu & \pi_a\pi_b^2\nu^2
	\end{pmatrix}.
\end{equation}
Now, as the map $\delta$ is additive, it follows from \eqref{eq:gtgt} and \eqref{eq:commutator} that $\pi\nu$ mod $\pi^{2\ell}$ is a $\Zp$-linear combination of elements of the form $\pi_a\pi_b\nu$, for $a,b$ some powers of $\alpha$, $\beta$. As $\pi$ divides $\pi_a,\pi_b$, there is $c\in A$ such that $\pi\nu\equiv c\pi^2\nu \mod \pi^{2\ell}$, which implies that $\pi^{2\ell}$ divides $\pi\nu$. Hence, the coefficients of the right-hand-side of \eqref{eq:commutator} are divisible by $\pi^{2\ell}$, proving that $P' \subseteq 1+\pi^{2\ell}\rM$.
\end{proof}

\begin{proof}[Proof of Theorem \ref{thm:CG}]
	We keep the notation as before. By a repeated use of Proposition \ref{prop:proof_CG}, starting with $\ell=1$, we see that $P'=\{1\}$. Therefore, $\rho_{|G_K}$ factors through an abelian quotient of $G_K$.
\end{proof}

\begin{corollary}\label{coro:CG}
	Assume $\chi^2=1$ and $A_K'=\{0\}$, and let $g$ be an ordinary overconvergent cuspidal eigenform of arbitrary $p$-adic weight with $p$-adic Galois representation $\rho$. Assume that:
	\begin{enumerate}
		\item $g$ has tame level $N$, and
		\item $\tr(\ob{\rho})=\ob{\chi}+1$.
	\end{enumerate}
	If $\lambda(u_\chi)$ is a $p$-adic unit, or if $\rho$ is strongly split on $G_{\Qp}$, then $g$ has complex multiplication by $K$.
\end{corollary}
\begin{proof}
	Since $g$ has tame level $N$, $g$ contributes to a local factor of the Hida--Hecke algebra $\T(N)$ of tame level $N$. The assumption on the trace of $\ob{\rho}$ shows that, in fact, $g$ contributes to $\T$. Indeed, as $\ob{\chi}$ has conductor $N$, $g$ is minimally ramified, and the congruence condition between Fourier coefficients of $f$ and $g$ encoded in the maximal ideal $\gm$ in \eqref{eq:maximal_ideal_eisenstein} can be checked at $q\mid Np$ using \cite[Lemma 6.3(ii)]{dimitrov2014local}.
	
	If $\lambda(u_\chi)$ is a $p$-adic unit, then $\T=\T^{\CM}$ by Proposition \ref{prop:modularity_1_spl}(2), so there is nothing to prove. 
	
	By Proposition \ref{prop:surjection_R_to_T}, $\rho$ defines a point of $\cD^\ord$ valued in a finite totally ramified extension $\cO'$ of $\Zp$. If $\rho$ is strongly split on $G_{\Qp}$, then we deduce from Theorem \ref{thm:CG} that $\rho(G_K)$ is abelian, so $\rho\simeq\Ind_K^\Q\nu$ for some Hecke character $\nu$. This proves that $g$ has complex multiplication by $K$.
\end{proof}

\section{Congruence numbers and adjoint $p$-adic $L$-functions}\label{sec:trivial_zero_conjecture}
The goal of this section is to give further evidence to Conjecture \ref{conj:length_C0} through its connection with $p$-adic $L$-functions. Since the conjecture holds if $\chi$ is quadratic by Theorem \ref{thm:congruence_bound}(1), we assume in this section that $\chi$ is not quadratic for simplicity.

\subsection{Genuine $p$-adic $L$-functions}\label{sec:genuine}
Let $g$ be any $p$-ordinary classical eigenform of weight $k\geq2$ contributing to $\T$, and let $L_p(\ad g,s)$ be the adjoint $p$-adic $L$-function attached to $g$. It is a bounded $p$-adic analytic function of the variable $s\in\Zp$, defined as a shift of Schmidt's symmetric square $p$-adic $L$-function of $g$ (\cite{schmidt}), whose interpolation formula has the shape
\begin{equation}\label{eq:interpolation_L_p_ad}
	L_p(\ad g,j) = (*)\dfrac{L(\ad g \otimes \omega^{j-1},j)}{\pi^{2j+k-1}\langle g,g\rangle}, 
\end{equation}
for $2-k\leq j\leq 0$ even or $1\leq j\leq k-1$ odd. Here, $(*)$ is an explicit factor, $L(\ad g \otimes \psi,s)$ is the $L$-function of the Gelbart--Jacquet lift of $g$ twisted by a Dirichlet character $\psi$, and $\langle g,g\rangle$ the Petersson norm of $g$.
We refer to \cite[\S1.3]{dasguptapadicartin} for a precise definition of $L_p(\ad g,s)$. We will follow Dasgupta's normalization of $p$-adic $L$-functions, such as in \cite[Theorems 3.4 and 3.6]{dasguptapadicartin}.

Let $\cO_g$ be the extension of $\cO$ generated by the Fourier coefficients of $g$. When $\chi^2\neq1$, the integrality property of $L_p(\ad g,s)$ is clarified in the main theorem of \cite{hidaGL3}: there exists a power series $L_p(\ad g)(T)\in \Qp \otimes_{\Zp}\cO_g\lsem T\rsem$ such that 
\[L_p(\ad g,s)=L_p(\ad g)((1+p)^{s-1}-1)\]
for all $s\in\Z_p$. Moreover, $c^0_g\cdot L_p(\ad g)(T)\in \cO_g\lsem T\rsem$ for any congruence number $c^0_g\in\cO_g$ for the system $\pi_g \colon \T_k \otimes_{\cO}\cO_g \twoheadrightarrow \cO_g$ of Hecke eigenvalues of $g$ (see Definition \ref{def:C}).

In his monograph \cite{hidagenuine}, Hida pointed out that the ``correct'' $p$-adic $L$-function attached to the adjoint motive of $g$ as in the Coates--Perrin-Riou conjecture \cite{coatesperrinriou} should be 
\begin{equation}\label{eq:motivic_p_adic_L_g}
	L_p^\textrm{mot}(\ad g,s):=c^0_g\cdot L_p(\ad g,s),
\end{equation}
where $c^0_g\in\cO_g$ is any congruence number as above. While it is true that the multiplication by $c^0_g$ introduces an ambiguity (which is also present in the conjectural definition of motivic $p$-adic $L$-functions), Hida's early work on congruence numbers shows that $c^0_g$ is closely related to the ratio of the automorphic period $\pi^{2j+k-1}\langle g,g\rangle$ in \eqref{eq:interpolation_L_p_ad} by the product $\Omega_g^+\Omega_g^-$ of Shimura periods, the latter being the period appearing in the Coates--Perrin-Riou conjecture. We do not attempt to formulate a precise conjecture in our (residually Eisenstein, non-$p$-distinguished) setting but we refer to the introduction of \cite{maksoud_adjoint} for an elaboration on this remark.

\subsection{The exceptional zero formula in weight $1$}
Let $\cF$ be the unique Galois-conjugacy class of Hida families passing through $f$, as in \S\ref{sec:construction_congruences}. It has coefficients in the integral closure $\I$ of $\T/\ga$, where $\ga$ is the unique minimal ideal contained in $\gp_f$ (see Proposition \ref{prop:finite_length_C}). We let $\T_\I=\T \otimes_{\Lambda} \I$ and call $\pi\colon \T_\I \twoheadrightarrow \I$ the system of Hecke eigenvalues of $\cF$. As in \S\ref{sec:construction_congruences}, we let $\spe_f \colon \T_\I \to \cO$ be the specialization morphism attached to $f$. 

Let $L_p(\ad F,s)$ be the adjoint $p$-adic $L$-function attached by Hida to $F$ in \cite{hidaGL3} (again normalized as in \cite{dasguptapadicartin}). It interpolates $L_p(\ad g,s)$ for any specialization $g$ of weight $k\geq2$ of $F$. If $C^0(\pi)=\I/\eta_\cF$ is the congruence module of $\cF$ with $\eta_\cF=\pi(\mathrm{ann}_{\T_\I}(\ker\pi))$, then $L_p(\ad F,s)$ is of the form $L_p(\ad F)((1+p)^{s-1}-1)$ for some $L_p(\ad F)(T)\in \Frac(\I) \otimes_{\I} \I\lsem T\rsem$ with 
\[\eta_\cF \cdot L_p(\ad F)(T)\subseteq\I\lsem T\rsem.\]
By Proposition \ref{prop:finite_length_C}, the specialization of $\eta_\cF$ along $\spe_f$ is a nonzero ideal $(c^0_f)$ of $\cO$ with $c^0_f$ the congruence number of $f$, and we may define 
the adjoint $p$-adic $L$-function of $f$ as the specialization
\begin{equation}\label{eq:def:p:adic:L:ad:f}
	L_p(\ad f)(T):=L_p(\ad F)_{\big|\cF=f}(T)\in (c^0_f)^{-1} \cdot \cO\lsem T\rsem
\end{equation}
at $f$. We also define its ``motivic'' avatar as 
\[L_p^\textrm{mot}(\ad f)(T):=c^0_f \cdot L_p(\ad f)(T) \in \cO\lsem T\rsem, \]
and, for $s\in\Z_p$, we put 
\begin{equation}\label{eq:motivic_p_adic_L}
L_p^\textrm{mot}(\ad f,s):=L_p^\textrm{mot}(\ad f)((1+p)^{s-1}-1).
\end{equation}
By analogy with higher weight as in \eqref{eq:motivic_p_adic_L_g}, the $p$-adic $L$-function in \eqref{eq:motivic_p_adic_L} should be the correct candidate for the motivic adjoint $p$-adic $L$-function of $f$. The adjoint motive attached to $f$ decomposes as the sum 
\[\ad f = E \oplus E(\chi) \oplus E(\chi^{-1}), \]
so it is reasonable to expect that $L_p^\textrm{mot}(\ad f,s)$ factors as a product of three Kubota--Leopoldt $p$-adic $L$-functions, hence following a $p$-adic Artin formalism. We warn the reader that, however, the case of reducible motives is usually not covered in general conjectures on $p$-adic $L$-functions, and $p$-adic Artin formalism is \emph{a fortiori} not well understood (see, however, \cite[Remark 3.30]{himc} for a discussion around $p$-adic Artin formalism for Artin motives). We will justify that our expectations on $L_p^\textrm{mot}(\ad f,s)$ are compatible with Conjecture \ref{conj:length_C0}. 

In the following exceptional zero formula, we let, for $\varphi\in\{\chi,\chi^{-1}\}$, 
\[\cL(\varphi):=\log_p(1+p)\cdot \cL_\varphi=-\log_p(u_{\varphi^{-1}})/\ord_p(u_{\varphi^{-1}}) \]
be the usual $\cL$-invariant of $\varphi$, in contrast with the normalized $\cL$-invariant $\cL_\chi$. We also assume Conjecture \ref{conj:length_C0} and choose $c^0_f$ to be 
	\begin{equation}\label{eq:expr_c_f}
	c^0_f = \left(1-\dfrac{1}{p}\right) \cdot (\cL(\chi)+\cL(\chi^{-1}))\cdot L(\chi,0)\cdot L(\chi^{-1},0).
\end{equation}
This is indeed a valid choice of congruence number of $f$, as $\log_p(1+p)^{-1}$ and $(1-p^{-1})$ have the same $p$-adic valuation.

\begin{theorem}\label{thm:trivial_zero_formula}
	Assume $\chi^2\neq1$ as before, and assume Conjecture \ref{conj:length_C0} holds for $\chi$. Then $L_p^\textrm{mot}(\ad f,s)$ vanishes at $s=1$, and 
	\begin{equation}\label{eq:trivial_zero}
		\dfrac{\mathrm{d}}{\mathrm{d}s} L_p^\textrm{mot}(\ad f,s)_{\big|s=1}=\left(1-\dfrac{1}{p}\right)\cL(\chi)\cL(\chi^{-1})L(\chi,0)L(\chi^{-1},0),
	\end{equation}
	up to an explicit rational number that is a unit in $\cO$. 
	
	Moreover, the right hand side of \eqref{eq:trivial_zero} is also the first derivative at $s=1$ of the product $\zeta_p(s)L_p(\chi\omega,1-s)L_p(\chi^{-1}\omega,1-s)$ of the $p$-adic Riemann $\zeta$-function and Kubota--Leopoldt $p$-adic $L$-functions attached to $\chi\omega$ and $\chi^{-1}\omega$ respectively.
\end{theorem}
\begin{proof}
	The last statement follows from the residue formula $\res_{s=1}\zeta_p(s)=1-p^{-1}$ for $\zeta_p$ and the Ferrero--Greenberg formula for $\chi$ and $\chi^{-1}$ (which is recalled in \eqref{eq:ferrero_greenberg}).
	
	Given the choice of $c^0_f$ made in \eqref{eq:expr_c_f}, we are led to prove that $L_p(\ad f,s)$ vanishes at $s=1$ and
\begin{equation}\label{eq:formula_to_prove_L_p_ad}
	 L'_p(\ad f,1)\overset{?}{=}\dfrac{\cL(\chi)\cdot \cL(\chi^{-1})}{\cL(\chi)+\cL(\chi^{-1})}, 
\end{equation}
	up to an explicit element in $\Q^\times \cap \cO^\times$. We argue as in the proof by Dasgupta in \cite[Theorem 4]{dasguptapadicartin} of the exceptional zero formula for $\ad g$, with $g$ a weight $k\geq2$ member of $\cF$, and we let $k$ tend to $1$.
	
	By definition, one has $L_p(\ad\cF,s)=L_p(\Symm^2\cF\otimes\chi^{-1},s+k-1)$ with the symmetric square $p$-adic $L$-function being the one in Theorem 3.6 of \textit{loc. cit.}. As explained in \S1.3 of \textit{loc. cit.}, $L_p(\ad\cF,1)$ vanishes at any classical weight $k\geq2$ specialization $g$ of $\cF$, so vanishes at $f$ as well by density of classical weights. Furthermore, if one lets
	\[u:=-\dfrac{\prod_{\ell|N}(1-\ell^{-1})}{N\varphi(N)} \in \Q^\times \cap \cO^\times,\]
	then one has
\begin{equation}\label{eq:derivative_p_adic_L}
		L'_p(\ad \cF,1)_{\big|\cF=g} = \left(1-p^{-1}\right)^{-1} \cdot L_p(\cF,\cF_{\chi^{-1}},k)_{\big|\cF=g} =u\cdot\dfrac{a'_p(\cF)}{a_p(\cF)}\phantom{)}_{\big|\cF=g},
\end{equation}
	for any classical weight $k\geq2$ specialization $g$ of $\cF$ with nebentype $\chi$. Here, the $p$-adic $L$-function in the middle is Hida's three-variable Rankin--Selberg $p$-adic $L$-function of $\cF$ and its twist $\cF_{\chi^{-1}}$, evaluated at  $(g,g_{\chi^{-1}},k)$, and the first equality is (13) of \textit{loc. cit.}, which is a simple consequence of Dasgupta's factorization theorem. The second equality is stated in (78) of \textit{loc. cit.}, and $a_p'(\cF)_{\big|\cF=g}$ is the derivative at $g$ of the $p$-th Fourier coefficient of $\cF$.
	
	Note that $a_p(\cF)_{|\cF=f}=1$. To compute the limit of $a_p'(\cF)_{|\cF=g}$ when $g$ tends to $f$, we use $k$ as a local parameter of $\cF$ around $f$. More precisely, recall from Proposition \ref{prop:finite_length_C} that the localization-completion at the height $1$ prime $\gp_f$ of $\I$ attached to $f$ is étale over $\Lambda_{(X)}$, hence of the form $E\lsem X\rsem$. We obtain an injection of $\Lambda$-algebras $\I \hookrightarrow E\lsem X\rsem$ whose image lands in a ring of the form $\cO\lsem X/p^r\rsem$ for some integer $r\geq1$ by \cite[Lemme 4.18 and Proposition 4.19]{maks}. We may therefore specialize $\cF$ at $X=(1+p)^{k-1}-1$ for any $k\geq2$, $k\equiv1 \mod (p-1)p^r$, to obtain a cusp form $g$ of weight $k$ and nebentype $\chi$, and see $a_p(\cF)$ as an analytic function $a_p(k)$. Then $a_p'(\cF)_{\big|\cF=g}$ is, by definition, the derivative at $g$ of $a_p(k)$ with respect to $k$. We therefore obtain
	\begin{equation}\label{eq:derivative_a_p}
		a_p'(1)= \log_p(1+p)\cdot \dfrac{a_p(\cF)-1}{X}\phantom{.}_{\big|X=0} = \dfrac{\cL(\chi)\cdot \cL(\chi^{-1})}{\cL(\chi)+\cL(\chi^{-1})}, 
	\end{equation}
	the second equality being established in \cite[Proposition 5.8]{betinadimitrovpozzi}. Taking a limit $k \to 1$ in \eqref{eq:derivative_p_adic_L} then yields the equality \eqref{eq:formula_to_prove_L_p_ad} up to multiplication by the unit $u$, as wanted.
\end{proof}
\begin{remark}
	Without assuming Conjecture \ref{conj:length_C0}, the proof of Theorem \ref{thm:trivial_zero_formula} and, in particular Formulae \eqref{eq:derivative_p_adic_L} and \eqref{eq:derivative_a_p}, shows that 
	\[L_p(\cF,\cF_{\chi^{-1}},k)_{\big|\cF=f} = \dfrac{\log_p(u_{\chi})\cdot \log_p(u_{\chi^{-1}})}{\log_p(u_{\chi}) + \log_p(u_{\chi^{-1}})}\]
	up to an explicit rational number. The left-hand side of this expression computes the $p$-adic iterated integral attached to the pair $(f,f^*)$ of Eisenstein series, where $f^*=E_1^{(p)}(1,\chi^{-1})$, in Darmon--Lauder--Rotger's variant of the elliptic Stark conjecture \cite[Conjecture 1.2]{DLR2} (see Lemma 4.3 of \textit{loc. cit.}). As for the right-hand side, it is equal up to a rational factor to the $p$-adic period $\Omega_{f,f^*}$ attached to $(f,f^*)$ in \cite{DLR2} (which is made explicit in (11) of \textit{loc.cit.}). This proves that Darmon--Lauder--Rotger's conjecture holds for the pair $(f,f^*)$.
	\end{remark}

\section{Examples}\label{sec:examples}
We illustrate with examples and numerical data the main results and conjectures that have been discussed in the previous sections. They all rely on Wiese's algorithm \cite{wiesealgorithm}, implemented in Magma, that allows us to explicitly compute the Artinian $k$-algebra 
\[\ob{\T}:=\T/(X,p)=\T_1/(p) \]
for a given Dirichlet character $\chi$.
\subsection{Numerical data}\label{sec:numerical_data}

\subsubsection{Embedding dimension and Gorensteinness}
Denote by $\EmbDim \ob{\T}$ the embedding dimension of $\ob{\T}$ at its maximal ideal $\ob{\gm}$. By definition, we have
\[\EmbDim\ob{\T} = \dim_k \ob{\gm}/\ob{\gm}^2. \] 
Hence, if $n=\EmbDim \ob{\T}$ then $\ob{\T}\simeq k\lsem x_1,\ldots,x_n\rsem/(g_1,\ldots g_m)$ with $m\geq n$ and none of the $x_i$'s or $g_j$'s can be eliminated. Moreover, $n=m$ if  $\ob{\T}$ is a local complete intersection (l.c.i.), which is automatically the case if $n\leq1$. Also note that $\ob{\T}$ cannot be l.c.i. if $\ob{\T}$ is not Gorenstein.

Based on Corollary \ref{coro:structure_T_1_chi_quadratic} and Theorem \ref{thm:converse_gorenstein}, we counted in Table \ref{table1}  the number of imaginary quadratic fields $K$ with fundamental discriminant $-N$, $1<N<100\,000$, in which $p=3$ splits, $p\nmid \phi(N)$, and with prescribed values of $\EmbDim \ob{\T}$.

\begin{table}[htbp]
	\centering
	\caption{$\chi^2=1$, $p=3$ and $N < 100\,000$}\label{table1}
	\begin{tabular}{lr}
		\toprule
		\textbf{Property} & \textbf{Number of occurrences} \\
		\midrule
		$\EmbDim \ob{\T}=0$ & 1\,709 \\
		$\EmbDim \ob{\T}=1$ & 1\,635 \\
		$\EmbDim \ob{\T}=2$ and $\ob{\T}$ l.c.i. & 339 \\
		non-Gorenstein $\ob{\T}$ & 446 \\
		Any $\ob{\T}$ & 4\,129 \\ 
		\midrule
		\textbf{Percentage of Gorenstein $\ob{\T}$} & 89.20\,\% \\
		\bottomrule
	\end{tabular}
\end{table}

We also mention the following result by Wood \cite{wood}, building on classical work of Davenport--Heilbronn \cite{davenportheilbronn}.

\begin{proposition}
	Let $p=3$, and call $M(X)$ (resp $M'(X)$) the set of imaginary quadratic fields $K$ of absolute discriminant less than $X$ in which $p$ splits (resp. and such that $A'_K=\{0\}$). Then 
	\[\liminf_{X \to \infty}\#M'(X)/\#M(X)\geq5/6.\]
\end{proposition}
\begin{proof}
	For any imaginary quadratic field $K$ in which $p$ splits as $\gp_K\ob{\gp}_K$, let $N_K$ be the number of surjections $A_K \twoheadrightarrow \Z/3\Z$ sending $[\gp_K]$ to $0$. Then $N_K=0$ if and only if $A_K'=\{0\}$, and $N_K\geq2$ otherwise. It follows from \cite[Theorem 1.6]{wood} (and that $[\gp_K][\ob{\gp}_K]=1$ in $A_K$) that $(\#M(X))^{-1}\cdot \sum_{K\in M(X)}N_K$ tends to $1/3$ when $X \to \infty$. The sum is bounded below by $2 \cdot \#( M(X)\setminus M'(X))$, so the result follows.
\end{proof}

More generally, Cohen--Lenstra heuristics with local conditions \cite[Conjecture 1.5]{wood} predict that, for any odd $p$, the density $\#M'(X)/\#M(X)$ should tend to $\prod_{i\geq2}(1-p^{-i})$. In particular, this density is around $84\%$ for $p=3$.

\subsubsection{Ranks} It would be interesting to find a simple criterion for determining when the $k$-dimension of $\ob{\T}$ is as small as possible, assuming $\EmbDim \ob{\T}\geq1$. We present in Tables \ref{table2} and \ref{table3} the values of the $k$-dimension for $p=3$, $K$ quadratic of fundamental discriminant $N<5\,000$, focusing on the case when $\EmbDim\ob{\T}\in\{2,3\}$. We note that, in this quadratic setting, $\ob{\T}$ is always l.c.i. if $\EmbDim \ob{\T}=2$, and never Gorenstein if $\EmbDim\ob{\T}\geq3$. We will present in \S\ref{ex:f=181} a counterexample to this observation when $\chi^2\neq1$.

\begin{table}[htbp]
	\centering
	\caption{$\chi^2=1$, $p=3$ and $N<5\,000$: list of $N$'s with $\EmbDim \ob{\T}=2$}\label{table2}
	\begin{tabular}{c c l}
		\toprule
		\textbf{$\dim_k\ob{\T}$} & \textbf{Number of $N$'s} & \textbf{$N$} \\
		\midrule
		5  & 6 & 107, 503, 1583, 1619, 2036, 4283 \\
		7  & 2 & 2468, 2879 \\
		9  & 3 & 863, 1427, 2915 \\
		11 & 3 & 419, 1871, 3359 \\
		13 & 1 & 1319 \\
		\bottomrule
	\end{tabular}
\end{table}

\begin{table}[htbp]
	\centering
	\caption{$\chi^2=1$, $p=3$ and $N<5\,000$: list of $N$'s with $\EmbDim \ob{\T}\geq3$}\label{table3}
	\begin{tabular}{c c c l}
		\toprule
		\textbf{$\EmbDim \ob{\T}$} & \textbf{$\dim_k \ob{\T}$} & \textbf{Number of $N$'s} & \textbf{$N$} \\
		\midrule
		3 & 5  & 11 & 239, 743, 1844, 1931, 2591, 2687, \\
		&    &    &  2819, 3119, 4091, 4523, 4703 \\
		3 & 7  & 1  & 2423 \\
		3 & 11 & 4  & 971, 2063, 3716, 4388 \\
		\midrule
		5 & 43 & 1  & 3299 \\
		\bottomrule
	\end{tabular}
\end{table}

\subsection{Mod $p$ Hecke algebras}
We illustrate our results by explicitly describing $\ob{\T}$ in some examples where $\EmbDim \ob{\T}\geq2$. We recall that in the non-quadratic examples below, the Hecke algebras $\ob{\T}$ defined for $\chi$ and $\chi^{-1}$ are isomorphic by Remark \ref{rem:invariance_twist_chi}.

\subsubsection{$K=\Q(\sqrt{-107})$ and $p=3$}
In this quadratic case, one finds 
\[\ob{\T}\simeq \F_3\lsem x,y\rsem/(x^3+y^2,xy),\]
which is a local complete intersection. The relevant class groups are $A'_K=\{0\}$, $A_K=\Z/3\Z$. The relevant valuations are $v_p(\lambda(u_{\chi}))=2$ and $v_p(\ord_p(u_{\chi}))=1$, so $v_p(\cL_{\chi})=1$.

\subsubsection{$K=\Q(\sqrt{-239})$ and $p=3$}
In this quadratic case, one finds
\[\ob{\T}\simeq \F_3\lsem x,y,z\rsem/(x^2 - z^2,xy,xz - z^2,y^2,yz,z^3),\] 
which is not Gorenstein: its socle $\ob{\T}[\ob{\gm}]$ has $\F_3$-dimension greater than $1$ and, in fact, has basis $\{y,x-z,z^2\}$. The class groups are $A_K=A'_K=\Z/3\Z$, and $v_p(\lambda(u_{\chi}))=0$ and $v_p(\ord_p(u_{\chi}))=0$, so $v_p(\cL_{\chi})=0$.

\subsubsection{$K=\Q(\sqrt{-362+18\sqrt{181}})$ and $p=5$}\label{ex:f=181}
For $\chi$ one of the two $\Z_5$-valued quartic characters of conductor $N=181$, one finds
\[\ob{\T}\simeq \F_5\lsem x,y\rsem/(x^2,xy,y^2),\] 
which is not Gorenstein. The class groups are $A_K=\Z/25\Z$ and $A'_K=\Z/5\Z$ respectively. Since the intermediate field $F=\Q(\sqrt{181})$ has class number one, one sees that, exchanging $\chi$ and $\chi^{-1}$ if necessary, $A_K[\chi]=A_K$ and $A'_K[\chi]=A'_K$, whereas $A_K[\chi^{-1}]=A'_K[\chi^{-1}]=\{0\}$. As for the other numerical invariants, one finds 
\[v_p(\lambda(u_{\chi}))=0, \quad v_p(\lambda(u_{\chi^{-1}}))=2, \quad v_p(\ord_p(u_{\chi}))=1, \quad v_p(\ord_p(u_{\chi^{-1}}))=0, \]
so $\cL_{\chi}$, $\cL_{\chi^{-1}}$ and $\cL_{\chi}+\cL_{\chi^{-1}}$ have respective valuations $2,-1,-1$.

\subsubsection{$K=\Q(\sqrt{-3530+66\sqrt{1765}})$ and $p=3$}
For $\chi$ one of the two $\Z_9$-valued quartic characters of conductor $N=1765$, one finds
\[\ob{\T}\simeq \F_9\lsem x,y,z\rsem/(x^2, xy, xz, y^3, yz, z^2),\] 
which is not Gorenstein. Letting $F=\Q(\sqrt{1765})$, we have $A_F=\Z_9/(3)$ and $A_K=(\Z_9/(3))^3$. An inspection of the action of $\Gal(K/\Q)$ over $A_K$ reveals that $A_K[\chi^{\pm1}]=A'_K[\chi^{\pm1}]=\Z_9/(3)$. Moreover,
\[v_p(\lambda(u_{\chi}))=v_p(\lambda(u_{\chi^{-1}}))=v_p(\ord_p(u_{\chi}))=v_p(\ord_p(u_{\chi^{-1}}))=0, \]
and $\cL_{\chi}$, $\cL_{\chi^{-1}}$ and $\cL_{\chi}+\cL_{\chi^{-1}}$ have respective valuations $0,0,1$.

\subsubsection{$K=\Q(\sqrt{-778+7\sqrt{778}})$ and $p=3$}
For $\chi$ one of the two $\Z_9$-valued quartic characters of conductor $N=6224$, one finds
\[\ob{\T}\simeq \F_9\lsem x,y\rsem/(x^2+y^4, xy),\] 
which is Gorenstein. Letting $F=\Q(\sqrt{778})$, we have $A_F=\{0\}$ and $A_K=(\Z_9/(3))^2$. An inspection of the action of $\Gal(K/\Q)$ over $A_K$ reveals that $A_K[\chi^{\pm1}]=\Z_9/(3)$, whereas $A'_K[\chi^{\pm1}]=\{0\}$. Moreover,
\[v_p(\lambda(u_{\chi}))=v_p(\lambda(u_{\chi^{-1}}))=v_p(\ord_p(u_{\chi}))=v_p(\ord_p(u_{\chi^{-1}}))=1, \]
and $\cL_{\chi}$, $\cL_{\chi^{-1}}$ and $\cL_{\chi}+\cL_{\chi^{-1}}$ have respective valuations $0,0,0$. 

\subsection*{Acknowledgments}
I began working on exceptional zero formulas for weight $1$ forms in 2022, and I would like to thank the organizers of the conference in ``$p$-adic $L$-functions and Eigenvarieties'' at the University of Notre Dame, where a very preliminary version of this work could be presented. I would also like to thank Adel Betina, Mike Daas, Shaunak Deo, Mladen Dimitrov, Emmanuel Lecouturier, Lo\"ic Merel, Alice Pozzi, Rob Pollack, and Jun Wang for stimulating conversations from which this work greatly benefited. I am especially thankful to Gabor Wiese for fruitful discussions around theoretical and computational aspects of Hecke algebras. This work has been financially supported by the Max--Planck Institute for Mathematics, which I would like to thank for its hospitality.

\subsection*{AI disclosure}
AI tools (Claude, Anthropic) were used to help with \LaTeX{}, Magma and PARI/GP code, to find references in the literature (such as \cite[Conjecture 1.5]{wood}), to improve the writing of some passages, to double-check and correct some of the numerical examples in \S\ref{sec:examples}, and to test Conjecture \ref{conj:length_C0} numerically.  
The rank-$1$ criterion of Corollary \ref{coro:rank1} was established by the author in the quadratic case, and conjectured in general as a consequence of Conjecture \ref{conj:length_C0}. While testing this numerically, the AI observed that the non-quadratic case easily follows from Theorems \ref{thm:congruence_bound} and \ref{thm:modularity}; the author checked this argument and included it in the proof of Corollary \ref{coro:rank1}. Apart from this observation, all results and proofs are due to the author, who takes full responsibility for the content of the paper.

\bibliographystyle{amsalpha}
\bibliography{bib}

\end{document}